\documentclass[10pt,a4paper]{article}

\usepackage{amsmath, amssymb}
\usepackage{pgfplots}
\usepackage{algorithm}     
\usepackage{algpseudocode} 
\usepackage{arydshln}

\usepackage[square,numbers]{natbib}
\usepackage{color}
\usepackage{pstool}
\usepackage{graphicx}
\usepackage{psfrag}
\usepackage{subfigure}
\usepackage{tikz}
\usepackage{pgfplots}  

\usepackage{authblk}

\usepackage{amsthm}

\DeclareMathOperator*{\minimize}{minimize}

\DeclareMathOperator*{\subject}{subject\ to}

\DeclareMathOperator*{\argmin}{argmin}

\DeclareMathOperator*{\diag}{diag}

\DeclareMathOperator*{\rank}{rank}
\DeclareMathOperator*{\blkdiag}{blkdiag}

\newcounter{lem}

\newcounter{thm}

\newtheorem{lemma}[lem]{Lemma}
\newtheorem{theorem}[thm]{Theorem}

\definecolor{red}{rgb}{1,0,0}

\usetikzlibrary{shapes,arrows}
\tikzstyle{block} = [draw, fill=white, rectangle,
    minimum height=2.8em, minimum width=4.5em, node distance = 8em]
\tikzstyle{sum} = [draw, fill=white, circle, node distance=1cm, minimum size = 0.3cm]
\tikzstyle{input} = [coordinate]
\tikzstyle{output} = [coordinate]
\tikzstyle{pinstyle} = [pin edge={to-,thin,black}]
\tikzstyle{block2} = [draw,minimum width = 6em, minimum height=2.8em,text centered]

\begin{document}

\title{A Distributed Primal-Dual Interior-Point Method for Loosely Coupled Problems Using ADMM} 

\author[*]{Mariette Annergren} 
\author[**]{Sina Khoshfetrat Pakazad} 
\author[**]{Anders Hansson}
\author[*]{Bo Wahlberg}

\affil[*]{Department of Automatic Control, KTH Royal Institute of Technology, SE-100 44 Stockholm, Sweden. Email:\{mariette.annergren, bo.wahlberg\}@ee.kth.se}                                         
\affil[**]{Division of Automatic Control, Department of Electrical Engineering, Link\"oping University, SE-581 83 Link\"oping, Sweden. Email: \{sina.kh.pa, hansson\}@isy.liu.se.} 

\maketitle

\begin{abstract}
In this paper we propose an efficient distributed algorithm for solving loosely coupled convex optimization problems. The algorithm is based on a primal-dual interior-point method in which we use the alternating direction method of multipliers (ADMM) to compute the primal-dual directions at each iteration of the method. This enables us to join the exceptional convergence properties of primal-dual interior-point methods with the remarkable parallelizability of ADMM. The resulting algorithm has superior computational properties with respect to ADMM directly applied to our problem. The amount of computations that needs to be conducted by each computing agent is far less. In particular, the updates for all variables can be expressed in closed form, irrespective of the type of optimization problem. The most expensive computational burden of the algorithm occur in the updates of the primal variables and can be precomputed in each iteration of the interior-point method. We verify and compare our method to ADMM in numerical experiments.
\end{abstract}

%

\section{Introduction}
We are interested in solving convex optimization problems of the form
\begin{subequations}\label{eq:EOP}
\begin{align}
\minimize_x  \quad &f_1(x)+ \dots + f_N(x),\\
\subject \quad & G^i(x) \preceq 0, \quad i=1, \dots, N, \\
  & A^i x = b^i, \hspace{6mm} i=1, \dots, N,
\end{align}
\end{subequations}
where $f_i \colon \mathbb R^{n}\rightarrow \mathbb R$, $G^i \colon \mathbb R^{n}\rightarrow \mathbb R^{m_i}$ and $A^i \in \mathbb R^{p_i \times n}$ with $p_i < n$ and $\rank(A^i) = p_i$ for all $i = 1, \dots, N$. We assume that the function pairs $f_i$, $G^i$ and their corresponding $A^i$ for $i = 1, \dots, N$, depend only on a small subset of the elements of the variable $x$ and we denote the ordered set of the indices of these variables by $J_i$. We also denote  the ordered set of indices of triplets $f_i$, $G^i$, $A^i$ that depend on $x_i$ by $\mathcal I_i$, i.e., $\mathcal{I}_i = \{k \ | \ i \in J_k\}$. An optimization problem is called loosely coupled if $|\mathcal I_i|\ll N$ for all $i = 1, \dots, n$. We can explicitly express the coupling structure in \eqref{eq:EOP} using the so-called consistency or consensus constraints.

Centralized algorithms for solving optimization problems of the form \eqref{eq:EOP} can be unviable. This can be due to lack of powerful enough centralized computational units, or because the problem cannot be formed as a centralized optimization problem due to its structural constraints, such as privacy requirements. A sensible approach for circumventing such issues is to use distributed optimization algorithms, which rely on collaboration of multiple computing agents to solve the problem. In such a setting, each agent is assigned a local subproblem, and at every iteration it solves its subproblem and communicates or collaborates with certain other agents. This is done repeatedly until the network of agents arrives or agrees on a solution.

Distributed optimization methods have been studied for many years, and there are different approaches for devising such algorithms, see e.g. \cite{ber:97,eck:89,boyd:11,ned:09,ned:10}. One of the most common approaches for designing distributed algorithms is to apply first order or proximal point methods directly to the problem or some reformulation of it. In this class of distributed algorithms, the ones based on subgradient or gradient methods are perhaps among the simplest, see e.g. \cite{ned:09,ned:10}. The local computations that need to be performed by each agent are usually elementary. However, these algorithms are very sensitive to the scaling of the problem. They also generally require many iterations to converge to a solution with even medium accuracy, \citep{ber:97}. In order to alleviate these issues there has been a surge of interest to devise distributed algorithms based on proximal point methods, e.g. see \cite{ber:97,eck:89,boyd:11,com:11}. For certain classes of problems, for instance when the objective function of the equivalent unconstrained reformulation of the problem has two terms and/or is strongly convex, such algorithms commonly enjoy better convergence properties, \citep{gol:12,gols:12} and are less sensitive to the scaling of the problem. However, they are generally more complicated in that the local computational burden is higher, and the communication protocols are more sophisticated, see e.g. \cite{sum:12,ohl:13}. Moreover, extra care must be taken if one wishes to apply proximal point methods to more general  classes of problems, as these algorithms might even diverge, see e.g. \cite{che:13}. There have been suggestions on how to modify these methods to allow application to more general problems. However, the resulting algorithms can become overly complicated to implement, particularly in a distributed fashion, \citep{gol:12,han:12,hon:12}.

Another approach for designing distributed optimization algorithms is to use second order methods, e.g. see \cite{Chu:11,wei:13,nec:09}. For instance, in \cite{nec:09} the authors propose a distributed optimization method based on an interior-point method. The introduced algorithm is obtained by first performing a Lagrangian decomposition of the problem and then efficiently solving the subproblems using interior-point methods. However, in the proposed algorithm, the computational cost for solving the subproblems can still be considerable. The authors in \cite{Chu:11} propose a distributed Newton method for solving coupled unconstrained quadratic problems, which is used for anomaly detection in large populations. This distributed method is only applicable to unconstrained quadratic problems. In \cite{wei:13} a distributed Newton method for solving a network utility maximization problem is proposed. The cost function for such problems  is given by a summation of several terms where each term depends on a single scalar variable. This structure allows the authors to employ a matrix splitting method which in turn enables them to distribute the computations of the inexact Newton directions. However, this method relies on the special structure in the considered problem and hence can only be used in particular cases.

The approach presented in the latter paper falls in the class of inexact interior-point methods which have been studied thoroughly over the past two decades, e.g. see \cite{bel:98,fre:99,miz:99,han:00,kor:00,dur:03,bel:04,zho:04,bon:05,bon:07,toh:08,lu:09,al:09,cur:10,cur:12}. These methods combine primal or primal-dual interior point methods with iterative algorithms for solving linear systems of equations. This is motivated by the fact that we need to solve a linear system of equations in every iteration of a primal or primal-dual interior-point method, in order to compute primal or primal-dual directions. These methods provide bounds on the required accuracy of the computed directions at each iteration in order to guarantee convergence. The papers \cite{fre:99} and \cite{miz:99} consider Linear Programs (LPs) and focus on the design of these accuracy bounds. In particular, they provide bounds on primal and dual residuals and computed directions to assure convergence of their respective proposed inexact interior-point method. LPs are also considered in \cite{kor:00} where the author proposes an inexact interior-point method with Quasi-Minimal Residual (QMR) technique and Conjugate Gradient (CG) as inexact solvers of choice. Also in \cite{al:09}, the authors consider LPs and they focus on devising efficient pre-conditioners for CG algorithms for solving the underlying linear equations more efficiently, so-called Preconditioned Conjugate Gradient (PCG) algorithms. An inexact primal-dual method for solving robust optimal control problems is proposed in \cite{han:00} with QMR as the iterative solver of choice. The papers \cite{bel:04} and \cite{zho:04} consider semidefinite programs and propose inexact primal-dual interior point methods for solving the problem. The inexact solvers in these papers were PCG for which they both propose efficient pre-conditioners to improve the convergence properties. They also propose similar accuracy bounds on the computed directions that depend solely on the so-called complementarity gap. In \cite{toh:08} a quadratic semi-definite program is considered where the author uses a pre-conditioned QMR algorithm and proposes efficient pre-conditioners for further improvement of its convergence rate. Inexact interior-point methods have also been used for solving constrained nonlinear systems of equations, which can be considered as Karush-Kuhn-Tucker (KKT) optimality conditions for general optimization problems (not necessarily convex). For instance \cite{bel:98} proposes an inexact interior-point method for solving constrained nonlinear monotone systems of equations, under the assumption that the Jacobian of the system of equations is invertible at the solution. The authors in \cite{dur:03} put forth a similar framework for solving general constrained nonlinear systems of equations and they use the PCG algorithm for solving them with respect to primal-dual directions. In \cite{bon:05} an inexact interior-point method for solving constrained nonlinear system of equations is proposed which uses the so-called Hestenes' multipliers method for solving the underlying linear systems of equations. The authors further investigate the numerical properties of the proposed method and compare with the case when they use PCG as the iterative solver of choice in \cite{bon:07}.

Notice that design of distributed algorithms for solving optimization problems was not the focus of any of the works discussed in the previous paragraph. In this paper, we focus on devising a distributed optimization algorithm based on a primal-dual interior-point method for solving loosely coupled optimization problems. These constitute a more general class of problems than those considered by \cite{Chu:11,wei:13,nec:09}. To this end, we first exploit the coupling in the problem using consistency constraints and use proximal splitting methods, particularly Alternating Direction Method of Multipliers (ADMM), to compute the primal-dual directions in a distributed manner.

ADMM is a method for finding saddle points of an augmented Lagrangian and, as such, a method of finding a solution of an optimization problem \citep{gab:76}. In our approach, we use ADMM to solve the KKT conditions of a particular optimization problem that has the primal-dual directions as solution, see Section \ref{sec:PDIPM}. The benefits of using ADMM are several. The ADMM iterations
\begin{itemize}
\item converge to a solution under mild assumptions \citep{boyd:11}.
\item enable the solution to be calculated in a highly distributed way, see Section \ref{distributedmethod}.
\item consist of subproblems that are extremely cheap to solve, see Section \ref{distributedmethod}.
\end{itemize}

ADMM was first introduced in \cite{mar:75} for solving nonlinear Dirichlet problems. It was presented as a modified version of Uzawa's algorithm \citep{arr:64}. The method was developed further in \cite{gab:76}, where some convergence properties were stated.  In \cite{gab:83}, it was shown that ADMM is equivalent to Douglas-Rachford splitting for monotone operators \citep{dou:56} and similar to Peaceman-Rachford splitting \citep{pea:55}. ADMM is related to the method of multipliers, also known as Hestenes' multipliers method, \citep{ hes:69,pow:69}, and the proximal point algorithm \citep{eck:92}. For a detailed overview of ADMM and other related methods, see \cite{boyd:11}.

Our proposed distributed optimization algorithm has superior computational properties than other distributed solvers, and we believe that the key to achieving this has been the use of ADMM for computing the primal-dual directions. We are not aware of any other iterative solvers that would present the same characteristics as listed above. We illustrate the performance of the proposed algorithm using a numerical experiment.

\subsection*{Contribution}
We present a novel distributed optimization algorithm for solving loosely coupled problems of the form \eqref{eq:EOP}. The algorithm is a primal-dual interior-point method where ADMM is used to calculate the search direction in a distributed fashion. We also present an inexact version of the algorithm, where the search directions are calculated with an adaptive degree of accuracy. We formulate the conditions under which the inexact algorithm converges to a solution along with a formal proof thereof. In addition, we review how ADMM relates to other methods of solving linear system of equations in general and for our problem formulation specifically.

Our method exhibits several important qualities. Specifically, the method
\begin{itemize}
\item inherits the convergence properties of the primal-dual interior-point \\method.
\item inherits the ability of ADMM to distribute calculations.
\item has cheap search direction calculations, where the most expensive computational burden can be precomputed.
\end{itemize}

\subsection*{Outline}

First we define the notation, and in Section \ref{sec:LCP} we explain the problem formulation and the structure of loose coupling. In Section \ref{sec:PDIPM} we briefly describe a primal-dual interior-point method. We apply this method to loosely coupled problems in Section \ref{sec:DPDIPM} and describe the details of how we can devise a distributed algorithm for solving such problems by using ADMM. In order to increase the efficiency of the proposed algorithm we discuss the use of inexact primal-dual directions in the algorithm in sections \ref{sec:PDIIPM} and \ref{sec:IDP}. We then provide a connection to iterative saddle point solvers in Section \ref{sec:saddle}. Moreover, to further improve the convergence properties of the algorithm, over-relaxation and scaling for the ADMM iterations are briefly discussed in Section~\ref{sec:ADMM}. We illustrate the performance of the proposed algorithm using some numerical experiments in Section~\ref{sec:NE}. Conclusions and future work are stated in Section \ref{sec:C}. In Appendix \ref{app:global}, we provide a proof of global convergence of the inexact algorithm. In Appendix \ref{sec:Appendix2}, we derive the explicit relations between ADMM and Uzawa's method, and ADMM and fixed point iterations for the considered problem formulation.
\subsection*{Notation}
The set of real numbers is denoted by $\mathbb R$. The set of real $n$-dimensional vectors and $n\times m$ matrices are denoted by $\mathbb R^{n}$ and $\mathbb R^{n\times m}$, respectively, and the transpose of a matrix $A$ is denoted by $A^T$. Let $\mathbb{N}_p$ represent the ordered set of positive integers $\{1,2,\ldots,p\}$. Given a set $J \subset \{1,2,\ldots,n\}$, the matrix $E_J \in \mathbb{R}^{|J|\times n}$ is the matrix obtained by deleting the rows indexed by $\mathbb{N}_n \setminus J$ from an identity matrix of order $n$, where $|J|$ denotes the number of elements in set $J$. Consequently, $E_Jx$ is a $|J|$-dimensional vector with the components of $x$ that correspond to the elements in $J$, and we denote this vector $x_J$. We denote by $x^{i,(l)}_k$ the $k$th element of vector $x^i$ at the $l$th iteration. Given vectors $x^i$ and matrices $A^i$ for $i= 1, \dots, N$, the column vector $(x^1, \dots, x^N)$ is all of the given vectors stacked and $\blkdiag(A^1, \dots, A^N)$ represents a block-diagonal matrix with $A^i$ as its diagonal blocks. Similarly, given a vector $x \in \mathbb R^{n}$, $\diag(x_1, \dots, x_n)$ denotes a diagonal matrix with its diagonals  expressed by elements of $x$. The vector $\hat{e}$ is a vector of ones of appropriate dimensions given by the context. The minimum value of a set or of a function is denoted by ``$\min$'' and the minimizing argument of an optimization problem in denoted by ``$\argmin$''. The inequality $x\preceq y$, where $x,y\in \mathbb R^{n}$, means $x_i\leq y_i$ for $i=1,\ldots,n$. The standard uniform distribution over interval $[a,b]$ is denoted $U(a,b)$. To simplify notation we introduce
\begin{align*}
z^{(l)}&=(x^{(l)},s^{(l)},\lambda^{(l)},v^{(l)}),\\
\Delta z&=(\Delta x,\Delta s,\Delta \lambda,\Delta v),\\
z^{i,(l)}&=(x_{J_i}^{(l)},s^{i,(l)},\lambda^{i,(l)},v^{i,(l)}).
\end{align*}
Here $x$ denotes all primal variables, $s$ is the slack variable vector, $\lambda$ is the dual variable vector corresponding to inequality constraints and $v$ is the dual variable vector corresponding to equality constraints. When formulating a loosely coupled problem, we introduce an additional primal variable $w$ and an additional dual variable $v_c$, both correspond to the consistency constraint in the coupled problem. Given $z^{(l)}$ and $\Delta z$, we also define $f^{(l)}(\alpha)$ as $f(z^{(l)}+\alpha\Delta z)$ which yields $f^{(l)}(0)=f(z^{(l)})$.
\section{Loosely Coupled Problems}\label{sec:LCP}
 The problem in \eqref{eq:EOP} can be equivalently written as
\begin{subequations}\label{eq:DDEOP}
\begin{align}
\minimize_{W, x} \quad & \bar f_1(w^1) + \dots + \bar f_N(w^N),\\
\subject \quad & \bar G^i(w^i) \preceq 0,  \quad  i = 1, \dots, N,\\
  & \bar A^i w^i = b^i, \hspace{6mm} i=1, \dots, N, \label{eq:DEOPc}\\
  & \bar E x = W,\label{eq:DEOPb}
\end{align}
\end{subequations}
where $W = (w^1, \dots, w^N)$ and $\bar{E}  = \begin{bmatrix} E_{J_1}^T   &\cdots &  E_{J_N}^T \end{bmatrix}^T$ with $E_{J_i}$ as a $0$--$1$ matrix that is obtained from an identity matrix of order $n$ by deleting the rows indexed by $\mathbb{N}_n \setminus J$. We refer to the constraints in \eqref{eq:DEOPb} as consistency constraints. The functions $\bar f_i \colon \mathbb R^{|J_i|} \rightarrow \mathbb R$ are lower dimensional descriptions of the functions $f_i$ such that $f_i(x) = \bar f_i(E_{J_i}x)$ for all $x \in \mathbb R^n$ and $i = 1, \dots, N$. In this formulation, the functions $\bar G^i \colon \mathbb R^{|J_i|} \rightarrow \mathbb R^{m_i}$ are defined in the same manner as the functions $\bar f_i$, and the matrices $\bar A^i \in \mathbb R^{p_i \times |J_i|}$ are defined by removing unnecessary columns from $A^i$. We further assume that $p_i < |J_i|$ and that $\rank (\bar A^i) = p_i$ for all $i = 1, \dots, N$. In this paper, we intend to devise algorithms to solve problems of the form in \eqref{eq:EOP} or \eqref{eq:DDEOP} in a distributed manner, and we will investigate the possibility of using primal-dual interior-point methods, both exact and inexact. To ensure global convergence of our algorithm, when using an inexact interior point method, we further make the standard assumption that the functions $f_i$ and $G^i$, in addition to being convex, have Lipschitz continuous derivatives, see section \ref{distConRes}. Next, we briefly review primal-dual interior-point methods for solving convex problems.

\section{Primal-Dual Interior-Point Methods}\label{sec:PDIPM}

Let us consider the convex optimization problem
\begin{equation}\label{eq:ConvexIneq}
\begin{split}
\minimize & \quad  F(x)\\
\subject & \quad  g_i(x) \leq 0, \quad i = 1, \dots, m, \\
&\quad  Ax = b,
\end{split}
\end{equation}
where $F \colon \mathbb R^n \rightarrow \mathbb R$, $g_i \colon \mathbb R^n \rightarrow \mathbb R$ and $A \in \mathbb R^{p \times n}$ with $p<n$ and $\rank (A) = p$. We introduce slack variables $s\in\mathbb R^m$ and reformulate \eqref{eq:ConvexIneq} as
\begin{equation}\label{eq:ConvexIneqSlack}
\begin{split}
\minimize & \quad  F(x)\\
\subject & \quad  g_i(x) +s_i= 0, \quad i = 1, \dots, m, \\
&\quad  Ax = b,\\
&\quad  s_i \geq 0,\quad i = 1, \dots, m.
\end{split}
\end{equation}
The problem in \eqref{eq:ConvexIneq} is equivalent to \eqref{eq:ConvexIneqSlack}. This means that $x$ is optimal for \eqref{eq:ConvexIneq} if and only if $(x,s)$ is optimal for \eqref{eq:ConvexIneqSlack} with $s_i=-g_i(x)$ for $i=1,\dots,m$, see \cite{boyd:04}. The KKT optimality conditions for Problem \eqref{eq:ConvexIneqSlack} can be written as
\begin{subequations}\label{eq:ConvexIneqKKTP1}
\begin{align}
\nabla F(x)  + \sum_{i = 1}^m \lambda_i\nabla g_i(x) + A^T v&= 0,\\ \lambda_i&\geq 0, \quad i= 1, \dots, m, \\ s_i&\geq 0, \quad i= 1, \dots, m, \\  \lambda_i s_i &= 0 , \quad i= 1, \dots, m,  \label{eq:ConvexIneqKKTP1-d}\\  g_i(x) +s_i&= 0, \quad i= 1, \dots, m,\\A x &= b.
\end{align}
\end{subequations}
The conditions are equivalent to those obtained for \eqref{eq:ConvexIneq} if $s_i$ is exchanged with $-g_i(x)$ for $i=1,\dots,m$. Primal-dual methods solve the problem in \eqref{eq:ConvexIneqSlack} by dealing with a sequence of modified versions of the optimality conditions in \eqref{eq:ConvexIneqKKTP1} where we perturb~\eqref{eq:ConvexIneqKKTP1-d}~as $\lambda_i s_i = \mu$ with $\mu>0$. Particularly, in a primal-dual framework and at each iteration, we get the primal and dual search directions by linearizing the perturbed KKT conditions and solving the resulting set of linear equations with respect to the search directions. The perturbed KKT conditions for \eqref{eq:ConvexIneqSlack} are
\begin{subequations}\label{eq:ConvexIneqKKTPD}
\begin{align}
\nabla F(x)  + \sum_{i = 1}^m \lambda_i\nabla g_i(x) + A^T v&= 0,\\ \lambda_i s_i &= \mu , \quad i= 1, \dots, m,  \label{eq:ConvexIneqKKTP-d}\\ g_i(x) +s_i&= 0, \quad i= 1, \dots, m,\\  A x &= b,
\end{align}
\end{subequations}
with $\lambda_i > 0$ and $s_i > 0$ for all $i =1, \dots, m$. Given the primal and dual iterates $x^{(l)}$, $s^{(l)}$, $\lambda^{(l)}$ and $v^{(l)}$ at iteration $l$ such that $\lambda_i^{(l)} > 0$ and $s_i^{(l)} > 0$ for all $i =1, \dots, m$, we linearize \eqref{eq:ConvexIneqKKTPD} which results in
\begin{subequations}\label{eq:QAppIneqKKTPD}
\begin{align}
\begin{split}
\!\left(\!\!\nabla^2 F(x^{(l)}) + \sum_{i = 1}^m \lambda_i^{(l)}\nabla^2 g_i(x^{(l)})\!\right)\!\Delta x+& \\\sum_{i = 1}^m \nabla g_i(x^{(l)})\Delta \lambda_i + A^T\Delta v &= -r^{(l)}_{\text{dual}},\end{split}\\
\label{eq:QAppIneqKKTPD1}
 -\lambda_i^{(l)} \Delta s_i +s_i^{(l)}\Delta \lambda_i &= -\left(r^{(l)}_{\text{cent}}\right)_i+\mu^{(l)},
  i= 1, \dots, m,\\
\label{eq:QAppIneqKKTPD2}
 \nabla g_i(x^{(l)})^T \Delta x +\Delta s_i&= -\left(r^{(l)}_{\text{primal,1}}\right)_i,
  i= 1, \dots, m, \\
 A \Delta x &= -r^{(l)}_{\text{primal,2}},
\end{align}
\end{subequations}
where
\begin{subequations}\label{eq:res}
\begin{align}
r_{\text{dual}}^{(l)} &= \nabla F(x^{(l)}) +  \sum_{i = 1}^m \lambda^{(l)}_i\nabla g_i(x^{(l)}) + A^T v^{(l)},\\
\left(r_{\text{cent}}^{(l)} \right)_i &= \lambda^{(l)}_i s_i^{(l)}, \quad i=1, \dots, m,\\
\left(r^{(l)}_{\text{primal,1}}\right)_i& = g_i(x^{(l)}) + s_i^{(l)}, \quad i=1, \dots, m,\\
r_{\text{primal,2}}^{(l)}& = Ax^{(l)} - b.
\end{align}
\end{subequations}
The linearized KKT conditions in~\eqref{eq:QAppIneqKKTPD} can be written in a compact form as
\begin{align}\label{eq:QAppIneqKKTPDCom}
H'(z^{(l)})\Delta z=- H(z^{(l)})+ \mu^{(l)}\begin{bmatrix}0 \\ 0\\ 0  \\\hat{e} \end{bmatrix},
\end{align}
where
\begin{align*}
H'(z^{(l)})=&\begin{bmatrix}\nabla^2 F(x^{(l)}) + \sum_{i = 1}^m \lambda_i^{(l)}\nabla^2 g_i(x^{(l)}) & 0&Dg(x^{(l)})^T & A^T \\  Dg(x^{(l)}) & I&0& 0 \\ A & 0 & 0&0\\0&\Lambda^{(l)}&S^{(l)}&0\end{bmatrix},\end{align*}
and
\begin{align*}
H(z^{(l)})=& \begin{bmatrix} r^{(l)}_{\text{dual}} \\ r^{(l)}_{\text{primal,1}}  \\ r^{(l)}_{\text{primal,2}}\\ r^{(l)}_{\text{cent}} \end{bmatrix},
\end{align*}
with
\begin{subequations}
\begin{align*}
Dg(x) &= \begin{bmatrix} \nabla g_1(x) & \dots & \nabla g_m(x) \end{bmatrix}^T\in\mathbb R^{m\times n},\\
\Lambda^{(l)}&=\diag(\lambda_1^{(l)}, \dots, \lambda_m^{(l)}) \in\mathbb R^{m\times m},\\
S^{(l)}&=\diag(s_1^{(l)}, \dots, s_m^{(l)})\in\mathbb R^{m\times m},\\
\hat{e}&=(1 \dots 1)^T \in\mathbb R^{m}.
\end{align*}
\end{subequations}
We assume that $H'(z^{(l)})$ is nonsingular, which is a standard assumption in an interior-point method. One way to solve \eqref{eq:QAppIneqKKTPDCom} is by first eliminating $\Delta s$ and $\Delta \lambda$ as
\begin{subequations}
\label{eq:PDsolution}
\begin{align}
\label{eq:PDS}
\Delta s &= -Dg(x^{(l)})\Delta x -   r^{(l)}_{\text{primal,1}},\\
\label{eq:PDLambda}
\Delta \lambda &= -(S^{(l)})^{-1}\left( \Lambda^{(l)}\Delta s -   r^{(l)}_{\text{cent}}+\mu^{(l)}\hat{e} \right).
\end{align}
\end{subequations}
We can then rewrite \eqref{eq:QAppIneqKKTPDCom} as
\begin{align}\label{eq:PD}
\begin{bmatrix} H^{(l)}_{\text{pd}} & A^T \\ A & 0 \end{bmatrix}\begin{bmatrix} \Delta x \\ \Delta v \end{bmatrix}  = - \begin{bmatrix}r^{(l)}\\ r^{(l)}_{\text{primal,2}}     \end{bmatrix}
\end{align}
where
\begin{align*}
H_{\text{pd}}^{(l)} = \nabla^2 F(x^{(l)}) + \sum_{i = 1}^m \lambda_i^{(l)}\nabla^2 g_i(x^{(l)})+ \sum_{i=1}^m \frac{\lambda_i^{(l)}}{s_i^{(l)}} \nabla g_i(x^{(l)})\nabla g_i(x^{(l)})^T,
\end{align*}
and
\begin{multline*}
r^{(l)} = r^{(l)}_{\text{dual}} + Dg(x^{(l)}) ^T(S^{(l)})^{-1} r^{(l)}_{\text{cent}}-Dg(x^{(l)}) ^T(S^{(l)})^{-1} \mu^{(l)}\hat{e}+\\Dg(x^{(l)}) ^T(S^{(l)})^{-1} \Lambda^{(l)}r^{(l)}_{\text{primal,1}}.
\end{multline*}
The key observation for our proposed algorithm is that the set of equations in \eqref{eq:PD} express the optimality conditions for the quadratic program

\begin{align}\label{eq:PDQP}
\begin{split}
\minimize & \quad \frac{1}{2}\Delta x^T H^{(l)}_{pd} \Delta x + (r^{(l)})^T \Delta x,\\
\subject & \quad A \Delta x = - r_{\text{primal,2}}^{(l)}.
\end{split}
\end{align}
Hence $\Delta x$ and $\Delta v$ can be computed through solving \eqref{eq:PDQP}. Based on the solution obtained, $\Delta s$ and $\Delta \lambda$ can be calculated using \eqref{eq:PDsolution}. With this, we lay out a primal-dual interior-point method in Algorithm \ref{alg:alg1a}.
\begin{algorithm}[tb]
\caption{Primal-Dual Interior-Point Method, \cite[]{boyd:04}.}\label{alg:alg1a}
\begin{algorithmic}[1]
\State{Given $l = 0$, $\sigma \in (0\ 1)$, $\epsilon>0$, $\epsilon_{\text{feas}}>0$, $\lambda_i^{(0)} > 0$, $s_i^{(0)} > 0$ for all $i = 1, \dots, m$ and $\hat \eta^{(0)} = \sum_{i=1}^m \lambda_i^{(0)}s_i^{(0)}$.}
\Repeat
\State{Set $\mu = \sigma \hat \eta^{(l)}/m$.}
\State{Given $\mu$, $x^{(l)}$, $s^{(l)}$,  $v^{(l)}$ and $\lambda^{(l)}$  compute $\Delta x^{(l+1)}$, $\Delta s^{(l+1)}$, $\Delta v^{(l+1)}$ and $\Delta \lambda^{(l+1)}$ by solving \eqref{eq:PD} and~\eqref{eq:PDsolution}.}
\State{Compute $\alpha^{(l+1)}$ using line search.}
\State {Update:
\begin{align*}
x^{(l+1)} &=  x^{(l)} + \alpha^{(l+1)}\Delta x^{(l+1)},\\
s^{(l+1)} &=  s^{(l)} + \alpha^{(l+1)}\Delta s^{(l+1)},\\
\lambda^{(l+1)}& =  \lambda^{(l)} + \alpha^{(l+1)}\Delta \lambda^{(l+1)},\\
v^{(l+1)} &=  v^{(l)} + \alpha^{(l+1)}\Delta v^{(l+1)},\\
l &= l + 1.
\end{align*}}
\State{Set $\hat \eta^{(l)} =  \sum_{i=1}^m \lambda_i^{(l)}s_i^{(l)}$.}
\Until{$\| (r^{(l)}_{\text{primal,1}}, r^{(l)}_{\text{primal,2}}) \|\leq \epsilon_{\text{feas}}$, $\| r^{(l)}_{\text{dual}} \| \leq \epsilon_{\text{feas}}$ and $\hat \eta^{(l)} \leq \epsilon$.}
\end{algorithmic}
\end{algorithm}
\subsubsection*{Remark 1}
\emph{We do not use \eqref{eq:QAppIneqKKTPDCom} for computing the primal-dual directions. This is because the coefficient matrix in \eqref{eq:QAppIneqKKTPDCom} is not symmetric, which limits our capability to solve \eqref{eq:QAppIneqKKTPDCom} efficiently. Instead we focus on the linear system of equations in \eqref{eq:PD}, which is sometimes referred to as the \emph{augmented system}. The structure in \eqref{eq:PD}, or equivalently in \eqref{eq:PDQP}, enables us to distribute the computations of primal-dual directions. Another approach to computing the primal-dual directions eliminates $\Delta x$ and $\Delta s$ and then solves a linear set of equations, referred to as the \emph{normal equations}, for computing $\Delta v$. This, however, generally destroys the inherent structure of the problem and inhibits us from devising distributed solutions.}
\newpage
%
\subsection{Step Size Computations}\label{sec:PDIPMstep}
We briefly review one of the ways to compute suitable step sizes to ensure convergence of the interior-point method. At each iteration, $l$, in order to have $s^{(l+1)}\succ 0$ and $\lambda^{(l+1)}\succ 0$, we first compute
\begin{align*}
\alpha_{\textrm{max}} = \textrm{min} \left\{ 1, \underset{i}{\textrm{min}} \left\{ -\lambda_i^{(l)}/\Delta \lambda_i^{(l+1)}  \ \big | \ \Delta \lambda_i^{(l+1)} < 0 \right\} \right\},
\end{align*}
and perform a backtracking line search as\\
\begin{algorithmic}
  \While{$\exists \ i  \colon s_i^{(l)} + \alpha^{(l+1)} \Delta s_i^{(l+1)} \leq 0 $}
    \State $\alpha^{(l+1)} = \beta \alpha^{(l+1)}$
  \EndWhile
\end{algorithmic}
with $\beta \in (0,1)$ and $\alpha^{(l+1)}$ initialized as $0.99 \alpha_{\textrm{max}}$. In order to guarantee convergence of primal and dual residuals to zero we continue the back tracking as
\begin{flushleft}
\begin{algorithmic}
 \While{$\left\| H^{(l)}(\alpha^{(l+1)}) \right\| > (1 - \gamma \alpha^{(l+1)})  \left\| H^{(l)}(0) \right\|$}
    \State  $\alpha^{(l+1)} = \beta \alpha^{(l+1)}$
  \EndWhile
\end{algorithmic}
\end{flushleft}
\noindent where $\gamma \in [0.01, 0.1]$. The resulting $\alpha^{(l+1)}$ ensures that the iterates remain feasible and that the norm of the KKT conditions, $\left\| H(z^{(l)}) \right\|$, is decreased consistently after each iteration, \cite{boyd:04}.
\subsubsection*{Remark 2}
\emph{The primal-dual method presented in this section is an implementation of a so-called infeasible long-step interior-point method. There are other variants of primal-dual methods, such as short-step, predictor-corrector and Mehrotra's predictor-corrector, that differ in their choice of primal-dual directions. One of the major differences among these variants is in the way they perturb the KKT conditions, i.e., the choice of $\mu$ in~\eqref{eq:ConvexIneqKKTPD}. This means that regardless of the choice of primal-dual interior point method the structure of the coefficient matrix in the resulting linear system of equations remains the same, and hence the discussions that follow can be extended to other variants of primal-dual methods.}
\\

\noindent Next, we apply the described primal-dual interior-point method to the loosely coupled problem in \eqref{eq:DDEOP} and propose a distributed algorithm for solving the problem.
%
\section{A Distributed Primal-Dual Interior-Point\\ Method For solving Loosely Coupled Problems}\label{sec:DPDIPM}
Let us now apply the primal-dual interior-point method in Algorithm \ref{alg:alg1a} to the problem in \eqref{eq:DDEOP}. As can be seen in Section \ref{sec:PDIPM}, the primal-dual directions computation is at the heart of a primal-dual interior-point method. Hence, the first step in devising a distributed primal-dual interior-point method for solving \eqref{eq:DDEOP} is to distribute the computations of these directions. To this end, we focus on the structure of \eqref{eq:PD} for the problem in~\eqref{eq:DDEOP}, which is given by
\vspace{5pt}
{
\renewcommand{\arraystretch}{1.8}
\begin{align}\label{sysOfEq}
\begin{bmatrix} \begin{array} {c:c:c:c}   \bar H^{(l)}_{\textrm{pd}} &0 & \bar A^T & I \\ \hdashline 0 & 0 & 0 & - \bar E^T \\\hdashline \bar A & 0 & 0 & 0 \\ \hdashline I & - \bar E  & 0 & 0 \end{array}\end{bmatrix} \begin{bmatrix} \begin{array}{c} \Delta w^1\\\vdots\\ \Delta w^N \\ \hdashline \Delta x \\ \hdashline \Delta v^1 \\ \vdots \\ \Delta v^N \\ \hdashline \Delta v_c  \end{array}  \end{bmatrix} = -\begin{bmatrix} \begin{array} {c} r^{1,(l)} \\ \vdots \\ r^{N,(l)} \\ \hdashline -\bar E^Tv_c^{(l)} \\ \hdashline r_{\text{primal,2}}^{1,(l)} \\ \vdots \\ r_{\text{primal,2}}^{N,(l)} \\ \hdashline \\[-17pt] r_c^{(l)} \end{array}  \end{bmatrix},
\end{align}
}%
\vspace{5pt}
where $\Delta v$ and $\Delta v_c$ are the dual variable directions for the constraints in \eqref{eq:DEOPc} and \eqref{eq:DEOPb}, respectively; $\bar H^{(l)}_{\textrm{pd}} = \blkdiag\left(H_{\text{pd}}^{1,(l)}, \dots, H_{\text{pd}}^{N,(l)}\right)$ with
\vspace{5pt}
\begin{align*}
H_{\text{pd}}^{i,(l)} =& \nabla^2 \bar f_i(w^{i,(l)}) +\sum_{j = 1}^{m_i}\! \lambda^{i,(l)}_j \nabla^2 \bar G_j^i(w^{i,(l)})+ \\ &\sum_{j=1}^{m_i} \frac{\lambda^{i,(l)}_j}{s^{i,(l)}}\nabla \bar G_j^i(w^{i,(l)})\left(\nabla \bar G_j^i(w^{i,(l)})\right)^T,
\end{align*}
\vspace{5pt}
$\bar A = \blkdiag\left(\bar A^1, \dots, \bar A^N\right)$ and
\vspace{5pt}
\begin{align*}
r^{i,(l)} =& \nabla \bar f_i (w^{i,(l)}) + \sum_{j=1}^{m_i} \lambda^{i,(l)}_j \nabla \bar G_j^i(w^{i,(l)})+ (\bar A^i)^T v^{i,(l)} +  v_c^{i,(l)}+\\&D\bar G^i(w^{i,(l)}) (S^{i,(l)})^{-1} r_{\text{cent}}^{i,(l)}-D\bar G^i(w^{i,(l)}) (S^{i,(l)})^{-1} \mu^{(l)}\hat{e}+\\&D\bar G^i(w^{i,(l)}) (S^{i,(l)})^{-1}\Lambda^{i,(l)} r_{\text{primal,1}}^{i,(l)}, \\
r_{\text{primal,2}}^{i,(l)} = &A^iw^{i,(l)} - b^i,\\
r_c^{(l)} =&  W^{(l)} - \bar E x^{(l)},
\end{align*}
\vspace{5pt}
with \begin{align*}r^{i,(l)}_{\text{cent}}& =  \Lambda^{i,(l)}s^{i,(l)},\\ r^{i,(l)}_{\text{primal,1}} &= \bar{G}^i(w^{i,(l)})+s^{i,(l)}.\end{align*}
\newpage
The system of equations in \eqref{sysOfEq} coincides with the necessary and sufficient optimality conditions for the optimization problem
\begin{subequations}\label{eq:QAppEqLogPD}
\begin{align}
\minimize_{\Delta W, \Delta x} & \quad \sum_{i = 1}^{N} \frac{1}{2}(\Delta w^i)^T H_{\text{pd}}^{i,(l)} \Delta w^i +  (r^{i,(l)})^T\Delta w^i - (v_c^{(l)})^T \bar E \Delta x,\\
\subject & \quad \bar A^i (\Delta w^i + w^{i,(l)}) = b^i, \quad i = 1, \dots, N,\label{eq:QAppEqLogb}\\
& \quad \Delta W - \bar E \Delta x = \bar E x^{(l)}-W^{(l)}. \label{eq:QAppEqLogc}
\end{align}
\end{subequations}

\noindent Note that \eqref{eq:QAppEqLogPD} has the same coupling structure as in~\eqref{eq:DDEOP} and can be solved in a distributed way. This enables us to compute the primal-dual directions in a distributed manner. In the following sections, we describe how to distribute the calculation of the search directions, perturbation parameter, step sizes and stopping criteria for the over-all method.
\subsection{Distributed Primal-Dual Direction Computations}\label{distributedmethod}
The problem in \eqref{eq:QAppEqLogPD} is of the form
\begin{equation}\label{eq:ADMM}
\begin{split}
\minimize_{\Delta W, \Delta x}  & \quad  F_1 (\Delta W) + F_2(\Delta x), \\
\subject & \quad  A\Delta  W + B\Delta x = c,
\end{split}
\end{equation}
which can be solved in a distributed fashion using proximal splitting methods, for example ADMM as described in Algorithm \ref{alg:alg2}, \cite{boyd:11}, \cite{com:11}.
\begin{algorithm}[tb]
\caption{ADMM, \cite{boyd:11}.}\label{alg:alg2}
Let $\bar{v}$ denote the scaled dual variable, that is, $\bar{v}=(1/\rho)v$.
\begin{algorithmic}[1]
\State{Given $k= 0$, $\rho > 0$, $\epsilon_{\textrm{pri}}>0$, $\epsilon_{\textrm{dual}}>0$, $x^{(0)}$ and $\bar{v}^{(0)}$.}
\Repeat
\State  $\Delta W^{(k+1)} = \argmin_{\Delta W} \left\{ F_1(\Delta W) + \frac{\rho}{2} \| A\Delta  W - B \Delta x^{(k)} - c + \bar{v}^{(k)} \|^2  \right\}$.
\State  $\Delta x^{(k+1)} = \argmin_{\Delta x} \left\{ F_2(\Delta x) + \frac{\rho}{2} \| A\Delta  W^{(k+1)} - B\Delta  x - c + \bar{v}^{(k)} \|^2  \right\}$.
\State  $\bar{v}^{(k+1)} =  \bar{v}^{(k)} + \left(A\Delta  W^{(k+1)} + B\Delta  x^{(k+1)} - c \right)$.
\State{$k = k+1$.}
\If  {$\| A\Delta W^{(k+1)} + B\Delta x^{(k+1)} - c  \|^2 < \epsilon_{\text{pri}}$ and $\| \rho A^TB(\Delta x^{(k+1)} - \Delta x^{(k)}) \|^2 < \epsilon_{\text{dual}}$.}
\State Terminate the algorithm.
\EndIf
\State{$k = k+1$.}
\Until{Algorithm is terminated.}
\end{algorithmic}
\end{algorithm}
In particular, \eqref{eq:QAppEqLogPD} can be written as
\begin{equation}\label{eq:ADMMp}
\begin{split}
\minimize_{\Delta W, \Delta x} &  \hspace{2mm} \underbrace{\sum_{i = 1}^{N}\left( \frac{1}{2}(\Delta w^i)^T H_{\text{pd}}^{i,(l)} \Delta w^i + (r^{i,(l)})^T\Delta w^i\right)}_{F_1(\Delta W)} + \underbrace{(-v_c^{(l)})^T \bar E \Delta x}_{F_2(\Delta x)},\\
\subject &  \\
& \hspace{-3mm}  \underbrace{\begin{bmatrix} \bar A^1 & 0 & \dots & 0\\ 0 & \bar A^2 & \dots & 0 \\ \vdots & \vdots & \ddots & \vdots \\ 0 & 0 & \dots & \bar A^N \\ I & 0 & \dots & 0\\ 0 & I & \dots & 0 \\ \vdots & \vdots & \ddots & \vdots \\ 0 & 0 & \dots & I \end{bmatrix}}_{A} \underbrace{\begin{bmatrix} \Delta w^1 \\ \Delta w^2 \\ \vdots \\ \Delta w^N\end{bmatrix}}_{\Delta W} + \underbrace{\begin{bmatrix} 0 \\0 \\ \vdots \\ 0 \\ -\bar E \end{bmatrix}}_{B} \Delta x = \underbrace{-\begin{bmatrix} r_{\text{primal,2}}^{1,(l)} \\ r_{\text{primal,2}}^{2,(l)} \\ \vdots \\ r_{\text{primal,2}}^{N,(l)} \\ r_c^{(l)} \end{bmatrix}}_{c}.
\end{split}
\end{equation}
Applying ADMM to \eqref{eq:ADMMp} results in the following update rules for the primal variable directions:
\begin{multline*}
\Delta W^{(k+1)} = \argmin_{\Delta W} \left\{ \sum_{i=1}^{N}  \left( \frac{1}{2}(\Delta w^i)^T H_{\text{pd}}^{i,(l)} \Delta w^i +  (r^{i,(l)})^T\Delta w^i+ \right. \right.\\ \left. \left.\frac{\rho}{2} \| \Delta w^{i} - \Delta x_{J_i}^{(k)} + r_c^{i,(l)} + \Delta \bar v_c^{i,(k)} \|^2+ \frac{\rho}{2} \|\bar A^i\Delta w^{i} + r_{\text{primal,2}}^{i,(l)} + \Delta \bar v^{i,(k)} \|^2 \right)  \right\},
\end{multline*}
with $r_c^{i,(l)} = w^{i,(l)} - x_{J_i}^{(l)}$, and
\begin{align*}
\Delta x^{(k+1)} =  \argmin_{\Delta x} \left\{ (-v_c^{(l)})^T \bar E \Delta x+   \frac{\rho}{2} \| \Delta W^{(k+1)} - \bar E \Delta x +r_c^{(l)} + \Delta \bar v_c^{(k)} \|^2 \right\},
\end{align*}
which results in
\begin{align}\label{eq:xPD1}
\Delta x^{(k+1)} =  \left( \bar E^T \bar E  \right)^{-1}\bar E^T \left(  v_c^{(l)} + \Delta W^{(k+1)} +r_c^{(l)} + \Delta \bar v_c^{(k)} \right).
\end{align}
Note that the update for $\Delta W$ is highly parallelizable and can be rewritten as
\begin{multline*}
\Delta w^{i,(k+1)} = \argmin_{\Delta w^i} \left\{  \frac{1}{2}(\Delta w^i)^T H_{\text{pd}}^{i,(l)} \Delta w^i + \right.\\ \left. (r^{i,(l)})^T\Delta w^i +  \frac{\rho}{2} \| \Delta w^{i} - \Delta x_{J_i}^{(k)} + r_c^{i,(l)} + \Delta \bar v_c^{i,(k)} \|^2+ \right.\\ \left. \frac{\rho}{2} \| \bar A^i\Delta w^{i} + r_{\text{primal,2}}^{i,(l)} + \Delta \bar v^{i,(k)} \|^2  \right\},
\end{multline*}
which in turn results in
\begin{multline}\label{eq:SPD1}
\Delta w^{i,(k+1)} = -\left[  H^{i,(l)}_{pd} + \rho I + \rho (\bar A^i)^T\bar A^i  \right]^{-1}\times\\  \left[ r^{i,(l)} + \rho \left( r_c^{i,(l)} + \Delta \bar v_c^{i,(k)} -   \Delta x_{J_i}^{(k)} \right)+  \rho (\bar A^i)^T \left( r_{\text{primal,2}}^{i,(l)} + \Delta \bar v^{i,(k)}  \right)  \right],
\end{multline}
for $i = 1, \dots, N$. By considering the update in \eqref{eq:xPD1} and the structure in matrix $\bar E$, we see that each Agent $i$ can update their corresponding elements of $\Delta x$ (i.e. $\Delta x_{J_i}$) in a distributed manner, through communication with its neighbors defined by $\text{Ne}(i) = \left\{ j \ | \ J_i\cap J_j \neq \emptyset  \right\}.$ The updates for the dual variable directions are given by
\begin{align}\label{dualDir}
\begin{split}
\Delta \bar v^{i,(k+1)} &= \Delta \bar v^{i,(k)} + \left( \bar A^i\Delta w^{i,(k+1)} + r_{\text{primal,2}}^{i,(l)}\right) ,\\
\Delta \bar v_c^{i,(k+1)} &= \Delta \bar v_c^{i,(k)} + \left( \Delta w^{i,(k+1)} - \Delta x_{J_i}^{(k+1)} + r_c^{i,(l)}\right),
\end{split}
\end{align}
for $i = 1, \dots, N$. The dual variable directions \eqref{dualDir} are scaled, \cite{boyd:11}, and they have to be rescaled to give the actual dual variable directions, that is
\begin{align*}
\Delta \bar v= (1/\rho)\Delta v,\quad \Delta \bar v_c= (1/\rho)\Delta v_c.
\end{align*}
Having computed the directions $\Delta W$, $\Delta x$, $\Delta v_c$ and $\Delta v$, we can now compute $\Delta s$ and $\Delta \lambda$ as
\begin{align}
\label{extraDir}
\begin{split}
\Delta s^i\! &=\! -D\bar{G}^i(w^{(l)})\Delta w^i\!- \!  r^{i,(l)}_{\text{primal,1}},\\
\Delta \lambda^i\! &=\! -(S^{i,(l)})^{-1}\!\left(\! \Lambda^{i,(l)}\Delta s^i\! - \!  r^{i,(l)}_{\text{cent}}+\mu^{(l)}\hat{e} \right)\!,
\end{split}
\end{align}
for $i = 1, \dots, N$.
The distributed algorithm for computing the primal-dual directions is expressed in Algorithm \ref{alg:alg4PD}.
\subsubsection*{Remark 3}
\emph{The computational effort for each iteration of Algorithm \ref{alg:alg4PD} is dominated by the cost of updating the iterates $\Delta w^{i,(k+1)}$, which requires factorizing the matrices $H^{i,(l)}_{\text{pd}} + \rho\left( I + (\bar A^i)^T\bar A^i\right)$ for $i = 1, \dots, N.$ In case $\rho$ is chosen to be a constant, these matrices remain the same within each iteration of the algorithm, and hence the computational burden of each instance of Algorithm \ref{alg:alg4PD} can be significantly reduced by pre-caching the factorizations and reusing them in the subsequent iterations. In fact, even if $\rho$ is nonconstant we can adopt a procedure that would allow us to update the factorizations of these matrices without having to recompute them entirely \emph{\cite[Sec. 4.2]{liu:13}}.}
\subsubsection*{Remark 4}
\emph{We can use other proximal splitting methods than ADMM for solving \eqref{eq:QAppEqLogPD} in a distributed way and possibly get better convergence properties. However, other proximal splitting methods generally require a reformulation of \eqref{eq:QAppEqLogPD}, which can in turn complicate the recovery of the dual variable directions, $\Delta v$ and $\Delta v_c$. In order to keep the presentation simple, we have restricted ourselves to using ADMM in this paper.}

\begin{algorithm}[H]
\caption{ADMM-Based Primal-Dual Direction Computation}\label{alg:alg4PD}
\begin{algorithmic}[1]
\State{Given $k= 0$, $\rho > 0$, $\hat{\eta}>0$, $\epsilon_{\textrm{pri}}>0$, $\epsilon_{\textrm{dual}}>0$, $s^{i}$, $\lambda^{i}$, $m_i$ for $i=1,\dots,N$, $\Delta W^{(0)}$, $\Delta \bar v^{(0)}$ and $\Delta \bar v_c^{(0)}$.}
\For{$i = 1, 2, \dots,N$}
\State {Communicate with all agents $r$ belonging to $\text{Ne}(i)$.}
\For {all $j \in J_i$}
 \State $\Delta x_j^{(0)} = \frac{1}{|\mathcal{I}_j|}\sum_{q \in \mathcal{I}_j}^{} \left( E_{J_q}^T \Delta w^{q,(0)} \right)_j$.
 \EndFor
 \EndFor
 \Repeat
\For{$i = 1, 2, \dots,N$}
\State {\vspace*{-10pt}
\begin{multline*}
\hspace{6mm}\Delta w^{i,(k+1)} = -\left[  H^{i,(l)}_{\text{pd}} + \rho\left( I + (\bar A^i)^T\bar A^i\right)  \right]^{-1}\times\\  \left[ r^{i,(l)} + \rho \left( r_c^{i,(l)} + \Delta \bar v_c^{i,(k)} -   \Delta x_{J_i}^{(k)} \right)+ \right.  \\ \left.  \rho (\bar A^i)^T \left( r_{\text{primal,2}}^{i,(l)} + \Delta \bar v^{i,(k)}  \right)  \right].
\end{multline*}}
\State {Communicate with all agents $r$ belonging to $\text{Ne}(i)$.}
\For {all $j \in J_i$}
 \State $\Delta x_j^{(k+1)} = \frac{1}{|\mathcal{I}_j|}\sum_{q \in \mathcal{I}_j}^{} \left[ E_{J_q}^T\left( \Delta w^{q,(k+1)}+  \Delta \bar v_c^{q,(k)} + v_c^{q,(l)} + r_c^{q,(l)}   \right) \right]_j. $
 \EndFor
 \State $\Delta \bar v^{i,(k+1)} = \Delta \bar v^{i,(k)} + \left( \bar A^i\Delta w^{i,(k+1)} + r_p^{i,(l)}\right)$.
  \State $\Delta \bar v_c^{i,(k+1)} = \Delta \bar v_c^{i,(k)} + \left( \Delta w^{i,(k+1)} - \Delta x_{J_i}^{(k+1)} + r_c^{i,(l)}\right)$.
\State  Check whether $\| \Delta x_{J_i}^{(k+1)}-\Delta x_{J_i}^{(k)}\|^2 \leq \epsilon_{\text{dual}}/N$, $\| \Delta w^{i,(k+1)}-\Delta x_{J_i}^{(k+1)}+r_c^{i,(l)}\|^2 \leq \epsilon_{\text{pri}}/(2N)$ and $\| \bar A^i \Delta w^{i,(k+1)} + r_{\text{primal,2}}^{i,(l)}\|^2 \leq \epsilon_{\text{pri}}/(2N)$.
\EndFor
\If  {Condition in Step (17) satisfied for all $i = 1, \dots, N$.}
\State {Terminate the algorithm.}
\EndIf
\State{$k = k+1$.}
\Until{Algorithm is terminated.}
\For{$i = 1, 2, \dots,N$}
\State{\vspace*{-10pt}\begin{align*}
\Delta s^i &= -D\bar{G}^i(w^{(l)})\Delta w^i-   r^{i,(l)}_{\text{primal,1}},\\
\Delta \lambda^i &= -(S^{i,(l)})^{-1}\left( \Lambda^{i,(l)}\Delta s^i -   r^{i,(l)}_{\text{cent}}+\mu^{(l)}\hat{e} \right).
\end{align*}}
\EndFor
 \end{algorithmic}
 \end{algorithm}
%

%
%
\subsection{Distributed Computations of Perturbation Parameter, Step Size and Stopping Criterion}\label{distStep}
In the distributed case, we set $\mu$ at iteration $l$ to
\begin{align}\label{eq:Distmu}
\mu^{(l)} = \sigma \frac{\underset{i}{\textrm{min}}\big(\hat{\eta}^{i,(l)}\big)}{\sum_{i=1}^N m_i},
\end{align}
where $\hat{\eta}^{i,(l)}=(s^{i,(l)})^T\lambda^{i,(l)}$ and $\sigma \in (0,1)$ is a user-defined constant. We can distribute the required minimum value computations using algorithms such as \emph{min-consensus}, \cite{Iut:12}. To calculate the step size in a distributed way, we need to provide an alternative representation of $H(z^{(l)})$. For this purpose, we express $\|r_{\textrm{dual}}(z^{(l)})\|^2=\sum_{i=1}^N\|r_{\textrm{dual}}^{i}(z^{i,(l)})\|^2$ with
\begin{align}\label{rdual}
r_{\textrm{dual}}^{i}(z^{i,(l)})=\nabla \bar{f}_i(w^{i,(l)})+\sum_{j=1}^{m_i}\nabla \bar{G}_j^i(w^{i,(l)})\lambda_j^i+(\bar{A}^i)^Tv^{i,(l)}+v_c^{i,(l)}.
\end{align}
 Note that the sum $-\sum_{i=1}^NE_{J_i}^Tv_c^{i,(l)}$ is not included in the expression for the residual. This is because it is always zero for each iteration of the algorithm provided that we initialize $v_c^{i,(0)}$ such that $-\sum_{i=1}^NE_{J_i}^Tv_c^{i,(0)}=0$ and choose the same step size for all subsystems. To see this, consult the derivation of why the second block in residuals vector \eqref{residual2} is always equal to zero in the proof of Theorem \ref{connectStop}. We express the primal residuals as before,
\begin{align}\label{rprimal}
\begin{split}
(r_{\textrm{primal,1}}^{i}(z^{i,(l)}),r_{\textrm{primal,2}}^{i}(z^{i,(l)}))=
\begin{bmatrix}\begin{array}{c}
\bar{G}^i(w^{i,(l)})-s^i\\
\hdashline\\
\bar{A}^iw^{i,(l)}-b^i\\
w^{i,(l)}-x_{J_i}^{(l)}\\
\end{array}\end{bmatrix}.
\end{split}
\end{align}
We are now ready to express $\|H(z^{(l)})\|^2=\sum_{i=1}^N\|H^{i}(z^{i,(l)})\|^2$ with
\begin{align}\label{H}
\begin{split}
H^{i}(z^{i,(l)})=
\begin{bmatrix}\begin{array}{c}
r_{\textrm{dual}}^{i}(z^{i,(l)})\\
(r_{\textrm{primal,1}}^{i}(z^{i,(l)}),r_{\textrm{primal,2}}^{i}(z^{i,(l)}))\\
S^{i,(l)}\Lambda^{i,(l)}\hat{e}\\
\end{array}\end{bmatrix},
\end{split}
\end{align}
and are thus able to distribute the evaluation of $\|H(z^{(l)})\|^2$. At this point, let each agent $i$ compute its local step size $\alpha^{i,(l+1)}$ using the approach described in Section \ref{sec:PDIPMstep}, that is the agent first sets
\begin{align*}
\alpha_{\textrm{max}}^i = \textrm{min} \left\{ 1, \underset{j}{\textrm{min}} \left\{ -\lambda_j^{i,(l)}/\Delta \lambda_j^{i,(l+1)} \ \big | \ \Delta \lambda_j^{i,(l+1)}<0  \right\} \right\},
\end{align*}
for $j=1, \dots, m_i$, and perform a backtracking line search as\\
\begin{algorithmic}
  \While{$\exists \ j  \colon s_j^{i,(l)} + \alpha^{i,(l+1)} \Delta s_j^{i,(l+1)} \leq 0 $}
    \State $\alpha^{i,(l+1)} = \beta \alpha^{i,(l+1)}$
  \EndWhile
\end{algorithmic}
with $\beta \in (0,1)$ and $\alpha^{i,(l+1)}$ initialized as $0.99 \alpha_{\textrm{max}}^i$. In order to guarantee convergence of the KKT conditions to zero we continue the back tracking as
\begin{flushleft}
\begin{algorithmic}
 \While{$\left\| H^{i,(l)}(\alpha^{i,(l+1)}) \right\|^2> (1 - \gamma \alpha^{i,(l+1)})^2  \left\| H^{i,(l)}(0)\right\|^2   $}
    \State  $\alpha^{i,(l+1)} = \beta \alpha^{i,(l+1)}$
  \EndWhile
\end{algorithmic}
\end{flushleft}
\noindent where $\gamma \in [0.01, 0.1]$. The resulting $\alpha^{i,(l+1)}$ ensures that the local iterates remain feasible with respect to local inequality constraints and that the norm of the local KKT residuals is decreased consistently after each iteration. Once all agents have computed their local step sizes, we then choose the global step size as the smallest one, that is
\begin{align*}
\alpha^{(l+1)} = \underset{i}{\textrm{min}}\left\{\alpha^{i,(l+1)}\right\}, \textrm{ for } i = 1,\dots,N.
\end{align*}
This then allows us to guarantee the aforementioned properties consistently for all agents. As for the perturbation parameter, the minimum value computations can be performed in a distributed fashion using algorithms such as \emph{min-consensus}, \cite{Iut:12}.

It is also possible to check the stopping criterion at each iteration in a distributed way. For this purpose we need to distribute the check of whether the primal and dual residual norms together with the centrality residual are small enough. Note that, due to~\eqref{rdual} and~\eqref{rprimal}, we have the following implications:
\begin{align*}
\|(r_{\textrm{primal,1}}^{i,(l)},r_{\textrm{primal,2}}^{i,(l)})\|^2\leq\frac{\epsilon_{\textrm{feas}}^2}{N} \textrm{ for all } i=1,\dots,N \Rightarrow &\|(r_{\textrm{primal,1}}^{(l)}r_{\textrm{primal,2}}^{(l)})\|\!\leq\! \epsilon_{\textrm{feas}},\\
\|r_{\textrm{dual}}^{i,(l)}\|^2\leq\frac{\epsilon_{\textrm{feas}}^2}{N} \textrm{ for all } i=1,\dots,N \Rightarrow& \|r_{\textrm{dual}}^{(l)}\|\leq \epsilon_{\textrm{feas}}
\end{align*}
and
\begin{align*}
\hat{\eta}^{i,(l)}=(s^{i,(l)})^T\lambda^{i,(l)}\leq\frac{\epsilon}{N}\textrm{ for all } i=1,\dots,N \Rightarrow \hat{\eta}^{(l)}\leq\epsilon.
\end{align*}
Hence, in case for all agents $i = 1, \dots, N$ we have
\begin{align*}
\|(r_{\textrm{primal,1}}^{i,(l)},r_{\textrm{primal,2}}^{i,(l)})\|^2&\leq\frac{\epsilon_{\textrm{feas}}^2}{N},\\
\|r_{\textrm{dual}}^{i,(l)}\|^2&\leq\frac{\epsilon_{\textrm{feas}}^2}{N},\\
\hat{\eta}^{i,(l)}=(s^{i,(l)})^T\lambda^{i,(l)}&\leq\frac{\epsilon}{N},
\end{align*}
then the termination condition of the primal-dual algorithm is satisfied and the algorithm is terminated. Notice that this can be performed in a distributed way and to this end, each agent would merely need to declare whether its local termination conditions are satisfied or not.

At this point we can distribute the computations of the primal-dual directions $\Delta W$, $\Delta x$, $\Delta v$, $\Delta v_c$, $\Delta s$ and $\Delta \lambda$ using Algorithm \ref{alg:alg4PD}. The algorithm can be used in Step 4 of Algorithm~\ref{alg:alg1a}. Combining algorithms~\ref{alg:alg1a} and~\ref{alg:alg4PD} with the modifications discussed above result in Algorithm~\ref{alg:algNPDexact}, which is a distributed primal-dual interior-point method for solving \eqref{eq:DDEOP}.

Notice that for Algorithm \ref{alg:algNPDexact} to function consistently, it is important that the primal-dual directions are computed accurately. This in turn can require many ADMM iterations, particularly for the first iterations of the primal-dual method. It is, however, expected that the number of ADMM iterations would decrease as we progress through the primal-dual iterations, thanks to warm-starting of the ADMM iterations as outlined in Step 10 of Algorithm \ref{alg:algNPDexact}. One way of improving the efficiency of our proposed algorithm is to incorporate the use of inexact directions. This can potentially reduce the number of required ADMM iterations for computing the primal-dual directions. Analysis of interior-point methods under inexact directions has been investigated, and such methods are referred to as inexact interior-point methods. Studying the convergence results for such methods suggests that during the first iterations of an interior-point method, when we are far away from the optimal solution, it is not necessary to compute the search directions accurately, and the accuracy requirements become more stringent as we progress through the interior-point method iterations, e.g., see \cite{han:00}, \cite{al:09}, \cite{bel:98}. This means that such methods utilize an adaptive stopping criterion for the search direction calculations, which yields higher and higher accuracy as the interior-point iterates get closer and closer to a solution.
Next we investigate the possibility of devising a distributed version of one such method.

\begin{algorithm}[tb]
\caption{Distributed Primal-Dual Interior-Point Method}\label{alg:algNPDexact}
\begin{algorithmic}[1]
\State
Given $l=0$, $\sigma \in (0,1)$, $v_c^{(0)}$ such that $\bar{E}^Tv_c^{(0)}=0$, $W^{0}$, $(s^{(0)},\lambda^{(0)})\succ 0$, $\epsilon_{\textrm{feas}}>0$, $\epsilon>0$, $\epsilon_{\textrm{pri}}>0$, $\epsilon_{\textrm{dual}}>0$, $m_i$ for $i=1,\dots,N$, $\Delta W^{(0)}$, $\Delta \bar v^{(0)}$ and $\Delta \bar v_c^{(0)}$.
\For{$i = 1, 2, \dots,N$}
\State {Communicate with all agents $r$ belonging to $\text{Ne}(i)$.}
\For {all $j \in J_i$}
 \State $x_j^{(0)} = \frac{1}{|\mathcal{I}_j|}\sum_{q \in \mathcal{I}_j}^{} \left( E_{J_q}^T  w^{q,(0)} \right)_j.$
 \EndFor
 \EndFor
\Repeat
\State{Compute $\mu^{(l)}$ as in \eqref{eq:Distmu}.}
\State{Given $\mu^{(l)}$ and $z^{(l)}$ compute $\Delta z^{(l+1)}$ using Alg. \ref{alg:alg4PD} with the initial iterates $\Delta z^{(l)}$. }
\For{$i = 1, 2, \dots,N$}
\State{Compute local step size $\alpha^{i,(l+1)}$ using the approach presented in Section \ref{distStep}.}
\EndFor
\State{Let $\alpha^{l+1}=\underset{i}{\textrm{min}}\{\alpha^{i,(l+1)}\}$.}
\For{$i = 1, 2, \dots,N$}
\State{Set $z^{i,(l+1)}=z^{i,(l)}+\alpha^{(l+1)}\Delta_z^{i,(l+1)}$.}
\EndFor
\State{Set $l=l+1$.}
\Until{$\|(r_{\textrm{primal,1}}^{i}(z^{i,(l)}),r_{\textrm{primal,2}}^{i}(z^{i,(l)}))\|^2\leq \epsilon_{\textrm{feas}}^2/N$, $\|(r_{\textrm{dual}}^{i}(z^{i,(l)})))\|^2\leq \epsilon_{\textrm{feas}}^2/N$ and $\hat{\eta}^{i,(l)}\leq \epsilon/N$ for all $i=1,\dots,N$.}
\end{algorithmic}
\end{algorithm}
\vspace{32pt}
\section{Primal-Dual Inexact Interior-Point Methods}\label{sec:PDIIPM}
\begin{algorithm}[tb]
\caption{Primal-dual Inexact Interior-point Method, \cite[]{bel:98}.}\label{alg:alg1b}
\begin{algorithmic}[1]
\State
Given $(s^{(0)},\lambda^{(0)})\succ 0$, $\bar{\tau}_1=\textrm{min}(\Lambda^{(0)}S^{(0)}\hat{e})/((\lambda^{(0)})^Ts^{(0)}/m),$ $\bar{\tau}_2=(\lambda^{(0)})^Ts^{(0)}/\|R(z^{(0)})\|,$
$\gamma^{(l-1)}\in[1/2,1)$, $\eta_{\textrm{max}}\in(0,1)$, $\beta\in(0,1)$, $\theta\in(0,1)$ and $\epsilon>0$.
\Repeat
		\State{
Choose $\sigma^{(l)}$, $\hat{\eta}^{(l)}$ and $\gamma^{(l)}\in[1/2,\gamma^{(l-1)}]$ such that $(\sigma^{(l)}+\hat{\eta}^{(l)})\in(0,\eta_{\textrm{max}})$ and
\begin{align*}
\sigma^{(l)}>\textrm{max}\left( \frac{\sqrt{m}+\bar{\tau}_1\gamma^{(l)}}{\sqrt{m}(1-\bar{\tau}_1\gamma^{(l)})},\frac{\sqrt{m}+\bar{\tau}_2\gamma^{(l)}}{m} \right)\hat{\eta}^{(l)}.
\end{align*}
Put $\mu^{(l)}=\sigma^{(l)}(s^{(l)})^T\lambda^{(l)}/m$ and $\bar{\eta}^{(l)}=\sigma^{(l)}+\hat{\eta}^{(l)}$.
		\State
Compute $\Delta z^{(l)}$ by solving \eqref{eq:QAppIneqKKTPDComInexact} with $\|\hat{r}^{(l)}\|\leq \hat{\eta}^{(l)}(s^{(l)})^T\lambda^{(l)}/m.$
		\State
Choose $\bar{\alpha}_1^{(l)}$ such that
$$\textrm{min}\left(S^{(l)}(\bar{\alpha}_1^{(l)})\Lambda^{(l)}(\bar{\alpha}_2^{(l)})\hat{e}\right)\geq\bar{\tau}_1\gamma^{(l)}\left(s^{(l)}(\bar{\alpha}_1^{(l)})\right)^T\lambda^{(l)}(\bar{\alpha}_1^{(l)})/m.$$}
\State{Choose $\bar{\alpha}_2^{(l)}$ such that
$$\left(s^{(l)}(\bar{\alpha}_2^{(l)})\right)^T\lambda^{(l)}(\bar{\alpha}_2^{(l)})\geq\bar{\tau}_2\gamma^{(l)}\|R(z^{(l)}(\bar{\alpha}_2^{(l)}))\|.$$
		\State
Set $\alpha^{(l)}=\textrm{min}(\bar{\alpha}_1^{(l)},\bar{\alpha}_2^{(l)}).$
		\State
Set $\eta^{(l)}=1-\alpha^{(l)}(1-\bar{\eta}^{(l)})$.
			\While{
$\| H^{(l)}(\alpha^{(l)}) \|> (1-\beta(1-\eta^{(l)}))\| H^{(l)}(0) \|$}
				\State{
$\alpha^{(l)}=\theta \alpha^{(l)}$ and $\eta^{(l)}=1-\theta(1-\eta^{(l)})$.}
			\EndWhile
		\State
Set $z^{(l+1)}=z^{(l)}+\alpha^{(l)}\Delta z^{(l)}$ and $l=l+1$.
	}
\Until{$\|H(z^{(l)})\|\leq \epsilon$.}
\end{algorithmic}
\end{algorithm}
In an inexact interior-point method we only need to solve \eqref{eq:QAppIneqKKTPDCom} approximately for the primal-dual directions. That is, we solve
\vspace{-5pt}
\begin{align}\label{eq:QAppIneqKKTPDComInexact}
H'(z^{(l)})\Delta z =- H(z^{(l)})+\mu^{(l)}\begin{bmatrix}0 \\ 0\\ 0  \\\hat{e} \end{bmatrix}+\hat{r}^{(l)},
\end{align}
where $\hat{r}^{(l)}$ is the residual. However, in order to assure the convergence of the algorithm, it is necessary to make modifications to the framework. Specifically, consider the framework laid out in Algorithm \ref{alg:alg1b} as introduced in \cite{bel:98}. In order to ensure a global convergent inexact interior-point method, \cite{bel:98}, \cite{bakry:96}, it is necessary to
\begin{itemize}
\item impose restrictions on the problem formulation;
\item add requirements on the residual $\hat{r}$ in \eqref{eq:QAppIneqKKTPDComInexact};
\item change the choice of $\hat{\eta}$, $\sigma$ and $\alpha$.
\end{itemize}

For example, the total residual $\tilde{r}^{(l)}$ of the KKT system must fulfill
\begin{align}
\|\tilde{r}^{(l)}\|=\left\|\mu^{(l)}\begin{bmatrix}0 \\ 0\\ 0  \\\hat{e} \end{bmatrix}+\hat{r}^{(l)}\right\|\leq(\sigma^{(l)}+\hat{\eta}^{(l)})\|H(z^{(l)})\|,
\end{align}
see \cite{bel:98}.

We define the set $\Omega(\epsilon)$ for a given $\epsilon>0$ as
\begin{align}\label{omega:set}
\begin{split}
\Omega(\epsilon)=\bigg\{&z\in \mathbb{R}^{n+2m+p}|\epsilon\leq\|H(z)\|\leq\|H(z^{(0)})\|,\\ &\textrm{min}\left(S\Lambda\hat{e}\right)\geq \frac{\bar{\tau}_1}{m}\frac{1}{2}s^T\lambda,\ s^T\lambda\geq\bar{\tau}_2\frac{1}{2}\|R(z)\|\bigg\},
\end{split}
\end{align}
with $\bar{\tau}_1=\textrm{min}(\Lambda^{(0)}S^{(0)}\hat{e})/((\lambda^{(0)})^Ts^{(0)}/m)$ and $\bar{\tau}_2=(\lambda^{(0)})^Ts^{(0)}/\|R(z^{(0)})\|,$
and the following assumptions
\begin{enumerate}
\item [A1] $H$ is continuously differentiable in $\Omega(0)$.
\item [A2] $\{z^{(l)}\}$ is bounded.
\item [A3] $H'(z)$ is nonsingular in $\Omega({\epsilon})$ with $\epsilon>0$.
\item [A4] $R'$ is Lipschitz continuous in $\Omega(0)$ with constant $L$, where $$R=(r_{\textrm{dual}},r_{\textrm{primal,1}},r_{\textrm{primal,2}}).$$
\end{enumerate}
Then if $\{z^{(l)}\}$ is generated by Algorithm \ref{alg:alg1b} and assumptions A1--A4 are fulfilled, the sequence $\{\|H(z^{(l)})\|\}$ will converge to zero and $z^{(l)}$ will converge to the limit point of $\{z^{(l)}\}$, see Theorem 3.3 in \cite{bel:98} (with the additional assumption that $\sigma^{(l)}$ is bounded away from zero). Similar to Algorithm~\ref{alg:algNPDexact}, it is also possible to distribute the computations in Algorithm \ref{alg:alg1b}, and that is discussed in the next section.

\section{A Distributed Primal-Dual Inexact Interior-Point Method for Solving Loosely Coupled Problems}\label{sec:IDP}
Let us apply the primal-dual inexact interior-point method described in Section \ref{sec:PDIIPM} to the problem in \eqref{eq:DDEOP}. Similar to Algorithm \ref{alg:algNPDexact}, we can use Algorithm \ref{alg:alg4PD} for computing the inexact directions in a distributed fashion. Particularly, this algorithm can be used in Step~4 of Algorithm~\ref{alg:alg1b}. However, in order for the computed directions to satisfy the required accuracy in Step~4 of Algorithm \ref{alg:alg1b}, we need to establish a connection between the ADMM stopping criteria and the norm of the residuals in \eqref{eq:QAppIneqKKTPDComInexact}. This connection is established in the following theorem.
\begin{theorem}\label{connectStop}
It is possible to choose the thresholds $\epsilon_{\text{pri}}, \epsilon_{\text{dual}}>0$ such that the stopping criteria in Algorithm \ref{alg:alg4PD} and the residual conditions in Step 4 of Algorithm \ref{alg:alg1b} are equivalent.
\end{theorem}
\begin{proof}
Note that in our approach for solving \eqref{eq:QAppIneqKKTPDCom}, we in fact solve \eqref{eq:QAppIneqKKTPD1} and~\eqref{eq:QAppIneqKKTPD2} exactly, since $\Delta s$ and $\Delta \lambda$ are eliminated, see \eqref{eq:PDsolution}. The residuals in~\eqref{eq:QAppIneqKKTPDComInexact} for our approach and for the problem in \eqref{eq:DDEOP} are therefore given as
\begin{align}\label{residual2}\hat{r}^{(l)}\!\!=\!\!
\begin{bmatrix}\begin{array}{c}H^{1,(l)}_{pd}\Delta w^1+(\bar{A}^1)^T\Delta v^1+\Delta v_c^1+r^{1(l)} \\ \vdots\\ H^{N,(l)}_{pd}\Delta w^N+(\bar{A}^N)^T\Delta v^N+\Delta v_c^N +r^{N,(l)} \\\hdashline
\sum_{i=1}^N(-E_{J_i}^T\Delta v_c-E_{J_i}^Tv_c^{(l)})\\
\hdashline \\ 0\\\vdots\\0\\\hdashline \bar{A}^1\Delta w^1+r_{\textrm{primal,2}}^{1,(l)}\\
 \vdots\\
 \bar{A}^N\Delta w^N+r_{\textrm{primal,2}}^{N,(l)}\\
\\ \Delta w^1-\Delta x_{J_1}+r_c^{1,(l)}\\ \vdots\\
\Delta w^N-\Delta x_{J_N}+r_c^{N,(l)}\\\hdashline \\ 0\\\vdots\\0\end{array}\end{bmatrix}\!\!.
\end{align}
The norm of the fourth block of the right hand side of \eqref{residual2} is already included in the stopping criteria in Algorithm \ref{alg:alg4PD}. Furthermore, in the ADMM iterations, $\Delta w^{(k+1)}$ and $\Delta v^{(k+1)}$ are computed such that
\begin{align*}
0=&\left( H^{i,(l)}_{pd} + \rho I + \rho (\bar A^i)^T\bar A^i\right)\Delta w^{i,(k+1)}+ \rho \left( r_c^{i,(l)} + \frac{1}{\rho}\Delta v_c^{i,(k)} - \Delta x_{J_i}^{(k)} \right)+\\&r^{i,(l)}+    \rho (\bar A^i)^T \left( r_{\text{primal,2}}^{i,(l)} + \frac{1}{\rho}\Delta v^{i,(k)}  \right)\\
=&\left( H^{i,(l)}_{pd}\! + \rho I + \rho (\bar A^i)^T\bar A^i\right)\Delta w^{i,(k+1)}\!+\!\rho \left( -\Delta w^{i,(k+1)}\!+\Delta x_{J_i}^{(k+1)}\!-\Delta x_{J_i}^{(k)} \right)
\!+\\& r^{i,(l)} +
\rho (\bar A^i)^T \left( \frac{1}{\rho}\Delta v^{i,(k+1)} -\bar{A}^i\Delta w^{i,(k+1)} \right),
\end{align*}
which gives
\begin{align*}
H^{i,(l)}_{pd}\Delta w^{i,(k+1)}+r^{i,(l)}+(\bar A^i)^T \Delta v^{i,(k+1)}+\Delta v_c^{i,(k+1)} =
\rho(\Delta x_{J_i}^{(k)}-\Delta x_{J_i}^{(k+1)}).
\end{align*}
Consequently, the norm of the first block of the right hand side of \eqref{residual2} would be small if and only if $\|\rho(\Delta x_{J_i}^{(k)}-\Delta x_{J_i}^{(k+1)})\|^2$  would be small, which is also included in the stopping criteria of Algorithm \ref{alg:alg4PD}. The only remaining part of the residuals vector in \eqref{residual2} is the second block which is always equal to zero. This is because, in the ADMM iterations, $\Delta x^{(k+1)}$ and $\Delta v_c^{(k+1)}$ are chosen such that
\begin{align*}
0\!=\!-\bar{E}^T\!v_c^{(l)}\!-\!\rho\bar{E}^T\!\!\left(\!\!\Delta w^{(k+1)} \!-\!\bar{E}\Delta x^{(k+1)}\!+r_c^{(l)}\!+\frac{1}{\rho}\Delta v_c^{(k)}\!\!\right)\!=\!-\bar{E}^T\!v_c^{(l)}\!-\!\bar{E}^T\!\Delta v_c^{(k+1)}.
\end{align*}
As a result, we have
\begin{multline}
\label{residualhat}
\| \hat r^{(l)}  \|^2 = \sum_{i= 1}^N \left(\|\rho(\Delta x_{J_i}^{(k)}-\Delta x_{J_i}^{(k+1)})\|^2 +\right. \\ \left. \| \Delta w^{i,(k+1)}-\Delta x_{J_i}^{(k+1)}+r_c^{i,(l)}\|^2  +\| \bar A^i \Delta w^{i,(k+1)} + r_{\text{primal,2}}^{i,(l)}\|^2\right)
\end{multline}
and hence if the thresholds $\epsilon_{\text{pri}}, \epsilon_{\text{dual}}>0$ are chosen appropriately, then the stopping criteria in Algorithm \ref{alg:alg4PD} and the residual conditions in Step 4 of Algorithm \ref{alg:alg1b} will be equivalent.
\end{proof}
Up to this point, we have illustrated that it is possible to use Algorithm \ref{alg:alg4PD} to distribute the computations of the inexact directions. Next we show how to choose the thresholds in the stopping criteria of Algorithm \ref{alg:alg4PD} so that the computed directions satisfy the necessary accuracy requirements. Moreover,  we describe how to distribute the update of the remaining iteration-dependent parameters and the remaining steps of Algorithm~\ref{alg:alg1b}.

\subsection{Distributed Computations of Perturbation Parameter, Step Size and Stop Criterion}\label{sec:diststep}
%

Let us first define $R^{i}(z^{i,(l)})$, $\sigma^{i,(l)}$ and $\hat{\eta}^{i,(l)}$ for each agent $i=1,\dots,N$. Particularly, we define $R^{i}(z^{i,(l)})$ such that $\|R(z^{(l)})\|^2=\sum_{i=1}^N\|R^{i}(z^{i,(l)})\|^2$ with
\begin{align}\label{R}
R^{i}(z^{i,(l)})=(r_{\textrm{dual}}^i(z^{i,(l)}),r_{\textrm{primal,1}}^i(z^{i,(l)}),r_{\textrm{primal,2}}^i(z^{i,(l)})),
\end{align}
and we choose $\sigma^{i,(l)}$ such that
\begin{align*}
\sigma^{i,(l)}>\bar{\tau}_2^i\gamma^{i,(l)}\hat{\eta}^{i,(l)}(s^{i,(l)})^T\lambda^{i,(l)}/\underset{i}{\textrm{min}}\left((s^{i,(l)})^T\lambda^{i,(l)}\right) + \epsilon_{\sigma},
\end{align*}
with $\bar{\tau}_2^i=(\lambda^{i,(0)})^Ts^{i,(0)}/\|R^i(z^{i,(0)})\|$, $\epsilon_{\sigma}\in (0,1)$ and $\hat{\eta}^{i,(l)}$ such that $(\sigma^{i,(l)}+\hat{\eta}^{i,(l)})\in (0,\eta_{\textrm{max}})$, $\gamma^{i,(l)}\in[1/2,\gamma^{i,(l-1)}]$ and $\gamma^{i,(0)}\in[1/2,1)$. Then choose $\hat{\eta}^{(l)} = \min_i\left\{ \hat{\eta}^{i,(l)} \right\}$, $\sigma^{(l)} = \max_i\left\{ \sigma^{i,(l)} \right\}$ and set $\bar{\eta}^{l}=\sigma^{(l)}+\hat{\eta}^{(l)}$. At this point, we can describe how to compute proper thresholds for the ADMM iterations termination criteria. In the distributed setting we set the residual norm condition in Step 4 in Algorithm~\ref{alg:alg1b} for the problem in~\eqref{eq:DDEOP} to
\begin{align}\label{stopIIP}
\|\hat{r}^{(l)}\|\leq \hat{\eta}^{(l)}\sum_{i=1}^N(s^{i,(l)})^T\lambda^{i,(l)}/m,
\end{align}
where $m = \sum_{i=1}^N m_i$ and the residual $\hat{r}^{(l)}$ is defined in \eqref{residual2}. There are several choices of $\epsilon_{\text{pri}}$ and $\epsilon_{\text{dual}}$ that ensure that criterion \eqref{stopIIP} is fulfilled when the stopping criteria of ADMM are satisfied. In this paper we set
\begin{align}\label{epsADMMInexact1}
\epsilon_{\text{pri}}^{i,(l)}=\frac{N}{2}\left(\hat{\eta}^{(l)}(s^{i,(l)})^T\lambda^{i,(l)}/m\right)^2
\end{align}
and
\begin{align}\label{epsADMMInexact2}
\epsilon_{\text{dual}}^{i,(l)}=\frac{N}{2}\left(\hat{\eta}^{(l)}(s^{i,(l)})^T\lambda^{i,(l)}/(\rho m)\right)^2.
\end{align}
Note that $\epsilon_{\text{pri}}$ and $\epsilon_{\text{dual}}$ are now subproblem-specific and they change with each iteration $l$ of the primal-dual inexact interior-point method.

Next we focus on computation of the perturbation parameter. In order to compute $\mu$ at iteration $l$, each agent first needs to compute
\begin{align*}
\mu^{i,(l)} = \sigma^{(l)}(s^{i,(l)})^T\lambda^{i,(l)}/m.
\end{align*}
Then the perturbation parameter is chosen as $\mu^{(l)} = \underset{i}{\textrm{min}} \left\{ \mu^{i,(l)} \right\}$.

It now remains to compute a proper step size for updating the iterates. As it is laid out in Algorithm \ref{alg:alg1b}, the process of computing the step size $\alpha^{(l)}$ at each iteration consists of two stages, namely, computation of an upper-bound on the step size and the line search. Similar to the approach we undertook for computing a step size in Section \ref{distStep}, each agent first needs to compute their local step size $\alpha^{i,(l)}$. To this end, each agent $i$ initially sets
\begin{align*}
\alpha^{i,(l)}=\textrm{min}\left\{\bar{\alpha}_1^{i,(l)},\bar{\alpha}_2^{i,(l)}\right\},
\end{align*}
where
\begin{align*}
\bar \alpha_j^{i,(l)}=\underset{\alpha \in [0,1]}{\textrm{max}}\left\{\alpha|f_j(\alpha')\geq 0, \textrm{ for all }\alpha'\leq\alpha\right\},
\end{align*}
with $j \in \{1, 2 \}$, and
\begin{align*}f_1(\alpha)=\textrm{min}\left(S^{i,(l)}(\alpha)\Lambda^{i,(l)}(\alpha)\hat{e}\right)-\bar{\tau}^i_1\gamma^{i,(l)}\left(s^{i,(l)}(\alpha)\right)^T\lambda^{i,(l)}(\alpha)/m_i,\end{align*}
and
$$f_2(\alpha)=\left(s^{i,(l)}(\alpha)\right)^T\lambda^{i,(l)}(\alpha)-\bar{\tau}^i_2\gamma^{i,(l)}\|R^i(z^{i,(l)}(\alpha))\|.$$
Then the agent sets $\eta^{i,(l)}=1-\alpha^{i,(l)}(1-\bar{\eta}^{(l)})$ and performs a line search as
\begin{flushleft}
\begin{algorithmic}
 \While{$\left\| H^{i,(l)}(\alpha^{i,(l)}) \right\|> (1 - \beta ( 1 -  \eta^{i,(l+1)}))  \left\| H^{i,(l)}(0)\right\|   $}
    \State  $\alpha^{i,(l)} = \theta \alpha^{i,(l)}$
    \State  $\eta^{i,(l)}=1-\theta(1-\eta^{i,(l)})$
  \EndWhile
\end{algorithmic}
\end{flushleft}
When all agents are done computing their local step sizes, we then set $\alpha^{(l)}=\underset{i}{\textrm{min}}\{\alpha^{i,(l)}\}$. Finally, similar to the approach in Section \ref{distStep}, we check that $\| H^i(z^{i,(l)}) \|^2\leq \epsilon^2/N$ for $i=1,\dots,N$, which implies that $\| H(z^{(l)}) \|\leq \epsilon$, to decide whether to terminate the primal-dual iterations or not.

\subsection{Distributed Primal-Dual Inexact Interior-Point\\ Method}

Using Algorithm \ref{alg:alg4PD} with $\epsilon_{\text{pri}}^{i,(l)}$ and $\epsilon_{\text{dual}}^{i,(l)}$ set in accordance to \eqref{epsADMMInexact1} and \eqref{epsADMMInexact2}, respectively, we can distribute the computations of the primal-dual directions $\Delta W$, $\Delta x$, $\Delta v$, $\Delta v_c$, $\Delta s$ and $\Delta \lambda$. This algorithm can then be used in Step 4 of Algorithm~\ref{alg:alg1b} which distributes the major computations in the primal-dual inexact interior-point method. Combining algorithms \ref{alg:alg4PD} and~\ref{alg:alg1b} with the modifications discussed above result in Algorithm~\ref{alg:algNPD}, which is a distributed primal-dual inexact interior-point method for solving \eqref{eq:DDEOP}.

\begin{algorithm}[tb]
\caption{Distributed Primal-dual Inexact Interior-point Method}\label{alg:algNPD}
\begin{algorithmic}[1]
\State
Given $l=0$, $\rho>0$, $v_c^{(0)}$ such that $\bar{E}^Tv_c^{(0)}=0$, $W^{(0)}$, $(s^{(0)},\lambda^{(0)})\succ 0$,
 $\bar{\tau}_1^i=\textrm{min}(\Lambda^{i,(0)}S^{i,(0)}\hat{e})/((\lambda^{i,(0)})^Ts^{i,(0)}/m_i)$, $\bar{\tau}_2^i=(\lambda^{i,(0)})^Ts^{i,(0)}/\|R^i(z^{i,(0)})\|,$ $\gamma^{i,(0)}\in[1/2,1)$, $\eta_{\textrm{max}}^i\in(0,1)$, for $i=1,\dots,N$, $\epsilon_{\textrm{feas}}>0$, $\epsilon>0$, $\beta\in(0,1)$, $\theta\in(0,1)$, $\Delta W^{(0)}$,
$\Delta \bar v^{(0)}$ and $\Delta \bar v_c^{(0)}$.
\For{$i = 1, 2, \dots,N$}
\State {Communicate with all agents $r$ belonging to $\text{Ne}(i)$.}
\For {all $j \in J_i$}
 \State $x_j^{(0)} = \frac{1}{|\mathcal{I}_j|}\sum_{q \in \mathcal{I}_j}^{} \left( E_{J_q}^T  w^{q,(0)} \right)_j$.
 \EndFor
 \EndFor
\Repeat
\State Compute $\Delta z^{(l+1)}$ using Alg. \ref{alg:alg4PD} with the initial iterates $\Delta z^{(l)}$ with the stopping criteria thresholds given in \eqref{epsADMMInexact1} and \eqref{epsADMMInexact2}.
%
\For{$i=1,\dots,N$}
\State{Compute local step size $\alpha^{i,(l+1)}$ using the approach presented in Section \ref{distStep}.}
\EndFor
\State{Set $\alpha^{(l+1)}=\underset{i}{\textrm{min}}\{\alpha^{i,(l+1)}\}$.}
\State{Set $z^{(l+1)}=z^{(l)}+\alpha^{(l+1)}\Delta z^{(l+1)}$.}
\State{Set $l=l+1$.}
%
\Until{$\| H^i(z^{i,(l)}) \|^2\leq \epsilon^2/N$ for all $i=1,\dots,N$.}
\end{algorithmic}
\end{algorithm}
%

\subsection{Distributed Convergence Result}\label{distConRes}
To ensure a global convergent inexact interior-point method, we modify the conditions stated in Section \ref{sec:PDIIPM}. In particular, we define the set $\Omega^i(\epsilon)$ for a given $\epsilon\geq0$ as
\begin{align}\label{omega:setinexact}
\begin{split}
\Omega^i(\epsilon)=\bigg\{&z^i\in \mathbb{R}^{|J_i|+2m_i+p_i}\big|\frac{\epsilon}{N}\leq\|H^i(z^i)\|\leq\|H^i(z^{i,(0)})\|,\\
&\textrm{min}\left(S^{i}\Lambda^{i}\hat{e}\right)\geq \frac{\bar{\tau}_1^i}{m_i}\frac{1}{2}(s^{i})^T\lambda^{i},\
(s^{i})^T\lambda^{i}\geq\bar{\tau}_2^i\frac{1}{2}\|R(z^{i})\|\bigg\},
\end{split}
\end{align}
with $$\bar{\tau}_1^i=\textrm{min}(\Lambda^{i,(0)}S^{i,(0)}\hat{e})/((\lambda^{i,(0)})^Ts^{i,(0)}/m_i),\ \bar{\tau}_2^i=(\lambda^{i,(0)})^Ts^{i,(0)}/\|R^i(z^{i,(0)})\|,$$ and the following assumptions
\begin{enumerate}
\item [B1] $H^i$ is continuously differentiable in $\Omega^i(0)$.
\item [B2] $\{z^{i,(l)}\}$ is bounded.
\item [B3] $H'$ is nonsingular in $\underset{i=1,\dots,N}{\prod}\Omega^i(\epsilon)$ with $\epsilon>0$.
\item [B4] $(R^i)'$ is Lipschitz continuous in $\Omega^i(0)$ with constant $L^i$, where $$R^i=(r_{\textrm{dual}}^i,r_{\textrm{primal,1}}^i,r_{\textrm{primal,2}}^i).$$
\end{enumerate}

If $\{z^{(l)}\}$ is generated by Algorithm \ref{alg:algNPD} and assumptions B1--B4 are fulfilled, then $\{\|H(z^{(l)})\|\}$ converges to zero and $z^{(l)}$ will converge to the limit point of $\{z^{(l)}\}$. For a proof of this see Appendix \ref{app:global}. Note that $\{z^{i,(l)}\}$ generated by Algorithm \ref{alg:algNPD} lies in $\Omega^i(0)$ for $i=1,\dots,N$ and all $l$.  Next we establish the connection of our proposed approach to that of iterative solvers and put forth suggestions on how to improve the convergence properties of the algorithm.
\section{Iterative Solvers for Saddle Point Systems}\label{sec:saddle}
The optimality conditions \eqref{sysOfEq} is a saddle point system where the solution, since strong duality holds, is a saddle point of the Lagrangian function of optimization problem \eqref{eq:ADMM}, see \cite{boyd:04}. In addition, a saddle point of the Lagrangian function is a saddle point of the augmented Lagrangian, and vice versa, \citep{gab:76}. Consequently, to find a solution of \eqref{sysOfEq}, we can instead consider the saddle point system corresponding to the augmented Lagrangian. The benefit of using the augmented Lagrangian is improved convergence properties when using dual methods for solving the saddle point system, \citep{arr:64,hes:69,for:83}.

\subsection{Uzawa's Method and Fixed Point Iterations}\label{sec:UMAFPI}
A well known algorithm for solving saddle point systems such as \eqref{sysOfEq} is Uzawa's method \citep{arr:64}. We solve the system of equations \eqref{sysOfEq} using ADMM (Algorithm \ref{alg:alg2}), which was originally derived as a modified version of Uzawa's method, see \citep{mar:75,gab:76}. ADMM applied to \eqref{eq:ADMM} is equivalent to Uzawa's method applied to the problem corresponding to the augmented Lagrangian of \eqref{eq:ADMM}, \citep{arr:64,hes:69,pow:69}, with one Gauss-Seidel iteration, \citep{saa:03}, in the update of the primal variables, \citep{mar:75,gab:76}. In addition, ADMM can be viewed as fixed point iterations of a pre-conditioned version of \eqref{sysOfEq}, see e.g. \cite{boy:14}. The similarities between ADMM, Uzawa's method and fixed point iterations are explored explicitly for our problem in Appendix \eqref{admom}.

Uzawa's method is also equivalent to the method of multipliers when applied to the augmented Lagrangian and the relaxation parameter in Uzawa's method is set to be equal to the penalty parameter in the method of multipliers, see \cite{ben:05}.
%
\subsection{Other Iterative Methods}
For a rigorous overview of the iterative methods available, we refer to \cite{saa:03,ben:05}. The problem in \eqref{sysOfEq} has an indefinite system matrix with an upper left block matrix that is singular, which limits the number of applicable methods or at least requires some pre-conditioning beforehand. For example, we could use the CG method applied to the normal equations of \eqref{sysOfEq}, see \cite{saa:03}. However, to the best of our knowledge, one has then destroyed the inherent structure of the problem which prevent us from distributing the calculations.

\section{Improving Convergence Rate of ADMM}\label{sec:ADMM}
We can improve the convergence rate of ADMM by using over-relaxation, warm starting the ADMM iterations, choosing the penalty parameter $\rho$ carefully and scaling the problem formulation appropriately.

In over-relaxation, we replace the primal quantity $A\Delta W^{k+1}$ with
\[\alpha_{\text{OR}} A\Delta W^{k+1}-(1-\alpha_{\text{OR}})(B \Delta x^{k}-c),\  \alpha_{\text{OR}} \in (1,2),\]
 in the update of $\Delta x^k$ and scaled dual variables $\Delta \bar{v}$ and $\Delta \bar{v}_c$. Empirical studies have shown that an $\alpha_{\text{OR}} \in [1.5,1.8]$ may improve the rate of convergence, see \citep{boyd:11}.

We can also improve the convergence rate by warm starting ADMM. That is, using the solution of the previous ADMM iteration as initial condition in the current iteration, see \citep{boyd:11}. The improvement in convergence rate is due to that the primal-dual directions of the interior point method do not change much as the iterates are approaching the solution.

In general it is an open problem how to choose the penalty parameter $\rho$ optimally. Certain heuristics suggest that $\rho$ should be chosen such that the primal and dual residuals converge at the same rate, \citep{boyd:11}. In \cite{2013arXiv1303.6680T}, the authors rescale the optimization problem using a block-diagonal matrix. For the scaled problem, they derive the $\rho$ and $\alpha_{\text{OR}}$ which guarantee the lowest worst-case amount of iterations in ADMM. The scaling matrix can be constructed in a distributed way and since it is block-diagonal it maintains the structure of our problem. However, we would have to recalculate the scaling matrix for each iteration of the interior-point method.

\section{Numerical Experiment}\label{sec:NE}
To illustrate the proposed method, we apply it to a randomly generated optimization problem of the form  \eqref{eq:DDEOP}. 

In most of the examples, we warm start the ADMM algorithm. That is, we use the previous step direction as an initial point. We have tuned the penalty parameter $\rho$ of ADMM slightly, to provide a better balance of the convergence rate between the primal residual and the dual residual of the ADMM formulation. We do not use the scaling suggested by  \cite{2013arXiv1303.6680T}, although it is believed that this will improve the convergence rate of the search direction calculations.

\subsection{Simulation Set-Up 1}\label{setup1}
We consider fifty subproblems ($N=50$). The total number of variables ($n$), equality constraints ($p$) and inequality constraints ($m$) are 5091, 5089, and 1524, respectively. The number of local variables, local equality constraints and local inequality constraints  are drawn from the standard uniform distributions $U(55,65)$, $U(7,13)$ and $U(27,33)$, respectively. The indices defining the consistency constraints ($J_i$) are drawn from $U(0,900)$ and there are 3017 such constraints.

To ensure that the problem formulation is feasible, we first draw a global variable $x$ and slack variable $S$ from $U(-10,10)$ and $U(1,10)$, respectively. We then generate equality and inequality constraints from these values. The inequality constraints are affine, $A_{in}^ix+b_{in}^i\preceq0$ for $i=1,\dots,N$, where the elements of the matrices $A_{in}^i$ for $i=1,\dots,N$ are drawn from $U(0,1)$. The vectors $b_{in}^i$ for $i=1,\dots,N$ are then calculated. The equality constraints are constructed in the same way; the matrices $\bar{A}^i$ for $i=1,\dots,N$ are drawn from $U(0,1)$. The vectors $b^i$ for $i=1,\dots,N$ are then calculated.

The objective function is quadratic in the global variable with $f_i(x)=x^TP^ix+(q^i)^Tx+e^i$ for $i=1,\dots,N$. The elements of $P^i$, $q^i$ and $e^i$ are drawn from $U(0,1)$, $U(0,1)$ and $U(0,10)$, respectively.

The optimization problem is solved using Algorithm \ref{alg:algNPDexact}, Algorithm \ref{alg:algNPD} and ADMM (Algorithm \ref{alg:alg2}). For comparison, all three algorithms are terminated using the stop criteria of  Algorithm \ref{alg:algNPDexact}. The settings of each specific method are displayed in tables \ref{alg1_param}, \ref{alg2_param} and \ref{alg3_param}. 

We choose a relative value of the upper bounds on the stop criteria, $\epsilon$ and $\epsilon_{\textrm{feas}}$. That is, we choose them as

\begin{align*}
\epsilon=\epsilon_{\textrm{feas}}=10^{-6}\times\textrm{max}\{&1,\|\blkdiag(P^1, \dots,P^N)\|,\|\blkdiag(A_{in}^1, \dots, A_{in}^N)\|,\\
&\|\blkdiag(A^1, \dots,A^N)\|,\|(b_{in}^1,\dots,b_{in}^N)\|,\\
&\|(b^1,\dots,b^N)\|,\|(q^1,\dots,q^N)\| \}.
\end{align*}
For our specific problem generation we get $\epsilon=\epsilon_{\textrm{feas}}=0.0051$, that is we have a scaling factor equal to 5089.

We initialize the methods at the same point. The initial values of the global primal variable are drawn from $U(-10,10)$, the initial values of each local primal variable are given by the consistency constraints. The dual variables are all set to 10, except the dual variables that correspond to the consistency constraint, they are all set to zero. The search directions are initialized to zero if they are not warm started. For the approach using ADMM, we use indicator functions to represent all constraints except consensus. Consequently, we only have dual variables for the consistency constraints. When using ADMM on Problem \eqref{eq:DDEOP}, we get an inequality constrained optimization problem in the first primal variable update which we use \texttt{cvx}, a package for specifying and solving convex programs \citep{cvx}, to solve.

\begin{table}[h!]
\caption{Settings of Algorithm \ref{alg:algNPDexact}.}
\begin{center}
{\begin{tabular}{llllllll}
\hline
Parameter&$\mu$   &$\alpha$ &$\beta$ &$\epsilon_{\text{pri}}$&$\epsilon_{\text{dual}}$&$\rho$ &$\alpha_{\text{OR}}$ \\
 Value & 15 & 0.01 &0.5 &$50/2\times10^{-20}$&$50/2\times10^{-20}$ & 0.5 &1\\
\hline
\end{tabular}}
\end{center}
\label{alg1_param}
\end{table}

\begin{table}[h!]
\caption{Settings of Algorithm \ref{alg:algNPD}.}
\begin{center}
{\begin{tabular}{llllllll}
\hline
Parameter&$\eta_{\textrm{max}}^i$   &$\gamma^{i,(0)}$ &$\beta$ & $\theta$&$\rho$ &$\alpha_{\text{OR}}$ & $\epsilon_{\sigma}$ \\
 Value & 0.9 & 0.9 &0.1 &0.95& 0.5 &1 & 0.1\\
\hline
\end{tabular}}
\end{center}
\label{alg2_param}
\end{table}

\begin{table}[h!]
\caption{Settings of ADMM.}
\begin{center}
{\begin{tabular}{lll}
\hline
Parameter&$\rho$ &$\alpha_{\text{OR}}$ \\
 Value& 0.5 &1\\
\hline
\end{tabular}}
\end{center}
\label{alg3_param}
\end{table}

\subsubsection{Algorithm \ref{alg:algNPD} with and without Warm Starting}
We first compare Algorithm \ref{alg:algNPD} with and without warm starting. The total number of ADMM iterations is 17453 with warm starting and 21181 without warm starting. In Figure \ref{fig:iterADMMEx1}, the number of ADMM iterations in each instance of the primal-dual method is displayed. In Figure \ref{fig:pertADMMEx1}, the value of the perturbation parameter $\mu$ in each instance of the primal-dual method is shown. We benefit from warm staring ADMM with respect to the total number of ADMM iterations necessary for the primal-dual method to converge. However, for a specific iteration of the primal-dual method the warm started approach can require more iterations than the approach without warm starting, and the number of saved ADMM iterations fluctuate over the primal-dual method iterations. A smaller saving, or even no saving at all, occur when the new system of equations to solve is changed considerably from the previous one. Note, also that the number of iterations of the primal-dual method is affected. The number of iterations is 80 with warm starting and 81 without warm starting.

\begin{figure}[h!]
\begin{center}
 \centering
\pgfplotsset{width=numberOfIterationsADMMEx1.tikz\columnwidth,height=numberOfIterationsADMMEx1.tikz\columnwidth,compat=newest,plot coordinates/math parser=false}
%
%
\begin{tikzpicture}

\begin{axis}[%
scale only axis,
width=8cm,
height=5cm,
xmin=0, xmax=81,
ymin=0, ymax=550,
xlabel={iteration in primal-dual method},
ylabel={number of iterations in ADMM},
axis on top]
\addplot [
color=black,
solid
]
coordinates{
 (1,350)(2,84)(3,71)(4,102)(5,169)(6,260)(7,267)(8,213)(9,176)(10,152)(11,144)(12,138)(13,138)(14,145)(15,145)(16,145)(17,145)(18,144)(19,144)(20,144)(21,143)(22,143)(23,143)(24,143)(25,143)(26,142)(27,142)(28,142)(29,142)(30,143)(31,143)(32,143)(33,143)(34,144)(35,145)(36,145)(37,146)(38,147)(39,148)(40,150)(41,151)(42,153)(43,155)(44,157)(45,159)(46,163)(47,168)(48,172)(49,178)(50,183)(51,188)(52,193)(53,199)(54,204)(55,210)(56,215)(57,222)(58,232)(59,242)(60,253)(61,265)(62,263)(63,307)(64,363)(65,365)(66,366)(67,367)(68,368)(69,369)(70,369)(71,373)(72,377)(73,382)(74,388)(75,397)(76,407)(77,420)(78,435)(79,453)(80,466) 
};\label{solid}

\addplot [
color=black,
dashed
]
coordinates{
 (1,350)(2,297)(3,242)(4,151)(5,120)(6,239)(7,302)(8,279)(9,246)(10,215)(11,193)(12,180)(13,173)(14,172)(15,172)(16,171)(17,171)(18,170)(19,170)(20,169)(21,169)(22,169)(23,168)(24,168)(25,168)(26,168)(27,168)(28,168)(29,168)(30,168)(31,169)(32,169)(33,170)(34,170)(35,171)(36,172)(37,174)(38,175)(39,177)(40,178)(41,180)(42,183)(43,185)(44,188)(45,192)(46,197)(47,202)(48,208)(49,213)(50,219)(51,225)(52,232)(53,238)(54,244)(55,252)(56,260)(57,272)(58,285)(59,300)(60,317)(61,347)(62,404)(63,406)(64,408)(65,410)(66,412)(67,414)(68,416)(69,418)(70,420)(71,422)(72,423)(73,425)(74,426)(75,429)(76,435)(77,443)(78,455)(79,470)(80,492)(81,525) 
};\label{dashed}

\end{axis}
\end{tikzpicture}
\caption{ADMM iterations. The number of iterations of ADMM with warm starting \eqref{solid} and without warm starting \eqref{dashed} in each instance of the primal-dual method are shown.} \label{fig:iterADMMEx1}
\end{center}
\end{figure}
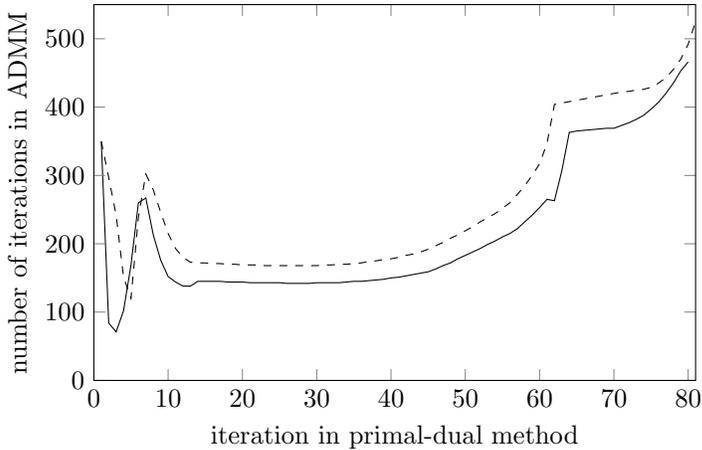
\begin{figure}[h!]
\begin{center}
 \centering
\pgfplotsset{width=perturbationADMMDistEx1.tikz\columnwidth,height=perturbationADMMDistEx1.tikz\columnwidth,compat=newest,plot coordinates/math parser=false}
%
%
\begin{tikzpicture}

\begin{axis}[%
scale only axis,
width=8cm,
height=5cm,
xmin=0, xmax=81,
ymin=-8, ymax=1,
xlabel={iteration in primal-dual method},
ylabel={$\textrm{log}(\mu)$},
axis on top]
\addplot [
color=black,
solid
]
coordinates{
 (1,0.0915945)(2,-0.0615819)(3,-0.181373)(4,-0.317885)(5,-0.479094)(6,-0.673555)(7,-0.909713)(8,-1.18404)(9,-1.38405)(10,-1.58736)(11,-1.76305)(12,-1.91597)(13,-2.02914)(14,-2.05283)(15,-2.06619)(16,-2.07954)(17,-2.09318)(18,-2.10723)(19,-2.12158)(20,-2.13622)(21,-2.15132)(22,-2.16674)(23,-2.18263)(24,-2.19885)(25,-2.21557)(26,-2.23282)(27,-2.25043)(28,-2.26859)(29,-2.28731)(30,-2.30662)(31,-2.32634)(32,-2.34688)(33,-2.36785)(34,-2.3897)(35,-2.41224)(36,-2.43549)(37,-2.45972)(38,-2.48471)(39,-2.51077)(40,-2.53764)(41,-2.56595)(42,-2.59516)(43,-2.62593)(44,-2.65801)(45,-2.69146)(46,-2.72671)(47,-2.76424)(48,-2.80337)(49,-2.84549)(50,-2.8899)(51,-2.9377)(52,-2.9892)(53,-3.04468)(54,-3.10577)(55,-3.17308)(56,-3.24812)(57,-3.33279)(58,-3.43172)(59,-3.55034)(60,-3.70019)(61,-3.91182)(62,-4.34093)(63,-4.99106)(64,-5.05208)(65,-5.11444)(66,-5.1817)(67,-5.25428)(68,-5.33264)(69,-5.41729)(70,-5.50981)(71,-5.60877)(72,-5.71588)(73,-5.83325)(74,-5.96201)(75,-6.10179)(76,-6.25926)(77,-6.43935)(78,-6.65423)(79,-6.93513)(80,-7.39298) 
};

\addplot [
color=black,
dashed
]
coordinates{
 (1,0.0915945)(2,-0.0615819)(3,-0.182489)(4,-0.319006)(5,-0.47865)(6,-0.672926)(7,-0.908962)(8,-1.17981)(9,-1.38196)(10,-1.58294)(11,-1.77459)(12,-1.92739)(13,-2.04186)(14,-2.06302)(15,-2.07638)(16,-2.09001)(17,-2.10393)(18,-2.11813)(19,-2.13262)(20,-2.14757)(21,-2.16298)(22,-2.17871)(23,-2.19476)(24,-2.21132)(25,-2.22839)(26,-2.24582)(27,-2.26379)(28,-2.28232)(29,-2.30144)(30,-2.32095)(31,-2.34128)(32,-2.36203)(33,-2.38366)(34,-2.40597)(35,-2.42898)(36,-2.45271)(37,-2.47745)(38,-2.50297)(39,-2.52957)(40,-2.5573)(41,-2.5862)(42,-2.61634)(43,-2.64776)(44,-2.68086)(45,-2.71537)(46,-2.75212)(47,-2.79045)(48,-2.83126)(49,-2.87473)(50,-2.92154)(51,-2.97142)(52,-3.02515)(53,-3.08368)(54,-3.14814)(55,-3.21998)(56,-3.301)(57,-3.39455)(58,-3.50406)(59,-3.64039)(60,-3.82233)(61,-4.12507)(62,-4.76866)(63,-4.81849)(64,-4.86676)(65,-4.91876)(66,-4.9748)(67,-5.0352)(68,-5.10034)(69,-5.17217)(70,-5.24971)(71,-5.3344)(72,-5.42697)(73,-5.52824)(74,-5.64041)(75,-5.7634)(76,-5.90001)(77,-6.05385)(78,-6.22756)(79,-6.42935)(80,-6.68138)(81,-7.04102) 
};

\end{axis}
\end{tikzpicture}
\caption{Perturbation of KKT conditions.  The value of the perturbation parameter $\mu$ with warm starting ADMM \eqref{solid} and without warm starting \eqref{dashed} ADMM in each instance of the primal-dual method are shown.} \label{fig:pertADMMEx1}
\end{center}
\end{figure}
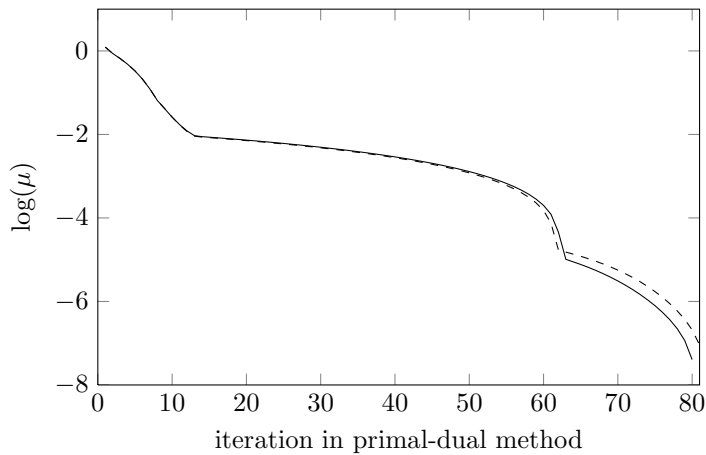

\subsubsection{Savings in ADMM Iterations Using Algorithm \ref{alg:algNPD} Compared to Algorithm \ref{alg:algNPDexact}}
We compare Algorithm \ref{alg:algNPDexact} and Algorithm \ref{alg:algNPD}, using warm starting in both methods. The total number of ADMM iterations is 38796 with Algorithm \ref{alg:algNPDexact} and 17453 with Algorithm \ref{alg:algNPD}. The number of iterations of the primal-dual method is 21 using Algorithm \ref{alg:algNPDexact} and 80 using Algorithm \ref{alg:algNPD}. We save 55\% in ADMM iterations, using Algorithm \ref{alg:algNPD}. This is due to the adaptive stop criteria of the search direction calculations. In particular, it dramatically decreases the amount of ADMM iterations necessary in the first couple of iterations of the primal-dual method. The convergence rate of Algorithm \ref{alg:algNPD} is highly dependent on the setting of the parameters in the method. With a different setting, we could get a much worse convergence rate.

\subsubsection{Stop Criteria}
We compare how the residuals and surrogate duality gap evolves between the three methods. We see in figures \ref{fig:primalResADMMEx1}, \ref{fig:dualResADMMEx1} and \ref{fig:surADMMEx1}, that the averaged, over all subproblems, value of the norm of local primal residuals, dual residuals and surrogate duality gaps are constantly decreasing as each method iterates. The same behavior is obtained for the global residuals and surrogate duality gap, see figures \ref{fig:globprimalResADMMEx1}, \ref{fig:dualResADMMEx1} and \ref{fig:surADMMEx1}. In addition, we can see that the stop criteria are fulfilled for all subproblems at approximately the same iteration. When comparing the local stop criteria with the global, we see that the local criteria can be conservative. In fact, if the global stop criteria were to be used instead of the local one, ADMM would terminate at an earlier iteration than with the local criteria, see Table \ref{alg_residuals} and figures \ref{fig:primalResADMMEx1}--\ref{fig:globsurADMMEx1}.

\begin{table}[h!]
\caption{Fulfillment of stop criteria.}
\begin{center}
{\begin{tabular}{llll}
\hline
&iteration $l$ at which&iteration $l$ at which&iteration $l$ at which\\
& local (global) primal& local (global) dual&local (global) gap\\
&constraint is ful-&constraint is ful-&constraint is ful-\\
&filled&filled&filled\\
\hline
 Alg. \ref{alg:algNPDexact} & 19 (18) &19 (20)& 21 (21)\\
 Alg. \ref{alg:algNPD} & 72 (68) & 75 (75)& 80 (80)\\
 ADMM & 304 (275)& 187 (169)& 0 (0)\\
\hline
\end{tabular}}
\end{center}
\label{alg_residuals}
\end{table}

\begin{figure}[h!]
\begin{center}
\begin{subfigure}
 \centering
\pgfplotsset{width=primalresidualADMMOrigEx1.tikz\columnwidth,height=primalresidualADMMOrigEx1.tikz\columnwidth,compat=newest,plot coordinates/math parser=false}
%
%
\begin{tikzpicture}

\begin{axis}[%
scale only axis,
width=3cm,
height=4cm,
xmin=1, xmax=21,
ymin=-15, ymax=5,
title={(a)},
xlabel={iteration in primal-dual method},
ylabel={$\frac{1}{N}\sum_{i=1}^N\textrm{log}\left(\|(r_{\textrm{primal},1}^i,r_{\textrm{primal},2}^i)\|\right)$},
axis on top]
\addplot [
color=black,
solid
]
plot [error bars/.cd, y dir = both, y explicit]
coordinates{
 (1,3.80306)+-(0.0,0.399865)(2,2.70842)+-(0.0,0.399865)(3,1.43506)+-(0.0,0.399865)(4,0.71455)+-(0.0,0.399865)(5,0.236691)+-(0.0,0.399865)(6,-0.489446)+-(0.0,0.399865)(7,-0.836469)+-(0.0,0.399865)(8,-1.20732)+-(0.0,0.399865)(9,-1.59077)+-(0.0,0.399865)(10,-1.76959)+-(0.0,0.399865)(11,-2.19017)+-(0.0,0.399865)(12,-2.64655)+-(0.0,0.399865)(13,-3.12916)+-(0.0,0.399865)(14,-3.5942)+-(0.0,0.399865)(15,-4.06897)+-(0.0,0.399865)(16,-4.54784)+-(0.0,0.399865)(17,-6.19896)+-(0.0,0.399865)(18,-6.76412)+-(0.0,0.399865)(19,-8.9467)+-(0.0,0.399865)(20,-11.4113)+-(0.0,0.399865)(21,-14.599)+-(0.0,0.399872) 
};

\addplot [
color=black,
dashed
]
coordinates{
 (1,-6.28571)(2,-6.28571)(3,-6.28571)(4,-6.28571)(5,-6.28571)(6,-6.28571)(7,-6.28571)(8,-6.28571)(9,-6.28571)(10,-6.28571)(11,-6.28571)(12,-6.28571)(13,-6.28571)(14,-6.28571)(15,-6.28571)(16,-6.28571)(17,-6.28571)(18,-6.28571)(19,-6.28571)(20,-6.28571)(21,-6.28571) 
};

\end{axis}
\end{tikzpicture}
 \end{subfigure}
\begin{subfigure}
 \centering
\pgfplotsset{width=primalresidualADMMDistEx1.tikz\columnwidth,height=primalresidualADMMDistEx1.tikz\columnwidth,compat=newest,plot coordinates/math parser=false}
%
%
\begin{tikzpicture}

\begin{axis}[%
scale only axis,
width=3cm,
height=4cm,
xmin=1, xmax=80,
ymin=-15, ymax=5,
title={(b)},
xlabel={iteration in primal-dual method},
ylabel={$\frac{1}{N}\sum_{i=1}^N\textrm{log}\left(\|(r_{\textrm{primal},1}^i,r_{\textrm{primal},2}^i)\|\right)$},
axis on top]
\addplot [
color=black,
solid
]
plot [error bars/.cd, y dir = both, y explicit]
coordinates{
 (1,4.25326)+-(0.0,0.399864)(2,4.04306)+-(0.0,0.399865)(3,3.79617)+-(0.0,0.399866)(4,3.49437)+-(0.0,0.399867)(5,3.11789)+-(0.0,0.399868)(6,2.64884)+-(0.0,0.399869)(7,2.09618)+-(0.0,0.399868)(8,1.71493)+-(0.0,0.399867)(9,1.32883)+-(0.0,0.399865)(10,1.0043)+-(0.0,0.399863)(11,0.706107)+-(0.0,0.399861)(12,0.473098)+-(0.0,0.399859)(13,0.425155)+-(0.0,0.399858)(14,0.398172)+-(0.0,0.399858)(15,0.371188)+-(0.0,0.399858)(16,0.343648)+-(0.0,0.399858)(17,0.315252)+-(0.0,0.399858)(18,0.286269)+-(0.0,0.399857)(19,0.256687)+-(0.0,0.399857)(20,0.226183)+-(0.0,0.399857)(21,0.19505)+-(0.0,0.399857)(22,0.162945)+-(0.0,0.399857)(23,0.130176)+-(0.0,0.399856)(24,0.096384)+-(0.0,0.399856)(25,0.0615367)+-(0.0,0.399856)(26,0.0259672)+-(0.0,0.399856)(27,-0.0107143)+-(0.0,0.399855)(28,-0.0485436)+-(0.0,0.399855)(29,-0.0875571)+-(0.0,0.399855)(30,-0.127381)+-(0.0,0.399855)(31,-0.168879)+-(0.0,0.399854)(32,-0.211239)+-(0.0,0.399854)(33,-0.255383)+-(0.0,0.399854)(34,-0.300913)+-(0.0,0.399853)(35,-0.347876)+-(0.0,0.399853)(36,-0.396821)+-(0.0,0.399853)(37,-0.447308)+-(0.0,0.399852)(38,-0.499929)+-(0.0,0.399852)(39,-0.554213)+-(0.0,0.399852)(40,-0.611389)+-(0.0,0.399851)(41,-0.670376)+-(0.0,0.39985)(42,-0.732513)+-(0.0,0.39985)(43,-0.797298)+-(0.0,0.399849)(44,-0.864845)+-(0.0,0.399848)(45,-0.93602)+-(0.0,0.399847)(46,-1.01182)+-(0.0,0.399846)(47,-1.09087)+-(0.0,0.399845)(48,-1.17598)+-(0.0,0.399844)(49,-1.2657)+-(0.0,0.399843)(50,-1.36235)+-(0.0,0.399841)(51,-1.46647)+-(0.0,0.39984)(52,-1.5787)+-(0.0,0.399838)(53,-1.70234)+-(0.0,0.399835)(54,-1.83865)+-(0.0,0.399833)(55,-1.99071)+-(0.0,0.399829)(56,-2.16242)+-(0.0,0.399825)(57,-2.36329)+-(0.0,0.39982)(58,-2.60452)+-(0.0,0.399813)(59,-2.90996)+-(0.0,0.399805)(60,-3.34334)+-(0.0,0.399791)(61,-4.23774)+-(0.0,0.399752)(62,-5.63726)+-(0.0,0.399675)(63,-5.7609)+-(0.0,0.399679)(64,-5.88725)+-(0.0,0.399675)(65,-6.02356)+-(0.0,0.399671)(66,-6.1707)+-(0.0,0.399667)(67,-6.3296)+-(0.0,0.399662)(68,-6.5013)+-(0.0,0.399656)(69,-6.68905)+-(0.0,0.39965)(70,-6.88992)+-(0.0,0.399644)(71,-7.10742)+-(0.0,0.399637)(72,-7.34587)+-(0.0,0.399629)(73,-7.60759)+-(0.0,0.399622)(74,-7.89188)+-(0.0,0.399615)(75,-8.21247)+-(0.0,0.399608)(76,-8.57961)+-(0.0,0.399601)(77,-9.01869)+-(0.0,0.399596)(78,-9.59532)+-(0.0,0.399591)(79,-10.5487)+-(0.0,0.39959)(80,-12.164)+-(0.0,0.399621) 
};

\addplot [
color=black,
dashed
]
coordinates{
 (1,-6.28571)(2,-6.28571)(3,-6.28571)(4,-6.28571)(5,-6.28571)(6,-6.28571)(7,-6.28571)(8,-6.28571)(9,-6.28571)(10,-6.28571)(11,-6.28571)(12,-6.28571)(13,-6.28571)(14,-6.28571)(15,-6.28571)(16,-6.28571)(17,-6.28571)(18,-6.28571)(19,-6.28571)(20,-6.28571)(21,-6.28571)(22,-6.28571)(23,-6.28571)(24,-6.28571)(25,-6.28571)(26,-6.28571)(27,-6.28571)(28,-6.28571)(29,-6.28571)(30,-6.28571)(31,-6.28571)(32,-6.28571)(33,-6.28571)(34,-6.28571)(35,-6.28571)(36,-6.28571)(37,-6.28571)(38,-6.28571)(39,-6.28571)(40,-6.28571)(41,-6.28571)(42,-6.28571)(43,-6.28571)(44,-6.28571)(45,-6.28571)(46,-6.28571)(47,-6.28571)(48,-6.28571)(49,-6.28571)(50,-6.28571)(51,-6.28571)(52,-6.28571)(53,-6.28571)(54,-6.28571)(55,-6.28571)(56,-6.28571)(57,-6.28571)(58,-6.28571)(59,-6.28571)(60,-6.28571)(61,-6.28571)(62,-6.28571)(63,-6.28571)(64,-6.28571)(65,-6.28571)(66,-6.28571)(67,-6.28571)(68,-6.28571)(69,-6.28571)(70,-6.28571)(71,-6.28571)(72,-6.28571)(73,-6.28571)(74,-6.28571)(75,-6.28571)(76,-6.28571)(77,-6.28571)(78,-6.28571)(79,-6.28571)(80,-6.28571) 
};

\end{axis}
\end{tikzpicture}
 \end{subfigure}
\begin{subfigure}
 \centering
\pgfplotsset{width=primalresidualADMMFirstEx1.tikz\columnwidth,height=primalresidualADMMFirstEx1.tikz\columnwidth,compat=newest,plot coordinates/math parser=false}\input{primalresidualADMMFirstEx1.tikz}
 \end{subfigure}
\caption{Norm of primal local residual.  The averaged, over all subproblems, value of the norm of the primal local residual along with its standard deviation \eqref{solid}  and upper bound, $\textrm{log}\left(\epsilon_{\textrm{feas}}^2/N\right)$, \eqref{dashed} are shown. In figures (a), (b) and (c), algorithms \ref{alg:algNPDexact}, \ref{alg:algNPD} and ADMM have been used, respectively.} \label{fig:primalResADMMEx1}\end{center}
\end{figure}
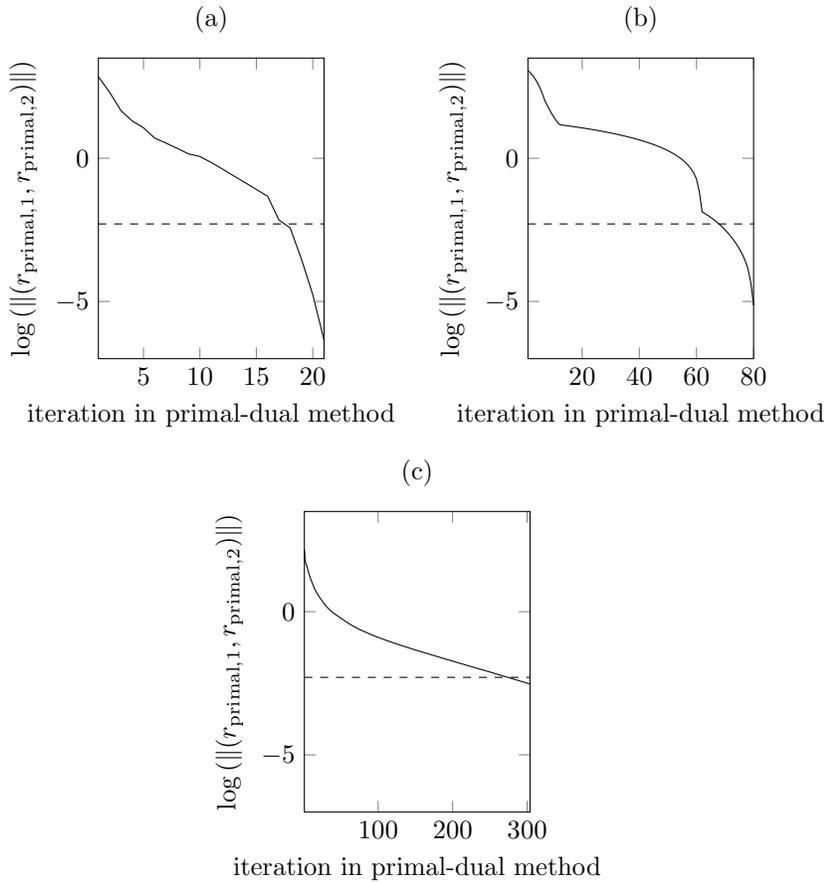
\begin{figure}[h!]
\begin{center}
\begin{subfigure}
 \centering
\pgfplotsset{width=globprimalresidualADMMOrigEx1.tikz\columnwidth,height=globprimalresidualADMMOrigEx1.tikz\columnwidth,compat=newest,plot coordinates/math parser=false}
%
%
\begin{tikzpicture}

\begin{axis}[%
scale only axis,
width=3cm,
height=4cm,
xmin=1, xmax=21,
ymin=-7, ymax=3.5,
title={(a)},
xlabel={iteration in primal-dual method},
ylabel={$\textrm{log}\left(\|(r_{\textrm{primal},1},r_{\textrm{primal},2})\|\right)$},
axis on top]
\addplot [
color=black,
solid
]
coordinates{
 (1,2.84864)(2,2.30132)(3,1.66464)(4,1.30438)(5,1.06545)(6,0.702386)(7,0.528875)(8,0.343447)(9,0.151724)(10,0.0623122)(11,-0.147977)(12,-0.376164)(13,-0.617469)(14,-0.849989)(15,-1.08738)(16,-1.32681)(17,-2.15237)(18,-2.43495)(19,-3.52624)(20,-4.75852)(21,-6.35239) 
};

\addplot [
color=black,
dashed
]
coordinates{
 (1,-2.29337)(2,-2.29337)(3,-2.29337)(4,-2.29337)(5,-2.29337)(6,-2.29337)(7,-2.29337)(8,-2.29337)(9,-2.29337)(10,-2.29337)(11,-2.29337)(12,-2.29337)(13,-2.29337)(14,-2.29337)(15,-2.29337)(16,-2.29337)(17,-2.29337)(18,-2.29337)(19,-2.29337)(20,-2.29337)(21,-2.29337) 
};

\end{axis}
\end{tikzpicture}
 \end{subfigure}
\begin{subfigure}
 \centering
\pgfplotsset{width=globprimalresidualADMMDistEx1.tikz\columnwidth,height=globprimalresidualADMMDistEx1.tikz\columnwidth,compat=newest,plot coordinates/math parser=false}
%
%
\begin{tikzpicture}

\begin{axis}[%
scale only axis,
width=3cm,
height=4cm,
xmin=1, xmax=80,
ymin=-7, ymax=3.5,
title={(b)},
xlabel={iteration in primal-dual method},
ylabel={$\textrm{log}\left(\|(r_{\textrm{primal},1},r_{\textrm{primal},2})\|\right)$},
axis on top]
\addplot [
color=black,
solid
]
coordinates{
 (1,3.07374)(2,2.96864)(3,2.84519)(4,2.69429)(5,2.50605)(6,2.27153)(7,1.9952)(8,1.80457)(9,1.61152)(10,1.44926)(11,1.30016)(12,1.18366)(13,1.15969)(14,1.1462)(15,1.1327)(16,1.11893)(17,1.10473)(18,1.09024)(19,1.07545)(20,1.0602)(21,1.04463)(22,1.02858)(23,1.0122)(24,0.995301)(25,0.977877)(26,0.960092)(27,0.941751)(28,0.922837)(29,0.90333)(30,0.883418)(31,0.862669)(32,0.841489)(33,0.819417)(34,0.796652)(35,0.77317)(36,0.748698)(37,0.723454)(38,0.697143)(39,0.670001)(40,0.641413)(41,0.61192)(42,0.580851)(43,0.548459)(44,0.514685)(45,0.479097)(46,0.441198)(47,0.401671)(48,0.359118)(49,0.314255)(50,0.265933)(51,0.213869)(52,0.157756)(53,0.0959337)(54,0.0277761)(55,-0.0482548)(56,-0.134109)(57,-0.234548)(58,-0.355162)(59,-0.507887)(60,-0.724581)(61,-1.17179)(62,-1.87157)(63,-1.93339)(64,-1.99657)(65,-2.06472)(66,-2.13829)(67,-2.21775)(68,-2.3036)(69,-2.39747)(70,-2.49791)(71,-2.60667)(72,-2.72589)(73,-2.85676)(74,-2.9989)(75,-3.1592)(76,-3.34277)(77,-3.56231)(78,-3.85063)(79,-4.32732)(80,-5.13498) 
};

\addplot [
color=black,
dashed
]
coordinates{
 (1,-2.29337)(2,-2.29337)(3,-2.29337)(4,-2.29337)(5,-2.29337)(6,-2.29337)(7,-2.29337)(8,-2.29337)(9,-2.29337)(10,-2.29337)(11,-2.29337)(12,-2.29337)(13,-2.29337)(14,-2.29337)(15,-2.29337)(16,-2.29337)(17,-2.29337)(18,-2.29337)(19,-2.29337)(20,-2.29337)(21,-2.29337)(22,-2.29337)(23,-2.29337)(24,-2.29337)(25,-2.29337)(26,-2.29337)(27,-2.29337)(28,-2.29337)(29,-2.29337)(30,-2.29337)(31,-2.29337)(32,-2.29337)(33,-2.29337)(34,-2.29337)(35,-2.29337)(36,-2.29337)(37,-2.29337)(38,-2.29337)(39,-2.29337)(40,-2.29337)(41,-2.29337)(42,-2.29337)(43,-2.29337)(44,-2.29337)(45,-2.29337)(46,-2.29337)(47,-2.29337)(48,-2.29337)(49,-2.29337)(50,-2.29337)(51,-2.29337)(52,-2.29337)(53,-2.29337)(54,-2.29337)(55,-2.29337)(56,-2.29337)(57,-2.29337)(58,-2.29337)(59,-2.29337)(60,-2.29337)(61,-2.29337)(62,-2.29337)(63,-2.29337)(64,-2.29337)(65,-2.29337)(66,-2.29337)(67,-2.29337)(68,-2.29337)(69,-2.29337)(70,-2.29337)(71,-2.29337)(72,-2.29337)(73,-2.29337)(74,-2.29337)(75,-2.29337)(76,-2.29337)(77,-2.29337)(78,-2.29337)(79,-2.29337)(80,-2.29337) 
};

\end{axis}
\end{tikzpicture}
 \end{subfigure}
\begin{subfigure}
 \centering
\pgfplotsset{width=globprimalresidualADMMFirstEx1.tikz\columnwidth,height=globprimalresidualADMMFirstEx1.tikz\columnwidth,compat=newest,plot coordinates/math parser=false}
%
%
\begin{tikzpicture}

\begin{axis}[%
scale only axis,
width=3cm,
height=4cm,
xmin=1, xmax=304,
ymin=-7, ymax=3.5,
title={(c)},
xlabel={iteration in primal-dual method},
ylabel={$\textrm{log}\left(\|(r_{\textrm{primal},1},r_{\textrm{primal},2})\|\right)$},
axis on top]
\addplot [
color=black,
solid
]
coordinates{
 (1,2.16618)(2,1.77594)(3,1.67465)(4,1.61589)(5,1.52739)(6,1.42823)(7,1.33769)(8,1.24985)(9,1.16945)(10,1.10038)(11,1.03782)(12,0.970418)(13,0.902175)(14,0.841159)(15,0.783349)(16,0.734042)(17,0.689379)(18,0.643465)(19,0.598233)(20,0.555271)(21,0.515761)(22,0.479144)(23,0.441977)(24,0.406022)(25,0.371448)(26,0.336462)(27,0.298343)(28,0.263889)(29,0.229355)(30,0.197436)(31,0.16843)(32,0.141475)(33,0.115925)(34,0.0913599)(35,0.0675787)(36,0.0445551)(37,0.0213391)(38,-0.00123821)(39,-0.0228058)(40,-0.043552)(41,-0.0635177)(42,-0.0826918)(43,-0.101488)(44,-0.120294)(45,-0.138283)(46,-0.155756)(47,-0.17272)(48,-0.189199)(49,-0.208788)(50,-0.227338)(51,-0.245284)(52,-0.263048)(53,-0.281051)(54,-0.298952)(55,-0.316295)(56,-0.334178)(57,-0.351602)(58,-0.368566)(59,-0.385394)(60,-0.40229)(61,-0.418782)(62,-0.434881)(63,-0.450649)(64,-0.466105)(65,-0.481411)(66,-0.496366)(67,-0.511044)(68,-0.52546)(69,-0.539632)(70,-0.55355)(71,-0.56716)(72,-0.580498)(73,-0.593587)(74,-0.606438)(75,-0.619065)(76,-0.631478)(77,-0.643689)(78,-0.655709)(79,-0.667548)(80,-0.679216)(81,-0.69073)(82,-0.702608)(83,-0.714445)(84,-0.726033)(85,-0.73744)(86,-0.748697)(87,-0.760011)(88,-0.771174)(89,-0.782212)(90,-0.793142)(91,-0.803969)(92,-0.814693)(93,-0.825319)(94,-0.835846)(95,-0.846278)(96,-0.856615)(97,-0.86686)(98,-0.877017)(99,-0.887088)(100,-0.897077)(101,-0.906987)(102,-0.916822)(103,-0.926585)(104,-0.936279)(105,-0.945905)(106,-0.955467)(107,-0.964968)(108,-0.974408)(109,-0.983791)(110,-0.993118)(111,-1.00239)(112,-1.01161)(113,-1.02079)(114,-1.02991)(115,-1.03899)(116,-1.04802)(117,-1.057)(118,-1.06595)(119,-1.07485)(120,-1.08372)(121,-1.09254)(122,-1.10133)(123,-1.11008)(124,-1.1188)(125,-1.12748)(126,-1.13613)(127,-1.14475)(128,-1.15334)(129,-1.16189)(130,-1.17042)(131,-1.17891)(132,-1.18738)(133,-1.19582)(134,-1.20424)(135,-1.21263)(136,-1.22099)(137,-1.22933)(138,-1.23765)(139,-1.24594)(140,-1.25421)(141,-1.26246)(142,-1.27068)(143,-1.27889)(144,-1.28707)(145,-1.29524)(146,-1.30339)(147,-1.31151)(148,-1.31962)(149,-1.32772)(150,-1.33579)(151,-1.34385)(152,-1.35189)(153,-1.35992)(154,-1.36792)(155,-1.37592)(156,-1.3839)(157,-1.39186)(158,-1.39981)(159,-1.40775)(160,-1.41568)(161,-1.42359)(162,-1.43148)(163,-1.43937)(164,-1.44724)(165,-1.4551)(166,-1.46295)(167,-1.47079)(168,-1.47862)(169,-1.48646)(170,-1.49439)(171,-1.50227)(172,-1.51014)(173,-1.51801)(174,-1.52586)(175,-1.53371)(176,-1.54155)(177,-1.54939)(178,-1.55721)(179,-1.56503)(180,-1.57285)(181,-1.58065)(182,-1.58845)(183,-1.59624)(184,-1.60401)(185,-1.61178)(186,-1.61954)(187,-1.62729)(188,-1.63504)(189,-1.64277)(190,-1.65051)(191,-1.65823)(192,-1.66595)(193,-1.67366)(194,-1.68137)(195,-1.68907)(196,-1.69676)(197,-1.70445)(198,-1.71213)(199,-1.71981)(200,-1.72748)(201,-1.73515)(202,-1.74281)(203,-1.75047)(204,-1.75812)(205,-1.76577)(206,-1.77341)(207,-1.78106)(208,-1.78869)(209,-1.79632)(210,-1.80395)(211,-1.81158)(212,-1.8192)(213,-1.82682)(214,-1.83443)(215,-1.84204)(216,-1.84965)(217,-1.85725)(218,-1.86485)(219,-1.87245)(220,-1.88005)(221,-1.88764)(222,-1.89523)(223,-1.90281)(224,-1.9104)(225,-1.91798)(226,-1.92556)(227,-1.93313)(228,-1.9407)(229,-1.94827)(230,-1.95584)(231,-1.96341)(232,-1.97097)(233,-1.97853)(234,-1.98609)(235,-1.99365)(236,-2.0012)(237,-2.00875)(238,-2.0163)(239,-2.02385)(240,-2.03139)(241,-2.03894)(242,-2.04648)(243,-2.05402)(244,-2.06156)(245,-2.06909)(246,-2.07662)(247,-2.08416)(248,-2.09169)(249,-2.09922)(250,-2.10674)(251,-2.11427)(252,-2.12179)(253,-2.12932)(254,-2.13684)(255,-2.14444)(256,-2.15239)(257,-2.16028)(258,-2.16814)(259,-2.17598)(260,-2.18382)(261,-2.19165)(262,-2.19948)(263,-2.20726)(264,-2.21509)(265,-2.22293)(266,-2.23074)(267,-2.23855)(268,-2.24636)(269,-2.25416)(270,-2.26196)(271,-2.26975)(272,-2.27754)(273,-2.28533)(274,-2.29311)(275,-2.30088)(276,-2.30866)(277,-2.31642)(278,-2.32419)(279,-2.33195)(280,-2.33971)(281,-2.34746)(282,-2.35521)(283,-2.36295)(284,-2.37069)(285,-2.37843)(286,-2.38617)(287,-2.3939)(288,-2.40163)(289,-2.40935)(290,-2.41708)(291,-2.42479)(292,-2.43251)(293,-2.44023)(294,-2.44794)(295,-2.45565)(296,-2.46335)(297,-2.47106)(298,-2.47876)(299,-2.48646)(300,-2.49415)(301,-2.50185)(302,-2.50954)(303,-2.51723)(304,-2.52492) 
};

\addplot [
color=black,
dashed
]
coordinates{
 (1,-2.29337)(2,-2.29337)(3,-2.29337)(4,-2.29337)(5,-2.29337)(6,-2.29337)(7,-2.29337)(8,-2.29337)(9,-2.29337)(10,-2.29337)(11,-2.29337)(12,-2.29337)(13,-2.29337)(14,-2.29337)(15,-2.29337)(16,-2.29337)(17,-2.29337)(18,-2.29337)(19,-2.29337)(20,-2.29337)(21,-2.29337)(22,-2.29337)(23,-2.29337)(24,-2.29337)(25,-2.29337)(26,-2.29337)(27,-2.29337)(28,-2.29337)(29,-2.29337)(30,-2.29337)(31,-2.29337)(32,-2.29337)(33,-2.29337)(34,-2.29337)(35,-2.29337)(36,-2.29337)(37,-2.29337)(38,-2.29337)(39,-2.29337)(40,-2.29337)(41,-2.29337)(42,-2.29337)(43,-2.29337)(44,-2.29337)(45,-2.29337)(46,-2.29337)(47,-2.29337)(48,-2.29337)(49,-2.29337)(50,-2.29337)(51,-2.29337)(52,-2.29337)(53,-2.29337)(54,-2.29337)(55,-2.29337)(56,-2.29337)(57,-2.29337)(58,-2.29337)(59,-2.29337)(60,-2.29337)(61,-2.29337)(62,-2.29337)(63,-2.29337)(64,-2.29337)(65,-2.29337)(66,-2.29337)(67,-2.29337)(68,-2.29337)(69,-2.29337)(70,-2.29337)(71,-2.29337)(72,-2.29337)(73,-2.29337)(74,-2.29337)(75,-2.29337)(76,-2.29337)(77,-2.29337)(78,-2.29337)(79,-2.29337)(80,-2.29337)(81,-2.29337)(82,-2.29337)(83,-2.29337)(84,-2.29337)(85,-2.29337)(86,-2.29337)(87,-2.29337)(88,-2.29337)(89,-2.29337)(90,-2.29337)(91,-2.29337)(92,-2.29337)(93,-2.29337)(94,-2.29337)(95,-2.29337)(96,-2.29337)(97,-2.29337)(98,-2.29337)(99,-2.29337)(100,-2.29337)(101,-2.29337)(102,-2.29337)(103,-2.29337)(104,-2.29337)(105,-2.29337)(106,-2.29337)(107,-2.29337)(108,-2.29337)(109,-2.29337)(110,-2.29337)(111,-2.29337)(112,-2.29337)(113,-2.29337)(114,-2.29337)(115,-2.29337)(116,-2.29337)(117,-2.29337)(118,-2.29337)(119,-2.29337)(120,-2.29337)(121,-2.29337)(122,-2.29337)(123,-2.29337)(124,-2.29337)(125,-2.29337)(126,-2.29337)(127,-2.29337)(128,-2.29337)(129,-2.29337)(130,-2.29337)(131,-2.29337)(132,-2.29337)(133,-2.29337)(134,-2.29337)(135,-2.29337)(136,-2.29337)(137,-2.29337)(138,-2.29337)(139,-2.29337)(140,-2.29337)(141,-2.29337)(142,-2.29337)(143,-2.29337)(144,-2.29337)(145,-2.29337)(146,-2.29337)(147,-2.29337)(148,-2.29337)(149,-2.29337)(150,-2.29337)(151,-2.29337)(152,-2.29337)(153,-2.29337)(154,-2.29337)(155,-2.29337)(156,-2.29337)(157,-2.29337)(158,-2.29337)(159,-2.29337)(160,-2.29337)(161,-2.29337)(162,-2.29337)(163,-2.29337)(164,-2.29337)(165,-2.29337)(166,-2.29337)(167,-2.29337)(168,-2.29337)(169,-2.29337)(170,-2.29337)(171,-2.29337)(172,-2.29337)(173,-2.29337)(174,-2.29337)(175,-2.29337)(176,-2.29337)(177,-2.29337)(178,-2.29337)(179,-2.29337)(180,-2.29337)(181,-2.29337)(182,-2.29337)(183,-2.29337)(184,-2.29337)(185,-2.29337)(186,-2.29337)(187,-2.29337)(188,-2.29337)(189,-2.29337)(190,-2.29337)(191,-2.29337)(192,-2.29337)(193,-2.29337)(194,-2.29337)(195,-2.29337)(196,-2.29337)(197,-2.29337)(198,-2.29337)(199,-2.29337)(200,-2.29337)(201,-2.29337)(202,-2.29337)(203,-2.29337)(204,-2.29337)(205,-2.29337)(206,-2.29337)(207,-2.29337)(208,-2.29337)(209,-2.29337)(210,-2.29337)(211,-2.29337)(212,-2.29337)(213,-2.29337)(214,-2.29337)(215,-2.29337)(216,-2.29337)(217,-2.29337)(218,-2.29337)(219,-2.29337)(220,-2.29337)(221,-2.29337)(222,-2.29337)(223,-2.29337)(224,-2.29337)(225,-2.29337)(226,-2.29337)(227,-2.29337)(228,-2.29337)(229,-2.29337)(230,-2.29337)(231,-2.29337)(232,-2.29337)(233,-2.29337)(234,-2.29337)(235,-2.29337)(236,-2.29337)(237,-2.29337)(238,-2.29337)(239,-2.29337)(240,-2.29337)(241,-2.29337)(242,-2.29337)(243,-2.29337)(244,-2.29337)(245,-2.29337)(246,-2.29337)(247,-2.29337)(248,-2.29337)(249,-2.29337)(250,-2.29337)(251,-2.29337)(252,-2.29337)(253,-2.29337)(254,-2.29337)(255,-2.29337)(256,-2.29337)(257,-2.29337)(258,-2.29337)(259,-2.29337)(260,-2.29337)(261,-2.29337)(262,-2.29337)(263,-2.29337)(264,-2.29337)(265,-2.29337)(266,-2.29337)(267,-2.29337)(268,-2.29337)(269,-2.29337)(270,-2.29337)(271,-2.29337)(272,-2.29337)(273,-2.29337)(274,-2.29337)(275,-2.29337)(276,-2.29337)(277,-2.29337)(278,-2.29337)(279,-2.29337)(280,-2.29337)(281,-2.29337)(282,-2.29337)(283,-2.29337)(284,-2.29337)(285,-2.29337)(286,-2.29337)(287,-2.29337)(288,-2.29337)(289,-2.29337)(290,-2.29337)(291,-2.29337)(292,-2.29337)(293,-2.29337)(294,-2.29337)(295,-2.29337)(296,-2.29337)(297,-2.29337)(298,-2.29337)(299,-2.29337)(300,-2.29337)(301,-2.29337)(302,-2.29337)(303,-2.29337)(304,-2.29337) 
};

\end{axis}
\end{tikzpicture}
 \end{subfigure}
\caption{Norm of primal global residual.  The value of the norm of the primal global residual \eqref{solid} and its upper bound $\textrm{log}\left(\epsilon_{\textrm{feas}}\right)$, \eqref{dashed} are shown. In figures (a), (b) and (c), algorithms \ref{alg:algNPDexact}, \ref{alg:algNPD} and ADMM have been used, respectively.} \label{fig:globprimalResADMMEx1}\end{center}
\end{figure}

\begin{figure}[h!]
\begin{center}
\begin{subfigure}
 \centering
\pgfplotsset{width=dualresidualADMMOrigEx1.tikz\columnwidth,height=dualresidualADMMOrigEx1.tikz\columnwidth,compat=newest,plot coordinates/math parser=false}
%
%
\begin{tikzpicture}

\begin{axis}[%
scale only axis,
width=3cm,
height=4cm,
xmin=1, xmax=21,
ymin=-15, ymax=5,
title={(a)},
xlabel={iteration in primal-dual method},
ylabel={$\frac{1}{N}\sum_{i=1}^N\textrm{log}\left(\|r_{\textrm{dual}}^i\|\right)$},
axis on top]
\addplot [
color=black,
solid
]
plot [error bars/.cd, y dir = both, y explicit]
coordinates{
 (1,5.45125)+-(0.0,0.0522606)(2,4.35661)+-(0.0,0.0522606)(3,3.08325)+-(0.0,0.0522606)(4,2.36274)+-(0.0,0.0522606)(5,1.88488)+-(0.0,0.0522606)(6,1.15874)+-(0.0,0.0522606)(7,0.81172)+-(0.0,0.0522606)(8,0.440864)+-(0.0,0.0522606)(9,0.0574184)+-(0.0,0.0522606)(10,-0.121405)+-(0.0,0.0522606)(11,-0.541983)+-(0.0,0.0522606)(12,-0.998357)+-(0.0,0.0522606)(13,-1.48097)+-(0.0,0.0522606)(14,-1.94601)+-(0.0,0.0522606)(15,-2.42079)+-(0.0,0.0522606)(16,-2.89966)+-(0.0,0.0522606)(17,-4.55077)+-(0.0,0.0522606)(18,-5.11594)+-(0.0,0.0522606)(19,-7.29851)+-(0.0,0.0522606)(20,-9.76307)+-(0.0,0.0522608)(21,-12.9508)+-(0.0,0.0523146) 
};

\addplot [
color=black,
dashed
]
coordinates{
 (1,-6.28571)(2,-6.28571)(3,-6.28571)(4,-6.28571)(5,-6.28571)(6,-6.28571)(7,-6.28571)(8,-6.28571)(9,-6.28571)(10,-6.28571)(11,-6.28571)(12,-6.28571)(13,-6.28571)(14,-6.28571)(15,-6.28571)(16,-6.28571)(17,-6.28571)(18,-6.28571)(19,-6.28571)(20,-6.28571)(21,-6.28571) 
};

\end{axis}
\end{tikzpicture}
 \end{subfigure}
\begin{subfigure}
 \centering
\pgfplotsset{width=dualresidualADMMDistEx1.tikz\columnwidth,height=dualresidualADMMDistEx1.tikz\columnwidth,compat=newest,plot coordinates/math parser=false}
%
%
\begin{tikzpicture}

\begin{axis}[%
scale only axis,
width=3cm,
height=4cm,
xmin=1, xmax=80,
ymin=-15, ymax=5,
title={(b)},
xlabel={iteration in primal-dual method},
ylabel={$\frac{1}{N}\sum_{i=1}^N\textrm{log}\left(\|r_{\textrm{dual}}^i\|\right)$},
axis on top]
\addplot [
color=black,
solid
]
plot [error bars/.cd, y dir = both, y explicit]
coordinates{
 (4,5.14256)+-(0.0,0.0522606)(5,4.76608)+-(0.0,0.0522606)(6,4.29703)+-(0.0,0.0522606)(7,3.74437)+-(0.0,0.0522606)(8,3.36312)+-(0.0,0.0522606)(9,2.97702)+-(0.0,0.0522606)(10,2.65249)+-(0.0,0.0522606)(11,2.3543)+-(0.0,0.0522606)(12,2.12129)+-(0.0,0.0522606)(13,2.07334)+-(0.0,0.0522606)(14,2.04636)+-(0.0,0.0522606)(15,2.01938)+-(0.0,0.0522606)(16,1.99184)+-(0.0,0.0522606)(17,1.96344)+-(0.0,0.0522606)(18,1.93445)+-(0.0,0.0522606)(19,1.90487)+-(0.0,0.0522606)(20,1.87437)+-(0.0,0.0522606)(21,1.84324)+-(0.0,0.0522606)(22,1.81113)+-(0.0,0.0522606)(23,1.77836)+-(0.0,0.0522606)(24,1.74457)+-(0.0,0.0522606)(25,1.70972)+-(0.0,0.0522606)(26,1.67415)+-(0.0,0.0522606)(27,1.63747)+-(0.0,0.0522606)(28,1.59964)+-(0.0,0.0522606)(29,1.56063)+-(0.0,0.0522606)(30,1.5208)+-(0.0,0.0522606)(31,1.4793)+-(0.0,0.0522606)(32,1.43694)+-(0.0,0.0522606)(33,1.3928)+-(0.0,0.0522606)(34,1.34727)+-(0.0,0.0522606)(35,1.30031)+-(0.0,0.0522606)(36,1.25136)+-(0.0,0.0522606)(37,1.20087)+-(0.0,0.0522606)(38,1.14825)+-(0.0,0.0522606)(39,1.09397)+-(0.0,0.0522606)(40,1.03679)+-(0.0,0.0522606)(41,0.977804)+-(0.0,0.0522606)(42,0.915666)+-(0.0,0.0522606)(43,0.850881)+-(0.0,0.0522606)(44,0.783333)+-(0.0,0.0522606)(45,0.712157)+-(0.0,0.0522606)(46,0.636359)+-(0.0,0.0522606)(47,0.557305)+-(0.0,0.0522606)(48,0.472198)+-(0.0,0.0522606)(49,0.382471)+-(0.0,0.0522606)(50,0.285827)+-(0.0,0.0522606)(51,0.181699)+-(0.0,0.0522606)(52,0.0694713)+-(0.0,0.0522606)(53,-0.0541731)+-(0.0,0.0522606)(54,-0.190489)+-(0.0,0.0522606)(55,-0.342552)+-(0.0,0.0522606)(56,-0.514261)+-(0.0,0.0522606)(57,-0.71514)+-(0.0,0.0522606)(58,-0.95637)+-(0.0,0.0522607)(59,-1.26182)+-(0.0,0.0522607)(60,-1.69521)+-(0.0,0.0522607)(61,-2.58964)+-(0.0,0.0522609)(62,-3.98923)+-(0.0,0.0522611)(63,-4.11287)+-(0.0,0.0522611)(64,-4.23922)+-(0.0,0.0522611)(65,-4.37553)+-(0.0,0.0522611)(66,-4.52268)+-(0.0,0.0522611)(67,-4.68158)+-(0.0,0.0522611)(68,-4.85329)+-(0.0,0.0522612)(69,-5.04104)+-(0.0,0.0522612)(70,-5.24192)+-(0.0,0.0522612)(71,-5.45943)+-(0.0,0.0522612)(72,-5.69789)+-(0.0,0.0522613)(73,-5.95962)+-(0.0,0.0522613)(74,-6.24391)+-(0.0,0.0522613)(75,-6.56451)+-(0.0,0.0522614)(76,-6.93166)+-(0.0,0.0522614)(77,-7.37075)+-(0.0,0.0522615)(78,-7.94739)+-(0.0,0.0522615)(79,-8.90077)+-(0.0,0.0522615)(80,-10.5161)+-(0.0,0.0522615) 
};

\addplot [
color=black,
dashed
]
coordinates{
 (1,-6.28571)(2,-6.28571)(3,-6.28571)(4,-6.28571)(5,-6.28571)(6,-6.28571)(7,-6.28571)(8,-6.28571)(9,-6.28571)(10,-6.28571)(11,-6.28571)(12,-6.28571)(13,-6.28571)(14,-6.28571)(15,-6.28571)(16,-6.28571)(17,-6.28571)(18,-6.28571)(19,-6.28571)(20,-6.28571)(21,-6.28571)(22,-6.28571)(23,-6.28571)(24,-6.28571)(25,-6.28571)(26,-6.28571)(27,-6.28571)(28,-6.28571)(29,-6.28571)(30,-6.28571)(31,-6.28571)(32,-6.28571)(33,-6.28571)(34,-6.28571)(35,-6.28571)(36,-6.28571)(37,-6.28571)(38,-6.28571)(39,-6.28571)(40,-6.28571)(41,-6.28571)(42,-6.28571)(43,-6.28571)(44,-6.28571)(45,-6.28571)(46,-6.28571)(47,-6.28571)(48,-6.28571)(49,-6.28571)(50,-6.28571)(51,-6.28571)(52,-6.28571)(53,-6.28571)(54,-6.28571)(55,-6.28571)(56,-6.28571)(57,-6.28571)(58,-6.28571)(59,-6.28571)(60,-6.28571)(61,-6.28571)(62,-6.28571)(63,-6.28571)(64,-6.28571)(65,-6.28571)(66,-6.28571)(67,-6.28571)(68,-6.28571)(69,-6.28571)(70,-6.28571)(71,-6.28571)(72,-6.28571)(73,-6.28571)(74,-6.28571)(75,-6.28571)(76,-6.28571)(77,-6.28571)(78,-6.28571)(79,-6.28571)(80,-6.28571) 
};

\end{axis}
\end{tikzpicture}
 \end{subfigure}
\begin{subfigure}
 \centering
\pgfplotsset{width=dualresidualADMMFirstEx1.tikz\columnwidth,height=dualresidualADMMFirstEx1.tikz\columnwidth,compat=newest,plot coordinates/math parser=false}\input{dualresidualADMMFirstEx1.tikz}
 \end{subfigure}
\caption{Norm of dual local residual.  The averaged, over all subproblems, value of the norm of the dual local residual along with its standard deviation \eqref{solid}  and upper bound, $\textrm{log}\left(\epsilon_{\textrm{feas}}^2/N\right)$, \eqref{dashed} are shown. In figures (a), (b) and (c), algorithms \ref{alg:algNPDexact}, \ref{alg:algNPD} and ADMM have been used, respectively.} \label{fig:dualResADMMEx1}\end{center}
\end{figure}
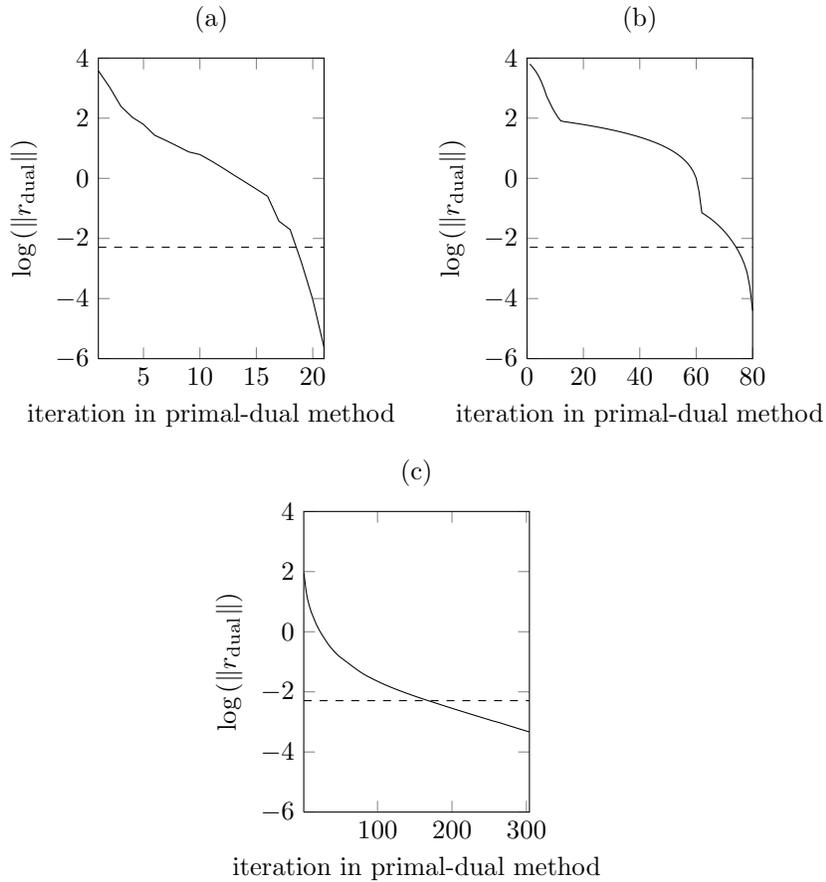
\begin{figure}[h!]
\begin{center}
\begin{subfigure}
 \centering
\pgfplotsset{width=globdualresidualADMMOrigEx1.tikz\columnwidth,height=globdualresidualADMMOrigEx1.tikz\columnwidth,compat=newest,plot coordinates/math parser=false}
%
%
\begin{tikzpicture}

\begin{axis}[%
scale only axis,
width=3cm,
height=4cm,
xmin=1, xmax=21,
ymin=-6, ymax=4,
title={(a)},
xlabel={iteration in primal-dual method},
ylabel={$\textrm{log}\left(\|r_{\textrm{dual}}\|\right)$},
axis on top]
\addplot [
color=black,
solid
]
coordinates{
 (1,3.57665)(2,3.02933)(3,2.39265)(4,2.03239)(5,1.79346)(6,1.43039)(7,1.25688)(8,1.07145)(9,0.879731)(10,0.790319)(11,0.58003)(12,0.351843)(13,0.110538)(14,-0.121982)(15,-0.359371)(16,-0.598806)(17,-1.42437)(18,-1.70695)(19,-2.79823)(20,-4.03052)(21,-5.62439) 
};

\addplot [
color=black,
dashed
]
coordinates{
 (1,-2.29337)(2,-2.29337)(3,-2.29337)(4,-2.29337)(5,-2.29337)(6,-2.29337)(7,-2.29337)(8,-2.29337)(9,-2.29337)(10,-2.29337)(11,-2.29337)(12,-2.29337)(13,-2.29337)(14,-2.29337)(15,-2.29337)(16,-2.29337)(17,-2.29337)(18,-2.29337)(19,-2.29337)(20,-2.29337)(21,-2.29337) 
};

\end{axis}
\end{tikzpicture}
 \end{subfigure}
\begin{subfigure}
 \centering
\pgfplotsset{width=globdualresidualADMMDistEx1.tikz\columnwidth,height=globdualresidualADMMDistEx1.tikz\columnwidth,compat=newest,plot coordinates/math parser=false}
%
%
\begin{tikzpicture}

\begin{axis}[%
scale only axis,
width=3cm,
height=4cm,
xmin=0, xmax=80,
ymin=-6, ymax=4,
title={(b)},
xlabel={iteration in primal-dual method},
ylabel={$\textrm{log}\left(\|r_{\textrm{dual}}\|\right)$},
axis on top]
\addplot [
color=black,
solid
]
coordinates{
 (1,3.80175)(2,3.69665)(3,3.5732)(4,3.4223)(5,3.23406)(6,2.99954)(7,2.72321)(8,2.53258)(9,2.33953)(10,2.17727)(11,2.02817)(12,1.91166)(13,1.88769)(14,1.8742)(15,1.86071)(16,1.84694)(17,1.83274)(18,1.81825)(19,1.80346)(20,1.78821)(21,1.77264)(22,1.75659)(23,1.7402)(24,1.72331)(25,1.70588)(26,1.6881)(27,1.66976)(28,1.65084)(29,1.63134)(30,1.61142)(31,1.59067)(32,1.56949)(33,1.54742)(34,1.52466)(35,1.50118)(36,1.4767)(37,1.45146)(38,1.42515)(39,1.39801)(40,1.36942)(41,1.33992)(42,1.30886)(43,1.27646)(44,1.24269)(45,1.2071)(46,1.1692)(47,1.12967)(48,1.08712)(49,1.04226)(50,0.993935)(51,0.941872)(52,0.885758)(53,0.823935)(54,0.755778)(55,0.679746)(56,0.593891)(57,0.493452)(58,0.372837)(59,0.220112)(60,0.00341552)(61,-0.4438)(62,-1.14359)(63,-1.20541)(64,-1.26859)(65,-1.33675)(66,-1.41032)(67,-1.48977)(68,-1.57562)(69,-1.6695)(70,-1.76994)(71,-1.87869)(72,-1.99792)(73,-2.12879)(74,-2.27093)(75,-2.43123)(76,-2.61481)(77,-2.83435)(78,-3.12267)(79,-3.59936)(80,-4.40701) 
};

\addplot [
color=black,
dashed
]
coordinates{
 (1,-2.29337)(2,-2.29337)(3,-2.29337)(4,-2.29337)(5,-2.29337)(6,-2.29337)(7,-2.29337)(8,-2.29337)(9,-2.29337)(10,-2.29337)(11,-2.29337)(12,-2.29337)(13,-2.29337)(14,-2.29337)(15,-2.29337)(16,-2.29337)(17,-2.29337)(18,-2.29337)(19,-2.29337)(20,-2.29337)(21,-2.29337)(22,-2.29337)(23,-2.29337)(24,-2.29337)(25,-2.29337)(26,-2.29337)(27,-2.29337)(28,-2.29337)(29,-2.29337)(30,-2.29337)(31,-2.29337)(32,-2.29337)(33,-2.29337)(34,-2.29337)(35,-2.29337)(36,-2.29337)(37,-2.29337)(38,-2.29337)(39,-2.29337)(40,-2.29337)(41,-2.29337)(42,-2.29337)(43,-2.29337)(44,-2.29337)(45,-2.29337)(46,-2.29337)(47,-2.29337)(48,-2.29337)(49,-2.29337)(50,-2.29337)(51,-2.29337)(52,-2.29337)(53,-2.29337)(54,-2.29337)(55,-2.29337)(56,-2.29337)(57,-2.29337)(58,-2.29337)(59,-2.29337)(60,-2.29337)(61,-2.29337)(62,-2.29337)(63,-2.29337)(64,-2.29337)(65,-2.29337)(66,-2.29337)(67,-2.29337)(68,-2.29337)(69,-2.29337)(70,-2.29337)(71,-2.29337)(72,-2.29337)(73,-2.29337)(74,-2.29337)(75,-2.29337)(76,-2.29337)(77,-2.29337)(78,-2.29337)(79,-2.29337)(80,-2.29337) 
};

\end{axis}
\end{tikzpicture}
 \end{subfigure}
\begin{subfigure}
 \centering
\pgfplotsset{width=globdualresidualADMMFirstEx1.tikz\columnwidth,height=globdualresidualADMMFirstEx1.tikz\columnwidth,compat=newest,plot coordinates/math parser=false}
%
%
\begin{tikzpicture}

\begin{axis}[%
scale only axis,
width=3cm,
height=4cm,
xmin=1, xmax=304,
ymin=-6, ymax=4,
title={(c)},
xlabel={iteration in primal-dual method},
ylabel={$\textrm{log}\left(\|r_{\textrm{dual}}\|\right)$},
axis on top]
\addplot [
color=black,
solid
]
coordinates{
 (1,1.96484)(2,1.79568)(3,1.58297)(4,1.38293)(5,1.22587)(6,1.10196)(7,0.997016)(8,0.904596)(9,0.820812)(10,0.736726)(11,0.660099)(12,0.596406)(13,0.538212)(14,0.478551)(15,0.416165)(16,0.352922)(17,0.293237)(18,0.239582)(19,0.189448)(20,0.141228)(21,0.0940209)(22,0.0479639)(23,0.00270974)(24,-0.0402089)(25,-0.080975)(26,-0.120057)(27,-0.15769)(28,-0.194187)(29,-0.231803)(30,-0.270328)(31,-0.309191)(32,-0.347607)(33,-0.384708)(34,-0.419996)(35,-0.453477)(36,-0.48552)(37,-0.516466)(38,-0.545817)(39,-0.574446)(40,-0.602661)(41,-0.630515)(42,-0.658189)(43,-0.685241)(44,-0.711247)(45,-0.73595)(46,-0.759972)(47,-0.783535)(48,-0.806659)(49,-0.828209)(50,-0.846632)(51,-0.864733)(52,-0.883028)(53,-0.90151)(54,-0.91979)(55,-0.938348)(56,-0.957097)(57,-0.975442)(58,-0.993925)(59,-1.01261)(60,-1.03135)(61,-1.04991)(62,-1.06863)(63,-1.0875)(64,-1.10643)(65,-1.12526)(66,-1.14395)(67,-1.16253)(68,-1.18094)(69,-1.19914)(70,-1.21738)(71,-1.23567)(72,-1.2537)(73,-1.27145)(74,-1.28892)(75,-1.30608)(76,-1.32295)(77,-1.3395)(78,-1.35574)(79,-1.37167)(80,-1.3873)(81,-1.40263)(82,-1.4174)(83,-1.43147)(84,-1.44497)(85,-1.45845)(86,-1.47196)(87,-1.48535)(88,-1.4984)(89,-1.51132)(90,-1.52417)(91,-1.53695)(92,-1.54966)(93,-1.56231)(94,-1.57489)(95,-1.58739)(96,-1.59981)(97,-1.61215)(98,-1.62439)(99,-1.63653)(100,-1.64855)(101,-1.66045)(102,-1.67224)(103,-1.6839)(104,-1.69544)(105,-1.70687)(106,-1.71818)(107,-1.72938)(108,-1.74047)(109,-1.75145)(110,-1.76234)(111,-1.77312)(112,-1.78381)(113,-1.7944)(114,-1.8049)(115,-1.81532)(116,-1.82566)(117,-1.83591)(118,-1.84609)(119,-1.85619)(120,-1.86622)(121,-1.87618)(122,-1.88608)(123,-1.8959)(124,-1.90567)(125,-1.91537)(126,-1.92501)(127,-1.9346)(128,-1.94412)(129,-1.9536)(130,-1.96302)(131,-1.97239)(132,-1.98171)(133,-1.99098)(134,-2.0002)(135,-2.00938)(136,-2.01851)(137,-2.0276)(138,-2.03665)(139,-2.04566)(140,-2.05462)(141,-2.06355)(142,-2.07244)(143,-2.08129)(144,-2.09011)(145,-2.09889)(146,-2.10764)(147,-2.11635)(148,-2.12503)(149,-2.13368)(150,-2.1423)(151,-2.15089)(152,-2.15945)(153,-2.16798)(154,-2.17648)(155,-2.18496)(156,-2.19341)(157,-2.20183)(158,-2.21024)(159,-2.21861)(160,-2.22697)(161,-2.2353)(162,-2.24361)(163,-2.2519)(164,-2.26017)(165,-2.26841)(166,-2.27664)(167,-2.28484)(168,-2.29302)(169,-2.30117)(170,-2.30914)(171,-2.31689)(172,-2.32467)(173,-2.33251)(174,-2.34038)(175,-2.34829)(176,-2.35621)(177,-2.36415)(178,-2.3721)(179,-2.38006)(180,-2.38802)(181,-2.39599)(182,-2.40399)(183,-2.41202)(184,-2.42005)(185,-2.42806)(186,-2.43606)(187,-2.44404)(188,-2.45202)(189,-2.45997)(190,-2.46792)(191,-2.47585)(192,-2.48377)(193,-2.49168)(194,-2.49957)(195,-2.50745)(196,-2.51532)(197,-2.52318)(198,-2.53103)(199,-2.53886)(200,-2.54669)(201,-2.55451)(202,-2.56232)(203,-2.57012)(204,-2.57791)(205,-2.58569)(206,-2.59347)(207,-2.60123)(208,-2.60899)(209,-2.61674)(210,-2.62449)(211,-2.63222)(212,-2.63995)(213,-2.64767)(214,-2.65538)(215,-2.66309)(216,-2.67079)(217,-2.67848)(218,-2.68617)(219,-2.69385)(220,-2.70153)(221,-2.7092)(222,-2.71686)(223,-2.72452)(224,-2.73217)(225,-2.73981)(226,-2.74745)(227,-2.75509)(228,-2.76272)(229,-2.77034)(230,-2.77796)(231,-2.78558)(232,-2.79318)(233,-2.80079)(234,-2.80839)(235,-2.81598)(236,-2.82357)(237,-2.83115)(238,-2.83873)(239,-2.8463)(240,-2.85386)(241,-2.86142)(242,-2.86898)(243,-2.87653)(244,-2.88407)(245,-2.89161)(246,-2.89914)(247,-2.90667)(248,-2.91418)(249,-2.92169)(250,-2.9292)(251,-2.9367)(252,-2.94422)(253,-2.95175)(254,-2.95921)(255,-2.96666)(256,-2.97399)(257,-2.98022)(258,-2.98629)(259,-2.99268)(260,-2.99938)(261,-3.00634)(262,-3.0135)(263,-3.01871)(264,-3.02784)(265,-3.03546)(266,-3.04312)(267,-3.0508)(268,-3.0585)(269,-3.06622)(270,-3.07396)(271,-3.08171)(272,-3.08947)(273,-3.09724)(274,-3.10502)(275,-3.1128)(276,-3.12057)(277,-3.12834)(278,-3.13611)(279,-3.14387)(280,-3.15162)(281,-3.15936)(282,-3.1671)(283,-3.17482)(284,-3.18253)(285,-3.19023)(286,-3.19791)(287,-3.20559)(288,-3.21326)(289,-3.22091)(290,-3.22855)(291,-3.23618)(292,-3.2438)(293,-3.25141)(294,-3.25901)(295,-3.26659)(296,-3.27416)(297,-3.28171)(298,-3.28926)(299,-3.29679)(300,-3.30431)(301,-3.31182)(302,-3.31931)(303,-3.32679)(304,-3.33426) 
};

\addplot [
color=black,
dashed
]
coordinates{
 (1,-2.29337)(2,-2.29337)(3,-2.29337)(4,-2.29337)(5,-2.29337)(6,-2.29337)(7,-2.29337)(8,-2.29337)(9,-2.29337)(10,-2.29337)(11,-2.29337)(12,-2.29337)(13,-2.29337)(14,-2.29337)(15,-2.29337)(16,-2.29337)(17,-2.29337)(18,-2.29337)(19,-2.29337)(20,-2.29337)(21,-2.29337)(22,-2.29337)(23,-2.29337)(24,-2.29337)(25,-2.29337)(26,-2.29337)(27,-2.29337)(28,-2.29337)(29,-2.29337)(30,-2.29337)(31,-2.29337)(32,-2.29337)(33,-2.29337)(34,-2.29337)(35,-2.29337)(36,-2.29337)(37,-2.29337)(38,-2.29337)(39,-2.29337)(40,-2.29337)(41,-2.29337)(42,-2.29337)(43,-2.29337)(44,-2.29337)(45,-2.29337)(46,-2.29337)(47,-2.29337)(48,-2.29337)(49,-2.29337)(50,-2.29337)(51,-2.29337)(52,-2.29337)(53,-2.29337)(54,-2.29337)(55,-2.29337)(56,-2.29337)(57,-2.29337)(58,-2.29337)(59,-2.29337)(60,-2.29337)(61,-2.29337)(62,-2.29337)(63,-2.29337)(64,-2.29337)(65,-2.29337)(66,-2.29337)(67,-2.29337)(68,-2.29337)(69,-2.29337)(70,-2.29337)(71,-2.29337)(72,-2.29337)(73,-2.29337)(74,-2.29337)(75,-2.29337)(76,-2.29337)(77,-2.29337)(78,-2.29337)(79,-2.29337)(80,-2.29337)(81,-2.29337)(82,-2.29337)(83,-2.29337)(84,-2.29337)(85,-2.29337)(86,-2.29337)(87,-2.29337)(88,-2.29337)(89,-2.29337)(90,-2.29337)(91,-2.29337)(92,-2.29337)(93,-2.29337)(94,-2.29337)(95,-2.29337)(96,-2.29337)(97,-2.29337)(98,-2.29337)(99,-2.29337)(100,-2.29337)(101,-2.29337)(102,-2.29337)(103,-2.29337)(104,-2.29337)(105,-2.29337)(106,-2.29337)(107,-2.29337)(108,-2.29337)(109,-2.29337)(110,-2.29337)(111,-2.29337)(112,-2.29337)(113,-2.29337)(114,-2.29337)(115,-2.29337)(116,-2.29337)(117,-2.29337)(118,-2.29337)(119,-2.29337)(120,-2.29337)(121,-2.29337)(122,-2.29337)(123,-2.29337)(124,-2.29337)(125,-2.29337)(126,-2.29337)(127,-2.29337)(128,-2.29337)(129,-2.29337)(130,-2.29337)(131,-2.29337)(132,-2.29337)(133,-2.29337)(134,-2.29337)(135,-2.29337)(136,-2.29337)(137,-2.29337)(138,-2.29337)(139,-2.29337)(140,-2.29337)(141,-2.29337)(142,-2.29337)(143,-2.29337)(144,-2.29337)(145,-2.29337)(146,-2.29337)(147,-2.29337)(148,-2.29337)(149,-2.29337)(150,-2.29337)(151,-2.29337)(152,-2.29337)(153,-2.29337)(154,-2.29337)(155,-2.29337)(156,-2.29337)(157,-2.29337)(158,-2.29337)(159,-2.29337)(160,-2.29337)(161,-2.29337)(162,-2.29337)(163,-2.29337)(164,-2.29337)(165,-2.29337)(166,-2.29337)(167,-2.29337)(168,-2.29337)(169,-2.29337)(170,-2.29337)(171,-2.29337)(172,-2.29337)(173,-2.29337)(174,-2.29337)(175,-2.29337)(176,-2.29337)(177,-2.29337)(178,-2.29337)(179,-2.29337)(180,-2.29337)(181,-2.29337)(182,-2.29337)(183,-2.29337)(184,-2.29337)(185,-2.29337)(186,-2.29337)(187,-2.29337)(188,-2.29337)(189,-2.29337)(190,-2.29337)(191,-2.29337)(192,-2.29337)(193,-2.29337)(194,-2.29337)(195,-2.29337)(196,-2.29337)(197,-2.29337)(198,-2.29337)(199,-2.29337)(200,-2.29337)(201,-2.29337)(202,-2.29337)(203,-2.29337)(204,-2.29337)(205,-2.29337)(206,-2.29337)(207,-2.29337)(208,-2.29337)(209,-2.29337)(210,-2.29337)(211,-2.29337)(212,-2.29337)(213,-2.29337)(214,-2.29337)(215,-2.29337)(216,-2.29337)(217,-2.29337)(218,-2.29337)(219,-2.29337)(220,-2.29337)(221,-2.29337)(222,-2.29337)(223,-2.29337)(224,-2.29337)(225,-2.29337)(226,-2.29337)(227,-2.29337)(228,-2.29337)(229,-2.29337)(230,-2.29337)(231,-2.29337)(232,-2.29337)(233,-2.29337)(234,-2.29337)(235,-2.29337)(236,-2.29337)(237,-2.29337)(238,-2.29337)(239,-2.29337)(240,-2.29337)(241,-2.29337)(242,-2.29337)(243,-2.29337)(244,-2.29337)(245,-2.29337)(246,-2.29337)(247,-2.29337)(248,-2.29337)(249,-2.29337)(250,-2.29337)(251,-2.29337)(252,-2.29337)(253,-2.29337)(254,-2.29337)(255,-2.29337)(256,-2.29337)(257,-2.29337)(258,-2.29337)(259,-2.29337)(260,-2.29337)(261,-2.29337)(262,-2.29337)(263,-2.29337)(264,-2.29337)(265,-2.29337)(266,-2.29337)(267,-2.29337)(268,-2.29337)(269,-2.29337)(270,-2.29337)(271,-2.29337)(272,-2.29337)(273,-2.29337)(274,-2.29337)(275,-2.29337)(276,-2.29337)(277,-2.29337)(278,-2.29337)(279,-2.29337)(280,-2.29337)(281,-2.29337)(282,-2.29337)(283,-2.29337)(284,-2.29337)(285,-2.29337)(286,-2.29337)(287,-2.29337)(288,-2.29337)(289,-2.29337)(290,-2.29337)(291,-2.29337)(292,-2.29337)(293,-2.29337)(294,-2.29337)(295,-2.29337)(296,-2.29337)(297,-2.29337)(298,-2.29337)(299,-2.29337)(300,-2.29337)(301,-2.29337)(302,-2.29337)(303,-2.29337)(304,-2.29337) 
};

\end{axis}
\end{tikzpicture}
 \end{subfigure}
\caption{Norm of dual global residual.  The value of the norm of the dual global residual \eqref{solid} and its upper bound $\textrm{log}\left(\epsilon_{\textrm{feas}}\right)$, \eqref{dashed} are shown. In figures (a), (b) and (c), algorithms \ref{alg:algNPDexact}, \ref{alg:algNPD} and ADMM have been used, respectively.} \label{fig:globdualResADMMEx1}\end{center}
\end{figure}

\begin{figure}[h!]
\begin{center}
\begin{subfigure}
 \centering
\pgfplotsset{width=surADMMOrigEx1.tikz\columnwidth,height=surADMMOrigEx1.tikz\columnwidth,compat=newest,plot coordinates/math parser=false}
%
%
\begin{tikzpicture}

\begin{axis}[%
scale only axis,
width=3cm,
height=4cm,
xmin=1, xmax=21,
ymin=-15, ymax=5,
title={(a)},
xlabel={iteration in primal-dual method},
ylabel={$\frac{1}{N}\sum_{i=1}^N\textrm{log}\left(\hat{\eta}^{i}\right)$},
axis on top]
\addplot [
color=black,
solid
]
plot [error bars/.cd, y dir = both, y explicit]
coordinates{
 (1,3.14521)+-(0.0,0.0304803)(2,2.67371)+-(0.0,0.0426082)(3,2.12466)+-(0.0,0.0526601)(4,1.78508)+-(0.0,0.0573721)(5,1.55675)+-(0.0,0.0599578)(6,1.21844)+-(0.0,0.0660235)(7,1.05079)+-(0.0,0.0669656)(8,0.871414)+-(0.0,0.0678667)(9,0.685629)+-(0.0,0.0686327)(10,0.597482)+-(0.0,0.0687677)(11,0.393247)+-(0.0,0.0696354)(12,0.171434)+-(0.0,0.0707442)(13,-0.0636908)+-(0.0,0.0721075)(14,-0.291215)+-(0.0,0.0734267)(15,-0.52407)+-(0.0,0.0747864)(16,-0.759474)+-(0.0,0.0760998)(17,-1.53966)+-(0.0,0.0974336)(18,-1.81564)+-(0.0,0.0993547)(19,-2.83851)+-(0.0,0.14994)(20,-3.90456)+-(0.0,0.228148)(21,-5.15001)+-(0.0,0.377563) 
};

\addplot [
color=black,
dashed
]
coordinates{
 (1,-3.99234)(2,-3.99234)(3,-3.99234)(4,-3.99234)(5,-3.99234)(6,-3.99234)(7,-3.99234)(8,-3.99234)(9,-3.99234)(10,-3.99234)(11,-3.99234)(12,-3.99234)(13,-3.99234)(14,-3.99234)(15,-3.99234)(16,-3.99234)(17,-3.99234)(18,-3.99234)(19,-3.99234)(20,-3.99234)(21,-3.99234) 
};

\end{axis}
\end{tikzpicture}
 \end{subfigure}
\begin{subfigure}
 \centering
\pgfplotsset{width=surADMMDistEx1.tikz\columnwidth,height=surADMMDistEx1.tikz\columnwidth,compat=newest,plot coordinates/math parser=false}
%
%
\begin{tikzpicture}

\begin{axis}[%
scale only axis,
width=3cm,
height=4cm,
xmin=1, xmax=80,
ymin=-15, ymax=5,
title={(b)},
xlabel={iteration in primal-dual method},
ylabel={$\frac{1}{N}\sum_{i=1}^N\textrm{log}\left(\hat{\eta}^{i}\right)$},
axis on top]
\addplot [
color=black,
solid
]
plot [error bars/.cd, y dir = both, y explicit]
coordinates{
 (1,3.34841)+-(0.0,0.0260389)(2,3.24694)+-(0.0,0.0262391)(3,3.12835)+-(0.0,0.0265299)(4,2.98444)+-(0.0,0.026976)(5,2.80644)+-(0.0,0.0276393)(6,2.58603)+-(0.0,0.0285877)(7,2.32608)+-(0.0,0.0299795)(8,2.14245)+-(0.0,0.0308511)(9,1.95611)+-(0.0,0.0320564)(10,1.79874)+-(0.0,0.033087)(11,1.65388)+-(0.0,0.034011)(12,1.54017)+-(0.0,0.034554)(13,1.51647)+-(0.0,0.0345703)(14,1.50311)+-(0.0,0.0345738)(15,1.48976)+-(0.0,0.0345772)(16,1.47612)+-(0.0,0.0345808)(17,1.46207)+-(0.0,0.0345845)(18,1.44772)+-(0.0,0.0345883)(19,1.43308)+-(0.0,0.0345923)(20,1.41799)+-(0.0,0.0345965)(21,1.40258)+-(0.0,0.0346008)(22,1.38669)+-(0.0,0.0346053)(23,1.37047)+-(0.0,0.0346099)(24,1.35375)+-(0.0,0.0346148)(25,1.33651)+-(0.0,0.0346199)(26,1.31891)+-(0.0,0.0346251)(27,1.30076)+-(0.0,0.0346305)(28,1.28205)+-(0.0,0.0346362)(29,1.26275)+-(0.0,0.0346421)(30,1.24305)+-(0.0,0.0346481)(31,1.22252)+-(0.0,0.0346545)(32,1.20157)+-(0.0,0.0346609)(33,1.17974)+-(0.0,0.0346677)(34,1.15722)+-(0.0,0.0346747)(35,1.13399)+-(0.0,0.0346819)(36,1.10979)+-(0.0,0.0346895)(37,1.08482)+-(0.0,0.0346973)(38,1.05881)+-(0.0,0.0347055)(39,1.03197)+-(0.0,0.0347138)(40,1.0037)+-(0.0,0.0347228)(41,0.974542)+-(0.0,0.034732)(42,0.943828)+-(0.0,0.0347419)(43,0.911809)+-(0.0,0.0347523)(44,0.878426)+-(0.0,0.0347632)(45,0.843253)+-(0.0,0.0347752)(46,0.805802)+-(0.0,0.0347887)(47,0.766745)+-(0.0,0.0348032)(48,0.724705)+-(0.0,0.0348201)(49,0.680386)+-(0.0,0.034839)(50,0.632659)+-(0.0,0.0348612)(51,0.581246)+-(0.0,0.0348875)(52,0.525842)+-(0.0,0.0349187)(53,0.464818)+-(0.0,0.0349575)(54,0.397558)+-(0.0,0.0350058)(55,0.322551)+-(0.0,0.0350672)(56,0.237884)+-(0.0,0.0351471)(57,0.138892)+-(0.0,0.0352584)(58,0.0201057)+-(0.0,0.0354207)(59,-0.130119)+-(0.0,0.0356827)(60,-0.342725)+-(0.0,0.0362167)(61,-0.776815)+-(0.0,0.0387727)(62,-1.44851)+-(0.0,0.0461906)(63,-1.50967)+-(0.0,0.0462105)(64,-1.57218)+-(0.0,0.0462298)(65,-1.63961)+-(0.0,0.0462528)(66,-1.71239)+-(0.0,0.0462798)(67,-1.79099)+-(0.0,0.0463113)(68,-1.8759)+-(0.0,0.0463477)(69,-1.96874)+-(0.0,0.0463912)(70,-2.06807)+-(0.0,0.046439)(71,-2.1756)+-(0.0,0.0464934)(72,-2.29347)+-(0.0,0.0465574)(73,-2.42282)+-(0.0,0.0466323)(74,-2.56331)+-(0.0,0.0467167)(75,-2.72168)+-(0.0,0.0468225)(76,-2.90298)+-(0.0,0.0469593)(77,-3.11967)+-(0.0,0.0471557)(78,-3.40385)+-(0.0,0.0475095)(79,-3.87158)+-(0.0,0.0486375)(80,-4.65487)+-(0.0,0.0530492) 
};

\addplot [
color=black,
dashed
]
coordinates{
 (1,-3.99234)(2,-3.99234)(3,-3.99234)(4,-3.99234)(5,-3.99234)(6,-3.99234)(7,-3.99234)(8,-3.99234)(9,-3.99234)(10,-3.99234)(11,-3.99234)(12,-3.99234)(13,-3.99234)(14,-3.99234)(15,-3.99234)(16,-3.99234)(17,-3.99234)(18,-3.99234)(19,-3.99234)(20,-3.99234)(21,-3.99234)(22,-3.99234)(23,-3.99234)(24,-3.99234)(25,-3.99234)(26,-3.99234)(27,-3.99234)(28,-3.99234)(29,-3.99234)(30,-3.99234)(31,-3.99234)(32,-3.99234)(33,-3.99234)(34,-3.99234)(35,-3.99234)(36,-3.99234)(37,-3.99234)(38,-3.99234)(39,-3.99234)(40,-3.99234)(41,-3.99234)(42,-3.99234)(43,-3.99234)(44,-3.99234)(45,-3.99234)(46,-3.99234)(47,-3.99234)(48,-3.99234)(49,-3.99234)(50,-3.99234)(51,-3.99234)(52,-3.99234)(53,-3.99234)(54,-3.99234)(55,-3.99234)(56,-3.99234)(57,-3.99234)(58,-3.99234)(59,-3.99234)(60,-3.99234)(61,-3.99234)(62,-3.99234)(63,-3.99234)(64,-3.99234)(65,-3.99234)(66,-3.99234)(67,-3.99234)(68,-3.99234)(69,-3.99234)(70,-3.99234)(71,-3.99234)(72,-3.99234)(73,-3.99234)(74,-3.99234)(75,-3.99234)(76,-3.99234)(77,-3.99234)(78,-3.99234)(79,-3.99234)(80,-3.99234) 
};

\end{axis}
\end{tikzpicture}
 \end{subfigure}
\begin{subfigure}
 \centering
\pgfplotsset{width=surADMMFirstEx1.tikz\columnwidth,height=surADMMFirstEx1.tikz\columnwidth,compat=newest,plot coordinates/math parser=false}\input{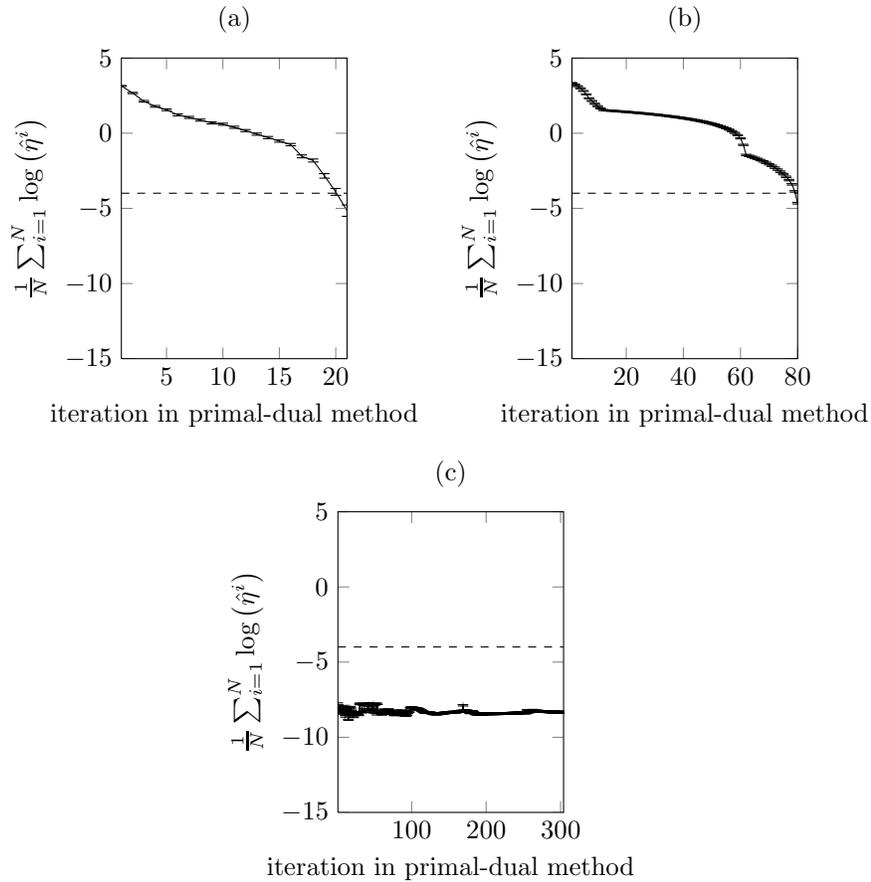}
 \end{subfigure}
\caption{Local surrogate duality gap.  The averaged, over all subproblems, value of the surrogate duality gap along with its standard deviation \eqref{solid}  and upper bound, $\textrm{log}\left(\epsilon/N\right)$, \eqref{dashed} are shown. In figures (a), (b) and (c), algorithms \ref{alg:algNPDexact}, \ref{alg:algNPD} and ADMM have been used, respectively.} \label{fig:surADMMEx1}\end{center}
\end{figure}
\begin{figure}[h!]
\begin{center}
\begin{subfigure}
 \centering
\pgfplotsset{width=globsurADMMOrigEx1.tikz\columnwidth,height=globsurADMMOrigEx1.tikz\columnwidth,compat=newest,plot coordinates/math parser=false}
%
%
\begin{tikzpicture}

\begin{axis}[%
scale only axis,
width=3cm,
height=4cm,
xmin=1, xmax=21,
ymin=-4, ymax=5.5,
title={(a)},
xlabel={iteration in primal dual method},
ylabel={$\textrm{log}\left(\hat{\eta}\right)$},
axis on top]
\addplot [
color=black,
solid
]
coordinates{
 (1,4.84522)(2,4.37469)(3,3.8267)(4,3.48769)(5,3.2597)(6,2.92223)(7,2.75472)(8,2.57547)(9,2.3898)(10,2.30168)(11,2.09758)(12,1.87594)(13,1.64103)(14,1.41372)(15,1.18109)(16,0.945897)(17,0.169724)(18,-0.105825)(19,-1.11411)(20,-2.13754)(21,-3.20396) 
};

\addplot [
color=black,
dashed
]
coordinates{
 (1,-2.29337)(2,-2.29337)(3,-2.29337)(4,-2.29337)(5,-2.29337)(6,-2.29337)(7,-2.29337)(8,-2.29337)(9,-2.29337)(10,-2.29337)(11,-2.29337)(12,-2.29337)(13,-2.29337)(14,-2.29337)(15,-2.29337)(16,-2.29337)(17,-2.29337)(18,-2.29337)(19,-2.29337)(20,-2.29337)(21,-2.29337)(22,-2.29337)(23,-2.29337)(24,-2.29337)(25,-2.29337)(26,-2.29337)(27,-2.29337)(28,-2.29337)(29,-2.29337)(30,-2.29337)(31,-2.29337)(32,-2.29337)(33,-2.29337)(34,-2.29337)(35,-2.29337)(36,-2.29337)(37,-2.29337)(38,-2.29337)(39,-2.29337)(40,-2.29337)(41,-2.29337)(42,-2.29337)(43,-2.29337)(44,-2.29337)(45,-2.29337)(46,-2.29337)(47,-2.29337)(48,-2.29337)(49,-2.29337)(50,-2.29337)(51,-2.29337)(52,-2.29337)(53,-2.29337)(54,-2.29337)(55,-2.29337)(56,-2.29337)(57,-2.29337)(58,-2.29337)(59,-2.29337)(60,-2.29337)(61,-2.29337)(62,-2.29337)(63,-2.29337)(64,-2.29337)(65,-2.29337)(66,-2.29337)(67,-2.29337)(68,-2.29337)(69,-2.29337)(70,-2.29337)(71,-2.29337)(72,-2.29337)(73,-2.29337)(74,-2.29337)(75,-2.29337)(76,-2.29337)(77,-2.29337)(78,-2.29337)(79,-2.29337)(80,-2.29337)(81,-2.29337)(82,-2.29337)(83,-2.29337)(84,-2.29337)(85,-2.29337)(86,-2.29337)(87,-2.29337)(88,-2.29337)(89,-2.29337)(90,-2.29337)(91,-2.29337)(92,-2.29337)(93,-2.29337)(94,-2.29337)(95,-2.29337)(96,-2.29337)(97,-2.29337)(98,-2.29337)(99,-2.29337)(100,-2.29337)(101,-2.29337)(102,-2.29337)(103,-2.29337)(104,-2.29337)(105,-2.29337)(106,-2.29337)(107,-2.29337)(108,-2.29337)(109,-2.29337)(110,-2.29337)(111,-2.29337)(112,-2.29337)(113,-2.29337)(114,-2.29337)(115,-2.29337)(116,-2.29337)(117,-2.29337)(118,-2.29337)(119,-2.29337)(120,-2.29337)(121,-2.29337)(122,-2.29337)(123,-2.29337)(124,-2.29337)(125,-2.29337)(126,-2.29337)(127,-2.29337)(128,-2.29337)(129,-2.29337)(130,-2.29337)(131,-2.29337)(132,-2.29337)(133,-2.29337)(134,-2.29337)(135,-2.29337)(136,-2.29337)(137,-2.29337)(138,-2.29337)(139,-2.29337)(140,-2.29337)(141,-2.29337)(142,-2.29337)(143,-2.29337)(144,-2.29337)(145,-2.29337)(146,-2.29337)(147,-2.29337)(148,-2.29337)(149,-2.29337)(150,-2.29337)(151,-2.29337)(152,-2.29337)(153,-2.29337)(154,-2.29337)(155,-2.29337)(156,-2.29337)(157,-2.29337)(158,-2.29337)(159,-2.29337)(160,-2.29337)(161,-2.29337)(162,-2.29337)(163,-2.29337)(164,-2.29337)(165,-2.29337)(166,-2.29337)(167,-2.29337)(168,-2.29337)(169,-2.29337)(170,-2.29337)(171,-2.29337)(172,-2.29337)(173,-2.29337)(174,-2.29337)(175,-2.29337)(176,-2.29337)(177,-2.29337)(178,-2.29337)(179,-2.29337)(180,-2.29337)(181,-2.29337)(182,-2.29337)(183,-2.29337)(184,-2.29337)(185,-2.29337)(186,-2.29337)(187,-2.29337)(188,-2.29337)(189,-2.29337)(190,-2.29337)(191,-2.29337)(192,-2.29337)(193,-2.29337)(194,-2.29337)(195,-2.29337)(196,-2.29337)(197,-2.29337)(198,-2.29337)(199,-2.29337)(200,-2.29337)(201,-2.29337)(202,-2.29337)(203,-2.29337)(204,-2.29337)(205,-2.29337)(206,-2.29337)(207,-2.29337)(208,-2.29337)(209,-2.29337)(210,-2.29337)(211,-2.29337)(212,-2.29337)(213,-2.29337)(214,-2.29337)(215,-2.29337)(216,-2.29337)(217,-2.29337)(218,-2.29337)(219,-2.29337)(220,-2.29337)(221,-2.29337)(222,-2.29337)(223,-2.29337)(224,-2.29337)(225,-2.29337)(226,-2.29337)(227,-2.29337)(228,-2.29337)(229,-2.29337)(230,-2.29337)(231,-2.29337)(232,-2.29337)(233,-2.29337)(234,-2.29337)(235,-2.29337)(236,-2.29337)(237,-2.29337)(238,-2.29337)(239,-2.29337)(240,-2.29337)(241,-2.29337)(242,-2.29337)(243,-2.29337)(244,-2.29337)(245,-2.29337)(246,-2.29337)(247,-2.29337)(248,-2.29337)(249,-2.29337)(250,-2.29337)(251,-2.29337)(252,-2.29337)(253,-2.29337)(254,-2.29337)(255,-2.29337)(256,-2.29337)(257,-2.29337)(258,-2.29337)(259,-2.29337)(260,-2.29337)(261,-2.29337)(262,-2.29337)(263,-2.29337)(264,-2.29337)(265,-2.29337)(266,-2.29337)(267,-2.29337)(268,-2.29337)(269,-2.29337)(270,-2.29337)(271,-2.29337)(272,-2.29337)(273,-2.29337)(274,-2.29337)(275,-2.29337)(276,-2.29337)(277,-2.29337)(278,-2.29337)(279,-2.29337)(280,-2.29337)(281,-2.29337)(282,-2.29337)(283,-2.29337)(284,-2.29337)(285,-2.29337)(286,-2.29337)(287,-2.29337)(288,-2.29337)(289,-2.29337)(290,-2.29337)(291,-2.29337)(292,-2.29337)(293,-2.29337)(294,-2.29337)(295,-2.29337)(296,-2.29337)(297,-2.29337)(298,-2.29337)(299,-2.29337)(300,-2.29337)(301,-2.29337)(302,-2.29337)(303,-2.29337)(304,-2.29337) 
};

\end{axis}
\end{tikzpicture}
 \end{subfigure}
\begin{subfigure}
 \centering
\pgfplotsset{width=globsurADMMDistEx1.tikz\columnwidth,height=globsurADMMDistEx1.tikz\columnwidth,compat=newest,plot coordinates/math parser=false}
%
%
\begin{tikzpicture}

\begin{axis}[%
scale only axis,
width=3cm,
height=4cm,
xmin=1, xmax=80,
ymin=-4, ymax=5.5,
title={(b)},
xlabel={iteration in primal dual method},
ylabel={$\textrm{log}\left(\hat{\eta}\right)$},
axis on top]
\addplot [
color=black,
solid
]
coordinates{
 (1,5.04814)(2,4.94669)(3,4.82811)(4,4.68423)(5,4.50627)(6,4.28592)(7,4.02606)(8,3.8425)(9,3.65623)(10,3.49894)(11,3.35415)(12,3.24048)(13,3.21678)(14,3.20342)(15,3.19007)(16,3.17644)(17,3.16238)(18,3.14804)(19,3.1334)(20,3.1183)(21,3.10289)(22,3.08701)(23,3.07079)(24,3.05407)(25,3.03683)(26,3.01923)(27,3.00108)(28,2.98237)(29,2.96307)(30,2.94337)(31,2.92284)(32,2.90189)(33,2.88005)(34,2.85754)(35,2.83431)(36,2.81011)(37,2.78515)(38,2.75913)(39,2.73229)(40,2.70403)(41,2.67487)(42,2.64415)(43,2.61213)(44,2.57875)(45,2.54358)(46,2.50613)(47,2.46707)(48,2.42504)(49,2.38072)(50,2.33299)(51,2.28158)(52,2.22618)(53,2.16516)(54,2.0979)(55,2.0229)(56,1.93824)(57,1.83926)(58,1.72048)(59,1.57028)(60,1.35772)(61,0.923839)(62,0.252849)(63,0.191683)(64,0.12918)(65,0.06175)(66,-0.0110276)(67,-0.0896173)(68,-0.174531)(69,-0.267366)(70,-0.366683)(71,-0.474211)(72,-0.592073)(73,-0.721421)(74,-0.861893)(75,-1.02026)(76,-1.20155)(77,-1.41821)(78,-1.70235)(79,-2.16995)(80,-2.95267) 
};

\addplot [
color=black,
dashed
]
coordinates{
 (1,-2.29337)(2,-2.29337)(3,-2.29337)(4,-2.29337)(5,-2.29337)(6,-2.29337)(7,-2.29337)(8,-2.29337)(9,-2.29337)(10,-2.29337)(11,-2.29337)(12,-2.29337)(13,-2.29337)(14,-2.29337)(15,-2.29337)(16,-2.29337)(17,-2.29337)(18,-2.29337)(19,-2.29337)(20,-2.29337)(21,-2.29337)(22,-2.29337)(23,-2.29337)(24,-2.29337)(25,-2.29337)(26,-2.29337)(27,-2.29337)(28,-2.29337)(29,-2.29337)(30,-2.29337)(31,-2.29337)(32,-2.29337)(33,-2.29337)(34,-2.29337)(35,-2.29337)(36,-2.29337)(37,-2.29337)(38,-2.29337)(39,-2.29337)(40,-2.29337)(41,-2.29337)(42,-2.29337)(43,-2.29337)(44,-2.29337)(45,-2.29337)(46,-2.29337)(47,-2.29337)(48,-2.29337)(49,-2.29337)(50,-2.29337)(51,-2.29337)(52,-2.29337)(53,-2.29337)(54,-2.29337)(55,-2.29337)(56,-2.29337)(57,-2.29337)(58,-2.29337)(59,-2.29337)(60,-2.29337)(61,-2.29337)(62,-2.29337)(63,-2.29337)(64,-2.29337)(65,-2.29337)(66,-2.29337)(67,-2.29337)(68,-2.29337)(69,-2.29337)(70,-2.29337)(71,-2.29337)(72,-2.29337)(73,-2.29337)(74,-2.29337)(75,-2.29337)(76,-2.29337)(77,-2.29337)(78,-2.29337)(79,-2.29337)(80,-2.29337) 
};

\end{axis}
\end{tikzpicture}
 \end{subfigure}
\begin{subfigure}
 \centering
\pgfplotsset{width=globsurADMMOnTopEx1.tikz\columnwidth,height=globsurADMMOnTopEx1.tikz\columnwidth,compat=newest,plot coordinates/math parser=false}
%
%
\begin{tikzpicture}

\begin{axis}[%
scale only axis,
width=3cm,
height=4cm,
xmin=1, xmax=304,
ymin=-8, ymax=-2,
title={(c)},
xlabel={iteration in primal-dual method},
ylabel={$\textrm{log}\left(\hat{\eta}\right)$},axis on top]
\addplot [
color=black,
solid
]
coordinates{
 (1,-6.07126)(2,-6.4433)(3,-6.60885)(4,-6.28106)(5,-6.54521)(6,-6.72591)(7,-6.55668)(8,-6.21829)(9,-6.37176)(10,-6.41141)(11,-6.90882)(12,-6.71658)(13,-6.48527)(14,-6.71865)(15,-7.14021)(16,-6.35234)(17,-6.6406)(18,-6.29075)(19,-6.46797)(20,-6.82821)(21,-6.80727)(22,-6.93534)(23,-6.63324)(24,-6.64207)(25,-6.75056)(26,-6.70397)(27,-6.75102)(28,-6.71113)(29,-6.81232)(30,-6.07821)(31,-6.08504)(32,-6.62624)(33,-6.6259)(34,-6.08603)(35,-6.08749)(36,-6.1507)(37,-6.44657)(38,-6.46944)(39,-6.47064)(40,-6.6054)(41,-6.16431)(42,-6.44659)(43,-6.04985)(44,-6.18955)(45,-6.11441)(46,-6.41409)(47,-6.46106)(48,-6.12957)(49,-6.64579)(50,-6.1358)(51,-6.08837)(52,-6.08076)(53,-6.76702)(54,-6.07386)(55,-6.59924)(56,-6.65344)(57,-6.6302)(58,-6.54748)(59,-6.59862)(60,-6.76818)(61,-6.72205)(62,-6.57866)(63,-6.45799)(64,-6.50263)(65,-6.53372)(66,-6.56369)(67,-6.60548)(68,-6.59415)(69,-6.44898)(70,-6.50406)(71,-6.60812)(72,-6.6615)(73,-6.69384)(74,-6.71707)(75,-6.73662)(76,-6.76617)(77,-6.80507)(78,-6.79378)(79,-6.76986)(80,-6.6871)(81,-6.52867)(82,-6.60898)(83,-6.63085)(84,-6.6707)(85,-6.63651)(86,-6.56866)(87,-6.5283)(88,-6.74522)(89,-6.78681)(90,-6.81284)(91,-6.83263)(92,-6.83612)(93,-6.8184)(94,-6.77529)(95,-6.65636)(96,-6.58232)(97,-6.58183)(98,-6.37881)(99,-6.32042)(100,-6.32809)(101,-6.33753)(102,-6.34662)(103,-6.35475)(104,-6.3622)(105,-6.36909)(106,-6.37742)(107,-6.38638)(108,-6.39812)(109,-6.41225)(110,-6.43714)(111,-6.4634)(112,-6.48981)(113,-6.51644)(114,-6.54585)(115,-6.57488)(116,-6.60352)(117,-6.64711)(118,-6.65261)(119,-6.65207)(120,-6.65138)(121,-6.65116)(122,-6.65371)(123,-6.65769)(124,-6.66551)(125,-6.67393)(126,-6.68469)(127,-6.69812)(128,-6.71438)(129,-6.73443)(130,-6.74415)(131,-6.74479)(132,-6.75007)(133,-6.75503)(134,-6.75137)(135,-6.74472)(136,-6.7378)(137,-6.73097)(138,-6.72404)(139,-6.71712)(140,-6.71024)(141,-6.70346)(142,-6.69695)(143,-6.6908)(144,-6.6848)(145,-6.67905)(146,-6.67365)(147,-6.66882)(148,-6.66411)(149,-6.65736)(150,-6.64961)(151,-6.6434)(152,-6.63906)(153,-6.63562)(154,-6.6327)(155,-6.6255)(156,-6.62076)(157,-6.62145)(158,-6.62592)(159,-6.63499)(160,-6.62516)(161,-6.60166)(162,-6.58506)(163,-6.57956)(164,-6.56903)(165,-6.57806)(166,-6.56278)(167,-6.54722)(168,-6.56911)(169,-6.1744)(170,-6.5473)(171,-6.52529)(172,-6.52568)(173,-6.54127)(174,-6.5357)(175,-6.54093)(176,-6.58724)(177,-6.62703)(178,-6.66462)(179,-6.67819)(180,-6.65988)(181,-6.58514)(182,-6.568)(183,-6.62861)(184,-6.66959)(185,-6.72085)(186,-6.74217)(187,-6.74166)(188,-6.74288)(189,-6.74592)(190,-6.74995)(191,-6.7471)(192,-6.74944)(193,-6.75443)(194,-6.74788)(195,-6.74283)(196,-6.73926)(197,-6.73672)(198,-6.73527)(199,-6.73476)(200,-6.73589)(201,-6.73589)(202,-6.73624)(203,-6.73698)(204,-6.73458)(205,-6.73403)(206,-6.73363)(207,-6.73159)(208,-6.73319)(209,-6.7332)(210,-6.73324)(211,-6.73363)(212,-6.73363)(213,-6.73407)(214,-6.73443)(215,-6.73473)(216,-6.73389)(217,-6.73162)(218,-6.72951)(219,-6.72738)(220,-6.72542)(221,-6.7228)(222,-6.72136)(223,-6.71941)(224,-6.71748)(225,-6.71568)(226,-6.71378)(227,-6.71192)(228,-6.7101)(229,-6.70833)(230,-6.70639)(231,-6.70468)(232,-6.70275)(233,-6.70086)(234,-6.69888)(235,-6.69692)(236,-6.69503)(237,-6.69318)(238,-6.69124)(239,-6.68907)(240,-6.68742)(241,-6.68564)(242,-6.68409)(243,-6.68277)(244,-6.68236)(245,-6.68104)(246,-6.68367)(247,-6.68272)(248,-6.677)(249,-6.6721)(250,-6.66891)(251,-6.66967)(252,-6.67969)(253,-6.70379)(254,-6.67082)(255,-6.49915)(256,-6.61431)(257,-6.63945)(258,-6.58756)(259,-6.5634)(260,-6.55073)(261,-6.54449)(262,-6.54206)(263,-6.4913)(264,-6.54313)(265,-6.54527)(266,-6.54789)(267,-6.55101)(268,-6.55454)(269,-6.55821)(270,-6.56181)(271,-6.56538)(272,-6.56895)(273,-6.5724)(274,-6.5763)(275,-6.57963)(276,-6.58302)(277,-6.5862)(278,-6.58954)(279,-6.59267)(280,-6.59556)(281,-6.59837)(282,-6.601)(283,-6.60372)(284,-6.60627)(285,-6.60861)(286,-6.61092)(287,-6.61307)(288,-6.61507)(289,-6.61699)(290,-6.61885)(291,-6.6205)(292,-6.6223)(293,-6.62392)(294,-6.62569)(295,-6.62701)(296,-6.62833)(297,-6.62956)(298,-6.63066)(299,-6.63196)(300,-6.63311)(301,-6.6341)(302,-6.63526)(303,-6.63596)(304,-6.63701) 
};

\addplot [
color=black,
dashed
]
coordinates{
 (1,-2.29337)(2,-2.29337)(3,-2.29337)(4,-2.29337)(5,-2.29337)(6,-2.29337)(7,-2.29337)(8,-2.29337)(9,-2.29337)(10,-2.29337)(11,-2.29337)(12,-2.29337)(13,-2.29337)(14,-2.29337)(15,-2.29337)(16,-2.29337)(17,-2.29337)(18,-2.29337)(19,-2.29337)(20,-2.29337)(21,-2.29337)(22,-2.29337)(23,-2.29337)(24,-2.29337)(25,-2.29337)(26,-2.29337)(27,-2.29337)(28,-2.29337)(29,-2.29337)(30,-2.29337)(31,-2.29337)(32,-2.29337)(33,-2.29337)(34,-2.29337)(35,-2.29337)(36,-2.29337)(37,-2.29337)(38,-2.29337)(39,-2.29337)(40,-2.29337)(41,-2.29337)(42,-2.29337)(43,-2.29337)(44,-2.29337)(45,-2.29337)(46,-2.29337)(47,-2.29337)(48,-2.29337)(49,-2.29337)(50,-2.29337)(51,-2.29337)(52,-2.29337)(53,-2.29337)(54,-2.29337)(55,-2.29337)(56,-2.29337)(57,-2.29337)(58,-2.29337)(59,-2.29337)(60,-2.29337)(61,-2.29337)(62,-2.29337)(63,-2.29337)(64,-2.29337)(65,-2.29337)(66,-2.29337)(67,-2.29337)(68,-2.29337)(69,-2.29337)(70,-2.29337)(71,-2.29337)(72,-2.29337)(73,-2.29337)(74,-2.29337)(75,-2.29337)(76,-2.29337)(77,-2.29337)(78,-2.29337)(79,-2.29337)(80,-2.29337)(81,-2.29337)(82,-2.29337)(83,-2.29337)(84,-2.29337)(85,-2.29337)(86,-2.29337)(87,-2.29337)(88,-2.29337)(89,-2.29337)(90,-2.29337)(91,-2.29337)(92,-2.29337)(93,-2.29337)(94,-2.29337)(95,-2.29337)(96,-2.29337)(97,-2.29337)(98,-2.29337)(99,-2.29337)(100,-2.29337)(101,-2.29337)(102,-2.29337)(103,-2.29337)(104,-2.29337)(105,-2.29337)(106,-2.29337)(107,-2.29337)(108,-2.29337)(109,-2.29337)(110,-2.29337)(111,-2.29337)(112,-2.29337)(113,-2.29337)(114,-2.29337)(115,-2.29337)(116,-2.29337)(117,-2.29337)(118,-2.29337)(119,-2.29337)(120,-2.29337)(121,-2.29337)(122,-2.29337)(123,-2.29337)(124,-2.29337)(125,-2.29337)(126,-2.29337)(127,-2.29337)(128,-2.29337)(129,-2.29337)(130,-2.29337)(131,-2.29337)(132,-2.29337)(133,-2.29337)(134,-2.29337)(135,-2.29337)(136,-2.29337)(137,-2.29337)(138,-2.29337)(139,-2.29337)(140,-2.29337)(141,-2.29337)(142,-2.29337)(143,-2.29337)(144,-2.29337)(145,-2.29337)(146,-2.29337)(147,-2.29337)(148,-2.29337)(149,-2.29337)(150,-2.29337)(151,-2.29337)(152,-2.29337)(153,-2.29337)(154,-2.29337)(155,-2.29337)(156,-2.29337)(157,-2.29337)(158,-2.29337)(159,-2.29337)(160,-2.29337)(161,-2.29337)(162,-2.29337)(163,-2.29337)(164,-2.29337)(165,-2.29337)(166,-2.29337)(167,-2.29337)(168,-2.29337)(169,-2.29337)(170,-2.29337)(171,-2.29337)(172,-2.29337)(173,-2.29337)(174,-2.29337)(175,-2.29337)(176,-2.29337)(177,-2.29337)(178,-2.29337)(179,-2.29337)(180,-2.29337)(181,-2.29337)(182,-2.29337)(183,-2.29337)(184,-2.29337)(185,-2.29337)(186,-2.29337)(187,-2.29337)(188,-2.29337)(189,-2.29337)(190,-2.29337)(191,-2.29337)(192,-2.29337)(193,-2.29337)(194,-2.29337)(195,-2.29337)(196,-2.29337)(197,-2.29337)(198,-2.29337)(199,-2.29337)(200,-2.29337)(201,-2.29337)(202,-2.29337)(203,-2.29337)(204,-2.29337)(205,-2.29337)(206,-2.29337)(207,-2.29337)(208,-2.29337)(209,-2.29337)(210,-2.29337)(211,-2.29337)(212,-2.29337)(213,-2.29337)(214,-2.29337)(215,-2.29337)(216,-2.29337)(217,-2.29337)(218,-2.29337)(219,-2.29337)(220,-2.29337)(221,-2.29337)(222,-2.29337)(223,-2.29337)(224,-2.29337)(225,-2.29337)(226,-2.29337)(227,-2.29337)(228,-2.29337)(229,-2.29337)(230,-2.29337)(231,-2.29337)(232,-2.29337)(233,-2.29337)(234,-2.29337)(235,-2.29337)(236,-2.29337)(237,-2.29337)(238,-2.29337)(239,-2.29337)(240,-2.29337)(241,-2.29337)(242,-2.29337)(243,-2.29337)(244,-2.29337)(245,-2.29337)(246,-2.29337)(247,-2.29337)(248,-2.29337)(249,-2.29337)(250,-2.29337)(251,-2.29337)(252,-2.29337)(253,-2.29337)(254,-2.29337)(255,-2.29337)(256,-2.29337)(257,-2.29337)(258,-2.29337)(259,-2.29337)(260,-2.29337)(261,-2.29337)(262,-2.29337)(263,-2.29337)(264,-2.29337)(265,-2.29337)(266,-2.29337)(267,-2.29337)(268,-2.29337)(269,-2.29337)(270,-2.29337)(271,-2.29337)(272,-2.29337)(273,-2.29337)(274,-2.29337)(275,-2.29337)(276,-2.29337)(277,-2.29337)(278,-2.29337)(279,-2.29337)(280,-2.29337)(281,-2.29337)(282,-2.29337)(283,-2.29337)(284,-2.29337)(285,-2.29337)(286,-2.29337)(287,-2.29337)(288,-2.29337)(289,-2.29337)(290,-2.29337)(291,-2.29337)(292,-2.29337)(293,-2.29337)(294,-2.29337)(295,-2.29337)(296,-2.29337)(297,-2.29337)(298,-2.29337)(299,-2.29337)(300,-2.29337)(301,-2.29337)(302,-2.29337)(303,-2.29337)(304,-2.29337) 
};

\end{axis}
\end{tikzpicture}
 \end{subfigure}
\caption{Global surrogate duality gap.  The value of the surrogate duality gap \eqref{solid} and its upper bound, $\textrm{log}\left(\epsilon/N\right)$, \eqref{dashed} are shown. In figures (a), (b) and (c), algorithms \ref{alg:algNPDexact}, \ref{alg:algNPD} and ADMM have been used, respectively.} \label{fig:globsurADMMEx1}\end{center}
\end{figure}
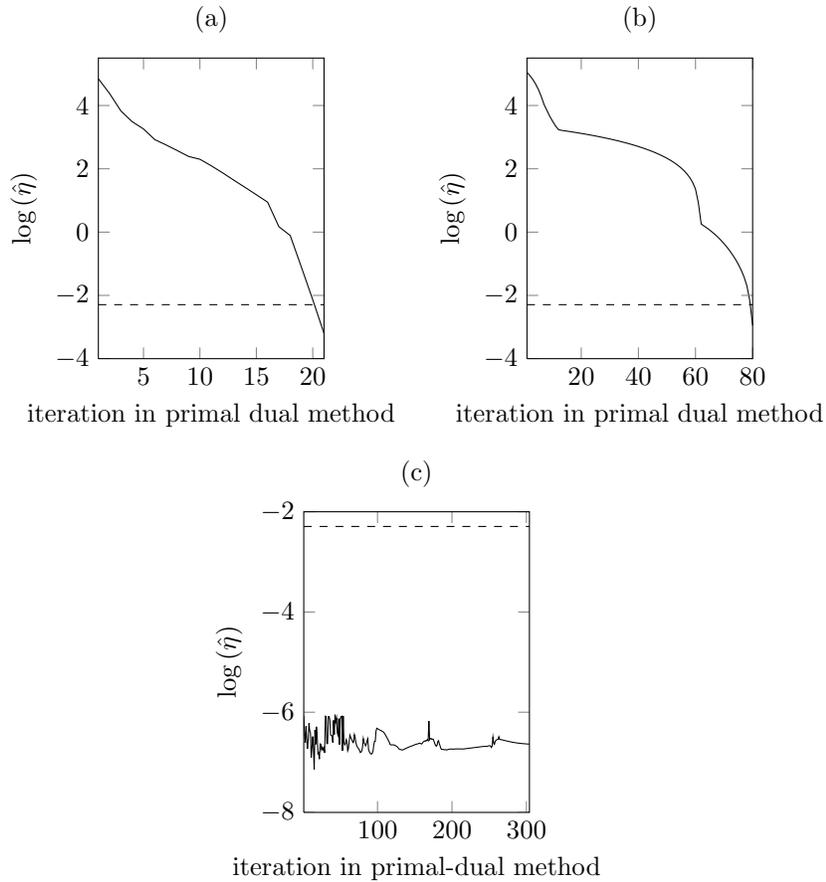

\subsubsection{Optimal Value}
We compare how the value of the objective function evolves for each iteration of the three methods. In Figure \ref{fig:optvalADMMEx1}, the relative error of the objective function is shown. The so-called true optimal value is obtained using \texttt{cvx}. We see that the relative error evolves to approximately the same level for all three methods. However, ADMM gives a slightly higher value (the relative error is $1.2\times 10^{-5}$) than Algorithm \ref{alg:algNPDexact} (the relative error is $8.4\times 10^{-8}$) and Algorithm \ref{alg:algNPD} (the relative error is $6.3\times 10^{-8}$). Note that since the iterates are not necessarily feasible in each iteration, it is possible to obtain a value of the objective function that is very close to or even smaller than the true optimal value, while the algorithm has not yet terminated. This explains the dip of the relative error which can be seen in Figure \ref{fig:optvalADMMEx1}. 

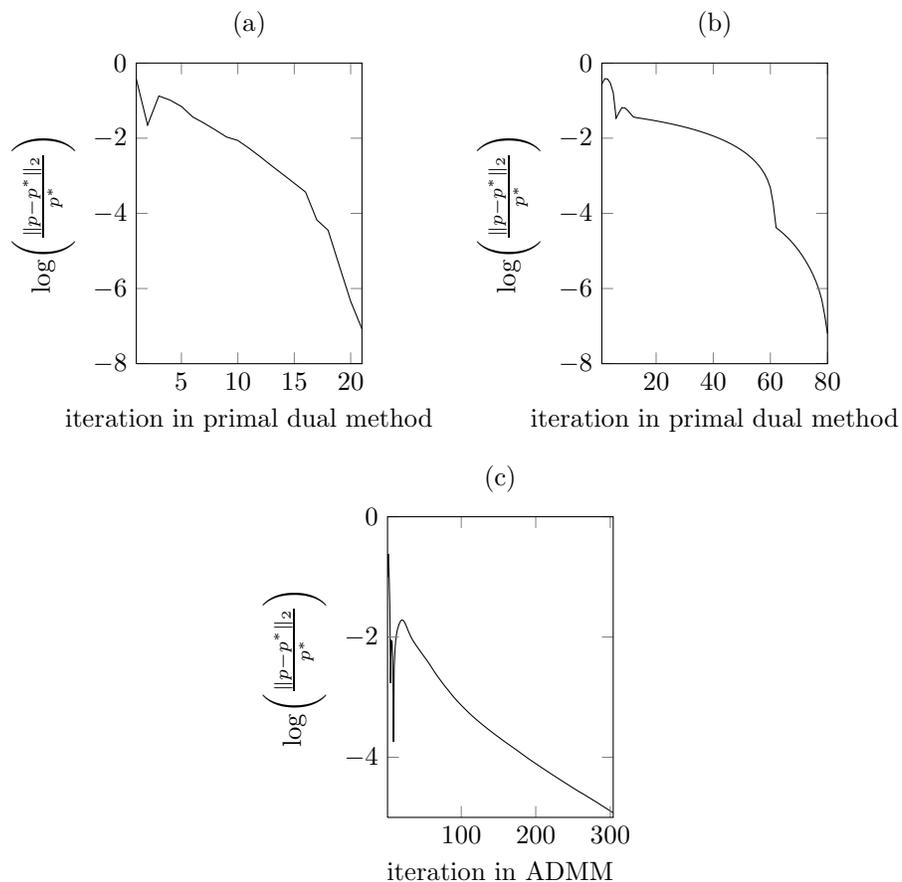
\begin{figure}[h!]
\begin{center}
\begin{subfigure}
 \centering
\pgfplotsset{width=optValADMMOrigEx1.tikz\columnwidth,height=optValADMMOrigEx1.tikz\columnwidth,compat=newest,plot coordinates/math parser=false}
%
%
\begin{tikzpicture}

\begin{axis}[%
scale only axis,
width=3cm,
height=4cm,
xmin=1, xmax=21,
ymin=-8, ymax=0,
title={(a)},
xlabel={iteration in primal dual method},
ylabel={$\text{log}\bigg(\frac{\|p-p^*\|_2}{p^*}\bigg)$},
axis on top]
\addplot [
color=black,
solid
]
coordinates{
 (1,-0.427671)(2,-1.64849)(3,-0.875814)(4,-0.987254)(5,-1.15543)(6,-1.43379)(7,-1.59828)(8,-1.77856)(9,-1.96859)(10,-2.05989)(11,-2.26882)(12,-2.49555)(13,-2.73458)(14,-2.96425)(15,-3.19777)(16,-3.43289)(17,-4.1746)(18,-4.44921)(19,-5.40126)(20,-6.34149)(21,-7.07533) 
};

\end{axis}
\end{tikzpicture}
 \end{subfigure}
\begin{subfigure}
 \centering
\pgfplotsset{width=optValADMMDistEx1.tikz\columnwidth,height=optValADMMDistEx1.tikz\columnwidth,compat=newest,plot coordinates/math parser=false}
%
%
\begin{tikzpicture}

\begin{axis}[%
scale only axis,
width=3cm,
height=4cm,
xmin=1, xmax=80,
ymin=-8, ymax=0,
title={(b)},
xlabel={iteration in primal dual method},
ylabel={$\text{log}\bigg(\frac{\|p-p^*\|_2}{p^*}\bigg)$},
axis on top]
\addplot [
color=black,
solid
]
coordinates{
 (1,-0.566091)(2,-0.417642)(3,-0.422982)(4,-0.53589)(5,-0.792623)(6,-1.46665)(7,-1.32128)(8,-1.18711)(9,-1.19232)(10,-1.25735)(11,-1.34599)(12,-1.43121)(13,-1.45153)(14,-1.4632)(15,-1.47494)(16,-1.487)(17,-1.49952)(18,-1.51237)(19,-1.52558)(20,-1.53928)(21,-1.55335)(22,-1.56795)(23,-1.58294)(24,-1.59849)(25,-1.61462)(26,-1.63119)(27,-1.64837)(28,-1.66619)(29,-1.68467)(30,-1.70364)(31,-1.72351)(32,-1.74391)(33,-1.76528)(34,-1.78743)(35,-1.81039)(36,-1.83444)(37,-1.85936)(38,-1.88545)(39,-1.91248)(40,-1.94106)(41,-1.97067)(42,-2.00196)(43,-2.03469)(44,-2.06892)(45,-2.10508)(46,-2.14368)(47,-2.18402)(48,-2.22753)(49,-2.27346)(50,-2.32298)(51,-2.37637)(52,-2.43391)(53,-2.49725)(54,-2.56699)(55,-2.6446)(56,-2.73195)(57,-2.83365)(58,-2.955)(59,-3.10729)(60,-3.32034)(61,-3.74275)(62,-4.37879)(63,-4.44025)(64,-4.50293)(65,-4.57049)(66,-4.64336)(67,-4.72197)(68,-4.80683)(69,-4.89949)(70,-4.99849)(71,-5.10552)(72,-5.22261)(73,-5.35079)(74,-5.48957)(75,-5.6453)(76,-5.82231)(77,-6.03129)(78,-6.29852)(79,-6.70605)(80,-7.19929) 
};

\end{axis}
\end{tikzpicture}
 \end{subfigure}
\begin{subfigure}
 \centering
\pgfplotsset{width=optValADMMEx1.tikz\columnwidth,height=optValADMMEx1.tikz\columnwidth,compat=newest,plot coordinates/math parser=false}
%
%
\begin{tikzpicture}

\begin{axis}[%
scale only axis,
width=3cm,
height=4cm,
xmin=1, xmax=304,
ymin=-5, ymax=-0,title={(c)},
xlabel={iteration in ADMM},
ylabel={$\text{log}\bigg(\frac{\|p-p^*\|_2}{p^*}\bigg)$},
axis on top]
\addplot [
color=black,
solid
]
coordinates{
 (1,-1.00653)(2,-0.61733)(3,-0.924474)(4,-1.43698)(5,-2.766)(6,-2.05966)(7,-2.07531)(8,-2.36593)(9,-3.74239)(10,-2.49421)(11,-2.21165)(12,-2.05415)(13,-1.95859)(14,-1.89295)(15,-1.84536)(16,-1.80606)(17,-1.77315)(18,-1.74515)(19,-1.72776)(20,-1.71893)(21,-1.71797)(22,-1.72439)(23,-1.73764)(24,-1.75785)(25,-1.78316)(26,-1.81136)(27,-1.84142)(28,-1.87247)(29,-1.90347)(30,-1.93293)(31,-1.96085)(32,-1.98712)(33,-2.01165)(34,-2.03446)(35,-2.05568)(36,-2.07548)(37,-2.09456)(38,-2.11335)(39,-2.13175)(40,-2.14977)(41,-2.16745)(42,-2.18478)(43,-2.20215)(44,-2.21934)(45,-2.23671)(46,-2.25432)(47,-2.27198)(48,-2.2895)(49,-2.30623)(50,-2.32305)(51,-2.33999)(52,-2.35702)(53,-2.37402)(54,-2.3911)(55,-2.40855)(56,-2.42665)(57,-2.44536)(58,-2.46426)(59,-2.48341)(60,-2.50266)(61,-2.522)(62,-2.54125)(63,-2.56029)(64,-2.57908)(65,-2.5974)(66,-2.61543)(67,-2.6332)(68,-2.65071)(69,-2.66795)(70,-2.68499)(71,-2.70193)(72,-2.71879)(73,-2.73554)(74,-2.75216)(75,-2.76862)(76,-2.78488)(77,-2.80092)(78,-2.81673)(79,-2.83231)(80,-2.84766)(81,-2.86281)(82,-2.87752)(83,-2.8927)(84,-2.9084)(85,-2.92437)(86,-2.94031)(87,-2.95587)(88,-2.97098)(89,-2.98577)(90,-3.00031)(91,-3.01464)(92,-3.02875)(93,-3.04266)(94,-3.05636)(95,-3.06983)(96,-3.08312)(97,-3.09624)(98,-3.10922)(99,-3.12209)(100,-3.13487)(101,-3.14756)(102,-3.16017)(103,-3.17272)(104,-3.18518)(105,-3.19757)(106,-3.20987)(107,-3.22207)(108,-3.23416)(109,-3.24614)(110,-3.25801)(111,-3.26976)(112,-3.28138)(113,-3.29289)(114,-3.30427)(115,-3.31554)(116,-3.32669)(117,-3.33774)(118,-3.34869)(119,-3.35954)(120,-3.3703)(121,-3.38096)(122,-3.39155)(123,-3.40205)(124,-3.41247)(125,-3.42283)(126,-3.43311)(127,-3.44333)(128,-3.45349)(129,-3.46358)(130,-3.47363)(131,-3.48361)(132,-3.49354)(133,-3.50342)(134,-3.51325)(135,-3.52302)(136,-3.53275)(137,-3.54243)(138,-3.55206)(139,-3.56166)(140,-3.57121)(141,-3.58072)(142,-3.59019)(143,-3.59962)(144,-3.60902)(145,-3.61837)(146,-3.62769)(147,-3.63698)(148,-3.64623)(149,-3.65544)(150,-3.66462)(151,-3.67376)(152,-3.68288)(153,-3.69196)(154,-3.70102)(155,-3.71005)(156,-3.71906)(157,-3.72805)(158,-3.73701)(159,-3.74595)(160,-3.75487)(161,-3.76375)(162,-3.77261)(163,-3.78146)(164,-3.79028)(165,-3.79903)(166,-3.80772)(167,-3.81639)(168,-3.82511)(169,-3.83365)(170,-3.8425)(171,-3.85127)(172,-3.86004)(173,-3.86886)(174,-3.8777)(175,-3.88655)(176,-3.89542)(177,-3.90432)(178,-3.91328)(179,-3.9223)(180,-3.93135)(181,-3.9404)(182,-3.94948)(183,-3.95857)(184,-3.96767)(185,-3.97675)(186,-3.9858)(187,-3.99476)(188,-4.00364)(189,-4.01243)(190,-4.02114)(191,-4.02977)(192,-4.03835)(193,-4.04689)(194,-4.05539)(195,-4.06385)(196,-4.07228)(197,-4.0807)(198,-4.08909)(199,-4.09747)(200,-4.10585)(201,-4.11423)(202,-4.12259)(203,-4.13095)(204,-4.13929)(205,-4.14762)(206,-4.15593)(207,-4.16423)(208,-4.17252)(209,-4.18079)(210,-4.18905)(211,-4.19729)(212,-4.20552)(213,-4.21373)(214,-4.22192)(215,-4.2301)(216,-4.23826)(217,-4.24641)(218,-4.25455)(219,-4.26267)(220,-4.27078)(221,-4.27887)(222,-4.28695)(223,-4.29501)(224,-4.30307)(225,-4.31111)(226,-4.31914)(227,-4.32715)(228,-4.33516)(229,-4.34315)(230,-4.35113)(231,-4.3591)(232,-4.36706)(233,-4.37502)(234,-4.38296)(235,-4.39089)(236,-4.39881)(237,-4.40672)(238,-4.41463)(239,-4.42252)(240,-4.4304)(241,-4.43828)(242,-4.44615)(243,-4.45401)(244,-4.46186)(245,-4.4697)(246,-4.47753)(247,-4.48535)(248,-4.49317)(249,-4.50097)(250,-4.50878)(251,-4.51659)(252,-4.52442)(253,-4.53231)(254,-4.54006)(255,-4.54785)(256,-4.55574)(257,-4.56335)(258,-4.57053)(259,-4.57759)(260,-4.58464)(261,-4.59175)(262,-4.59893)(263,-4.60615)(264,-4.61347)(265,-4.6209)(266,-4.62836)(267,-4.63586)(268,-4.64337)(269,-4.6509)(270,-4.65845)(271,-4.66599)(272,-4.67355)(273,-4.68113)(274,-4.68872)(275,-4.69634)(276,-4.70398)(277,-4.71164)(278,-4.71933)(279,-4.72704)(280,-4.73478)(281,-4.74254)(282,-4.75032)(283,-4.75813)(284,-4.76596)(285,-4.77381)(286,-4.78167)(287,-4.78955)(288,-4.79744)(289,-4.80534)(290,-4.81325)(291,-4.82116)(292,-4.82907)(293,-4.83699)(294,-4.84491)(295,-4.85282)(296,-4.86073)(297,-4.86863)(298,-4.87654)(299,-4.88444)(300,-4.89233)(301,-4.90023)(302,-4.90811)(303,-4.916)(304,-4.92388) 
};

\end{axis}
\end{tikzpicture}
 \end{subfigure}
\caption{Relative error of objective function.  The relative error of the objective function \eqref{solid} is shown. In figures (a), (b) and (c), algorithms \ref{alg:algNPDexact}, \ref{alg:algNPD} and ADMM have been used, respectively. The true optimal value is obtained using \texttt{cvx} and is denoted $p^*$. The value of the objective function is denoted $p$.} \label{fig:optvalADMMEx1}\end{center}
\end{figure}

\subsubsection{Total Number of Iterations}
The total number of iterations, that is ADMM iterations for Algorithm \ref{alg:algNPDexact} and Algorithm \ref{alg:algNPD}, and interior-point iterations in \texttt{cvx} for ADMM are compared. We get the total number of iterations equal to 38796, 17453 and 8194, for each method respectively. Thus, for this specific simulation set-up, ADMM beats the proposed algorithms in terms of the number of iterations. Note, however, that Algorithm \ref{alg:algNPDexact} and Algorithm \ref{alg:algNPD} are computationally very cheap. The major computational effort takes place when calculating the search directions of the local primal variables, that is Step 9 in Algorithm \eqref{alg:alg4PD}, see Remark 3. If $\rho$ is kept constant, we only have to factorize $H^{i,(l)}_{\text{pd}} + \rho\left( I + (\bar A^i)^T\bar A^i\right)$ for $i = 1, \dots, N,$ once in each primal-dual iteration. However, if we use ADMM  (Algorithm \ref{alg:alg2}) we have to factorize the corresponding matrices, which are of the same sizes as in the other algorithms, for each interior-point iteration in each ADMM iteration. That is, instead of factorizing the matrices 21 or 80 times with Algorithm \ref{alg:algNPDexact} or Algorithm \ref{alg:algNPD}, respectively, we have to factorize the matrices 8194 times with ADMM. We pay a price for the savings in the number of factorizations necessary. In Algorithm \ref{alg:algNPDexact} and Algorithm \ref{alg:algNPD} the nodes have to communicate in each inner iteration (ADMM iteration) whilst for ADMM applied to the original problem the nodes only have to communicate in each outer iteration (ADMM iteration). 

The convergence rate of Algorithm \ref{alg:algNPD} is sensitive to the setting of its parameters. How to choose them optimally, or even wisely, is an open question. Note though, that we can make use of an ad-hoc adaptive stop criteria for the search direction calculation in Algorithm \ref{alg:algNPDexact} which does not guarantee global convergence. Such ad-hoc criteria, can yield much faster convergence rate without much tuning of the settings. For example, with $\epsilon_{pri}=2\epsilon_{dual}=50/4\times10^{-3}$ for $l\leq 5$, $\epsilon_{pri}=2\epsilon_{dual}=50/4\times10^{-5}$ for $6\leq l\leq10$ and $\epsilon_{pri}=2\epsilon_{dual}=50/4\times10^{-8}$ for $l\geq 11$, we get 5469 number of iterations, a 33\% saving compared to ADMM. This result provides an incitement to investigate further the convergence rate properties of Algorithm \ref{alg:algNPD}.

\subsection{Simulation Set-Up 2}
We consider ten subproblems ($N=10$). The rest of the problem set-up coincides with that of Section \ref{setup1}.

The optimization problem is solved using Algorithm \ref{alg:algNPD} and ADMM (Algorithm \ref{alg:alg2}). For comparison, both algorithms are terminated using the stop criteria of  Algorithm \ref{alg:algNPDexact}. We initialize the methods and use the same settings as in Section \ref{setup1}. We perform 50 Monte-Carlo runs of the simulation set-up.

\subsubsection{Total Number of Iterations}
The total number of iterations, that is ADMM iterations for Algorithm \ref{alg:algNPD} and inner-point iterations in \texttt{cvx} for ADMM, are compared. We get the averaged, over all Monte-Carlo runs, total number of iterations equal to 569 (the standard deviation is 62) and 1257 (the standard deviation is 29), for each method respectively. Thus, for this specific simulation set-up, Algorithm \ref{alg:algNPD} beats ADMM in terms of the number of iterations. In fact, Algorithm \ref{alg:algNPD} terminates after 55\% less iterations than ADMM (with respect to the averaged value).

\section{Conclusion}\label{sec:C}

We have proposed two efficient distributed primal-dual interior-point method for loosely coupled problems using ADMM (Algorithm \ref{alg:algNPDexact} and Algorithm \ref{alg:algNPD}). Due to the nature of the interior-point method, the loosely coupled structure of the problem is preserved in the linear system of equations that provides the primal-dual directions. ADMM takes advantage of this structure and makes the direction calculations highly parallellizable. Consequently, the proposed method has superior computational properties with respect to other distributed algorithms. Of course, we can use Algorithm \ref{alg:algNPDexact} and Algorithm \ref{alg:algNPD} on problems with completely coupled structure as well, but we can not expect the same superior properties as for the loosely coupled structure.

The latter of the methods (Algorithm \ref{alg:algNPD}) adaptively chooses the required accuracy in the termination condition of the inner iterations (ADMM iterations) with respect to the accuracy obtained in the outer iterations (interior-point method iterations). This is to avoid unnecessary inner iterations when the accuracy of the current outer iteration is low, as elaborated in \cite{bon:05}. We have stated under which assumptions the method converges to the optimal solution. In addition, we have provided comparisons between Algorithm~ \ref{alg:algNPDexact}, Algorithm \ref{alg:algNPD} and ADMM in simulation.
\vspace{20pt}
\section*{Funding}
This work was partially supported by the European Research Council under the European Community’s Seventh Framework Programme (FP7/2007–2013) / ERC Grant Agreement No. 267381, the Swedish Research Council and the Linnaeus Center ACCESS at KTH.

\appendix

\section{Global Convergence of Distributed Inexact Primal-Dual Interior-Point Method}\label{app:global}
Here we state the proof for global convergence of the proposed distributed primal-dual inexact interior-point method. To this end, we use the definitions of $\Omega^i(\epsilon)$ for $i=1,\dots,N$ and assumptions B1--B4. The lemmas, theorems and proofs are adapted with very minor changes from \cite{eis:94} and \cite{bel:98}, more detailed references are given before each lemma and theorem. They are all included here for the sake of completeness.

The collection of lemmas and theorems in this appendix enables us to show that, through the run of the algorithm, the iterates are persistently updated and guaranteed to converge to the optimal solution. We first address the concept of \emph{break down} of the algorithm. That is, when we are not able to find a suitable step direction and step size \cite{eis:94}. We show that such a break down of the algorithm will not occur under our assumptions as defined in Section \ref{distConRes}. Second, we show that the algorithm is convergent towards an optimal solution. 

Particularly, Theorem \ref{notZero} and  Lemma \ref{breakdown} assure a persistent update of the iterates. The theorem illustrates that the upper bound of the step size is bounded away from zero. Lemma \ref{breakdown} states that, given a suitable search direction, it is always possible to find a step size which yields a satisfactory decrease in the merit function $\|H^i(z^i)\|$ for $i=1,\dots,N$.

The subsequent theorems are related to the convergence of the algorithm, that is, that the generated $z^{(l)}$ converges and $\|H(z^{(l)})\|$ converges to zero. Particularly, Theorem \eqref{conLimitPointTheorem} states that the sequence of iterates  $\{z^{(l)}\}$ generated by our method converges to a point $z_*$, and Theorem \eqref{conLimitPointTheoremPre} is a result necessary for the proof of Theorem \eqref{conLimitPointTheorem}.  Finally, in Theorem \eqref{globalCon}, we state under which assumptions the proposed algorithm is global convergent, that is, generates a sequence of iterates such that $\|H(z^{(l)})\|\rightarrow 0$. Theorem \eqref{divergent} is used in the proof of Theorem \eqref{globalCon}.
\vspace{20pt}
\subsection{Break Down}
There are four steps in the algorithm where a possible break down can occur. They are when calculating the step direction (Step 9 in Algorithm \ref{alg:algNPD}), choosing the intermediate step sizes $\alpha_1^i$ and $\alpha_2^i$ (steps 10-13 in Algorithm \ref{alg:algNPD}) and calculating the actual step size $\alpha$ (steps 10-13 in Algorithm \ref{alg:algNPD}). The hypothesis under which our proposed algorithm\emph{ does not} break down is addressed in the rest of this section.

A step direction can always be calculated provided $H'(z)$ is invertible. Hence, if assumption B3 is fulfilled, the algorithm will not break down at Step 9.  Next,  Theorem \ref{notZero} shows that the intermediate step sizes for any such step direction is always bounded away from zero. This theorem is based on Theorem 3.2 in \cite{bel:98}.
\vspace{10pt}
\begin{theorem}\label{notZero}
Assume $\{z^{(l)}\}$ is generated by Algorithm \ref{alg:algNPD} and assumptions B1-- B4 are fulfilled, then the sequence $\{\alpha^{(l)}\}$ with $\alpha^{(l)}=\underset{i}{\text{min}}(\alpha^{i,(l)})$ and $\alpha^{i,(l)}=\text{min}(\alpha^{i}_1,\alpha^{i}_2)$ is bounded away from zero.
\end{theorem}
\vspace{10pt}
\begin{proof}
The proof follows closely that of Theorem 3.2 in \cite{bel:98}. We first show that the sum of the complementary KKT condition is bounded away from zero in $\Omega^i(\epsilon)$. First, note that $(s^{i})^T\lambda^{i}\geq\|S^{i}\Lambda^{i}\hat{e}\|$ and $(s^{i})^T\lambda^{i}\geq\bar{\tau}_2^i\gamma^{i}\|R^i(z^{i})\|$. Using the quadratic mean we get
\begin{align*}
(s^{i})^T\lambda^{i}&\geq \sqrt{\frac{\|S^{i}\Lambda^{i}\hat{e}\|^2+(\bar{\tau}_2^i\gamma^{i}\|R^i(z^{i})\|)^2}{2}}\\
&\geq \textrm{min}\left(1,\frac{\bar{\tau}_2^i}{2}\right)\|H^i(z^{i,(l)})\|\frac{1}{\sqrt{2}}\\
&\geq\textrm{min}\left(1,\frac{\bar{\tau}_2^i}{2}\right)\frac{\epsilon}{N}\frac{1}{\sqrt{2}}.
\end{align*}
The right hand side of the last inequality is independent of the iteration $l$ and strictly larger than zero, thus $(s^{i})^T\lambda^{i}$ is bounded away from zero in $\Omega^i(\epsilon)$ for $i=1,\dots,N$. Due to assumptions B1, B3 and the fact that $\|\hat{r}^{(l)}\|$ is bounded, we have that $\|\Delta z^{i}\|$ is bounded and there exists $M_1^{i}>0$ and $M_2^i>0$ such that

$$|\Delta s_j^i\Delta \lambda_j^i-\frac{\tau_1^i\gamma^i}{m_i}(\Delta s^i)^T\Delta \lambda^i|\leq M_1^i, \textrm{ for all } i=1,\dots,N$$

\noindent and

$$|(\Delta s^i)^T\Delta \lambda^i-\tau_2^i\gamma^iL^i\|\Delta z^{i,(l)}\|^2|\leq M_2^i, \textrm{ for all } i=1,\dots,N,$$

\noindent respectively. Furthermore, we have $\hat{r}^{(l)}=(\hat{r}_1^{(l)},0)$, where the zero corresponds to the complementary KKT condition which is solved exactly in Algorithm \ref{alg:algNPD}.

To determine $\alpha^{i}_1$ at iteration $l$, we consider the expression (where we omit the superscript $l$)
\begin{align*}
s_j^{i}(\alpha)\lambda_j^{i}(\alpha)-\frac{\bar{\tau}_1^i\gamma^{i}}{m_i}(s^{i}(\alpha))^T\lambda^{i}(\alpha)&=\\
(s_j^{i}+\alpha \Delta s_j^{i})(\lambda_j^{i}+\alpha \Delta\lambda_j^{i})-\frac{\bar{\tau}_1^i\gamma^{i}}{m_i}(s^{i}+\alpha \Delta s^{i})^T(\lambda^{i}+\alpha \Delta\lambda^{i})&=\\
\left(\!s_j^{i}\lambda_j^{i}\!-\!\frac{\bar{\tau}_1^i\gamma^{i}}{m_i}(s^{i})^T\lambda^{i}\!\right)\!+\!\alpha\left(\!\Delta\lambda_j^{i} s_j^{i}+ \Delta s_j^{i}\lambda_j^{i}\! -\!\frac{\bar{\tau}_1^i\gamma^{i}}{m_i}(s^{i})^T\Delta\lambda^{i}\! -\!\frac{\bar{\tau}_1^i\gamma^{i}}{m_i}(\lambda^{i})^T\Delta s^{i}\!\right)&+\\
\alpha^2\left(\Delta s_j^{i}\Delta\lambda_j^{i}-\frac{\bar{\tau}_1^i\gamma^{i}}{m_i}(\Delta s^{i})^T\Delta\lambda^{i}\right)&=\\
\left(\!s_j^{i}\lambda_j^{i}\!-\!\frac{\bar{\tau}_1^i\gamma^{i}}{m_i}(s^{i})^T\lambda^{i}\!\right)+\alpha\left(\!-s_j^{i}\lambda_j^{i}+\mu-\frac{\bar{\tau}_1^i\gamma^{i}}{m_i}\hat{e}^T(-S^{i}\Lambda^{i}\hat{e}+\mu\hat{e})\right)&+\\
\alpha^2\left(\Delta s_j^{i}\Delta\lambda_j^{i}-\frac{\bar{\tau}_1^i\gamma^{i}}{m_i}(\Delta s^{i})^T\Delta\lambda^{i}\right)&=\\
(1-\alpha)\!\left(\!s_j^{i}\!\lambda_j^{i}\!-\!\frac{\bar{\tau}_1^i\gamma^{i}}{m_i}(s^{i})^T\lambda^{i}\!\right)\!+\alpha(1-\bar{\tau}_1^i\gamma^{i})\mu\!+\!\alpha^2\!\left(\!\Delta s_j^i\Delta \lambda_j^i-\frac{\tau_1^i\gamma^i}{m_i}(\Delta s^i)^T\Delta \lambda^i\!\right)\!&\geq\\
\alpha(1-\bar{\tau}_1^i\gamma^{i})\mu+\alpha^2\left(\Delta s_j^i\Delta \lambda_j^i-\frac{\tau_1^i\gamma^i}{m_i}(\Delta s^i)^T\Delta \lambda^i\right)&\geq\\
\alpha(1-\bar{\tau}_1^i\gamma^{i})\mu-\alpha^2\left |\Delta s_j^i\Delta \lambda_j^i-\frac{\tau_1^i\gamma^i}{m_i}(\Delta s^i)^T\Delta \lambda^i\right |&\geq\\
\alpha(1-\bar{\tau}_1^i\gamma^{i})\mu-\alpha^2M_1^i.
\end{align*}
Since $\alpha_1^i$ is defined as
\begin{align*}
\alpha_1^i=\underset{\alpha \in [0,1]}{\textrm{max}}\left\{\alpha|f_1(\alpha')\geq 0, \textrm{ for all }\alpha'\leq\alpha\right\},
\end{align*}
with
\begin{align*}f_1(\alpha)=\textrm{min}\left(S^{i,(l)}(\alpha)\Lambda^{i,(l)}(\alpha)\hat{e}\right)-\bar{\tau}_1^i\gamma^{i,(l)}\left(s^{i,(l)}(\alpha)\right)^T\lambda^{i,(l)}(\alpha)/m_i,\end{align*}
we see that
$$\alpha^{i}_1\geq (1-\bar{\tau}_1^i\gamma^{i})\mu\frac{1}{M^i_1}>0.$$
We have that $\mu$ is bounded away from zero (since both $\sigma$ and $(s^i)^T\lambda^i$ are bounded away from zero) and $\bar{\tau}_1^i\gamma^{i}$ is bounded away from one (since  $\bar{\tau}_1^i$ is at most one and $\gamma^{i}$ is bounded away from one). Hence, $\alpha^{i}_1$ is bounded away from zero in in $\Omega^i(\epsilon)$ for $i=1,\dots,N$.

To determine $\alpha^{i}_2$ at iteration $l$, we consider the expression (where we once again omit the superscript $l$)
$$(s^{i}(\alpha))^T\lambda^{i}(\alpha)-\bar{\tau}_2^i\gamma^{i}\|R^i(\alpha)\|.$$
In analogy to page 117 in \cite{bel:98}, we use the mean value theorem and the Lipschitz continuity of $(R^{i})'$. The mean value theorem gives
\begin{align*}
R^i(\alpha)&=R^i(z^{i})+\alpha (R^{i})'(z^{i})\Delta z^{i}+\alpha\left(\int_{0}^1 \left(  (R^{i})'(z^{i}+t\alpha\Delta z^{i})-(R^{i})'(z^{i}) \right)\Delta z^{i}\right)\\
&=R^i(z^{i})(1-\alpha)+\alpha\hat{r}^i_1+\alpha\left(\int_{0}^1 \left(  (R^{i})'(z^{i}+t\alpha\Delta z^{i})-(R^{i})'(z^{i}) \right)\Delta z^{i}\right),
\end{align*}
where $\hat{r}^i$ denotes the contributing elements of the $i$th term in $\|\hat{r}\|$, see \eqref{residualhat}.The Lipschitz continuity of $(R^{i})'$ gives
\begin{align}
\label{LipCon}
\begin{split}
\|R^i(\alpha)\|&\leq \|R^i(z^{i})\|(1-\alpha)+\alpha\|\hat{r}^i_1\|+\alpha^2L^i\|\Delta z^{i}\|^2.
\end{split}
\end{align}
From inequality \eqref{LipCon} and the stop criteria of the step direction calculation Algorithm~\ref{alg:algNPD}, we get
\begin{align*}
\|R^i(\alpha)\|\leq(1-\alpha)\|R^i(z^i)\|+\alpha\hat{\eta}\frac{(s^i)^T\lambda^i}{m}+\alpha^2L^i\|\Delta z^{i}\|^2.
\end{align*}
Thus,
\begin{align*}
(s^i(\alpha))^T\lambda^i(\alpha)-\bar{\tau}_2^i\gamma^i\|R^i(\alpha)\|&\geq\\
(\!s^{i}\!+\alpha \Delta s^{i})^T(\lambda^{i}\!+\alpha \Delta\lambda^{i})\!-\!\bar{\tau}_2^i\gamma^i\left(\!(1-\alpha)\|R^i(z^i)\|\!+\!\alpha\hat{\eta}\frac{(s^i)^T\lambda^i}{m}\!+\!\alpha^2L^i\|\Delta z^{i}\|^2\right)\!&=\\
(s^{i})^T\lambda^{i}+\alpha(s^{i})^T\Delta\lambda^{i}+\alpha(\Delta s^{i})^T\lambda^{i}+\alpha^2(\Delta s^{i})^T\Delta\lambda^{i}&-\\
\bar{\tau}_2^i\gamma^i\left((1-\alpha)\|R^i(z^i)\|+\alpha\hat{\eta}\frac{(s^i)^T\lambda^i}{m}+\alpha^2L^i\|\Delta z^{i}\|^2\right)&=\\
(1-\alpha)\left((s^i)^T\lambda^i-\bar{\tau}_2^i\gamma^i\|R^i(z^i)\|\right)+\alpha \left(m_i\mu-\frac{\bar{\tau}_2^i\gamma^i\hat{\eta}(s^i)^T\lambda^i}{m}\right)&+\\
\alpha^2\left((\Delta s^i)^T\Delta \lambda^i-\tau_2^i\gamma^iL^i\|\Delta z^{i,(l)}\|^2\right)&=\\
(1-\alpha)\left((s^i)^T\lambda^i-\bar{\tau}_2^i\gamma^i\|R^i(z^i)\|\right)+\alpha \left(m_i\mu-\frac{\bar{\tau}_2^i\gamma^i\hat{\eta}(s^i)^T\lambda^i}{m}\right)&+\\
\alpha^2\left((\Delta s^i)^T\Delta \lambda^i-\tau_2^i\gamma^iL^i\|\Delta z^{i,(l)}\|^2\right)&\geq\\
(1-\alpha)\left((s^i)^T\lambda^i-\bar{\tau}_2^i\gamma^i\|R^i(z^i)\|\right)+\alpha \left(\mu-\frac{\bar{\tau}_2^i\gamma^i\hat{\eta}(s^i)^T\lambda^i}{m}\right)&-\\
\alpha^2\left|(\Delta s^i)^T\Delta \lambda^i-\tau_2^i\gamma^iL^i\|\Delta z^{i,(l)}\|^2\right|&\geq\\
(1-\alpha)\!\left((s^i)^T\!\lambda^i\!\!-\!\bar{\tau}_2^i\gamma^i\|R^i(z^i)\|\right)\!+\!\alpha\! \left(\!\frac{\sigma}{m}\underset{i}{\textrm{min}}\left((s^i)^T\!\lambda^i\right)\!-\!\frac{\bar{\tau}_2^i\gamma^i\hat{\eta}(s^i)^T\!\lambda^i}{m}\right)\!-\!\alpha^2M_2^i&\geq\\
\alpha \left(\frac{\sigma}{m}\underset{i}{\textrm{min}}\left((s^i)^T\lambda^i\right)-\frac{\bar{\tau}_2^i\gamma^i\hat{\eta}(s^i)^T\lambda^i}{m}\right)-\alpha^2M_2^i.
\end{align*}
Since $\alpha_2^i$ is defined as
\begin{align*}
\alpha_2^i=\underset{\alpha \in [0,1]}{\textrm{max}}\left\{\alpha|f_2(\alpha')\geq 0, \textrm{ for all }\alpha'\leq\alpha\right\},
\end{align*}
with
$$f_2(\alpha)=\left(s^{i,(l)}(\alpha)\right)^T\lambda^{i,(l)}(\alpha)-\bar{\tau}_2^i\gamma^{i,(l)}\|R^i(\alpha)\|$$
and $\sigma^i>\bar{\tau}_2^i\gamma^{i}\hat{\eta}^i(s^i)^T\lambda^i/\underset{i}{\textrm{min}}\left((s^i)^T\lambda^i\right)+\epsilon_{\sigma}$,
we see that
$$\alpha^{i}_2\geq  \left(\sigma\underset{i}{\textrm{min}}\left((s^i)^T\lambda^i\right)-\bar{\tau}_2^i\gamma^i\hat{\eta}(s^i)^T\lambda^i\right)\frac{1}{mM_2^i}>0.$$
Hence, $\alpha^{i}_2$ is bounded away from zero in $\Omega^i(\epsilon)$ for $i=1,\dots,N$.
\end{proof}
Thus, we can always find a step size bounded away from zero. We continue by showing that, given a step direction $\tilde{p}^{i}$ calculated in accordance to Algorithm \ref{alg:algNPD}, there is a step $p^{i}$ such that 
\begin{align*}
\| H^i(z^{i}+p^{i}) \|\leq (1-\beta(1-\eta^{i}))\| H^i(z^{i}) \|.
\end{align*}
given $z^{i}$ and $\beta \in (0,1)$. That is, there exists an actual step size $\alpha$ such that the inequality above is fulfilled. To do this we state Lemma \eqref{breakdown} which is based on Lemma 3.1 in \cite{eis:94}.

\begin{lemma}\label{breakdown}
Given $z^{i}$ and $\beta \in (0,1)$, assume that $\tilde{p}^{i}$ is calculated in accordance to Algorithm \ref{alg:algNPD}. Then
\begin{align*}
\| H^i(z^{i})+(H^{i})'(z^{i})\tilde{p}^{i} \|<  \| H^i(z^{i}) \|,
\end{align*}
and there exists an $\eta_{\textrm{min}}^i\in [0,1)$ such that, for any $\eta^{i} \in [\eta_{\textrm{min}}^i,1)$, we can find a $p^{i}$ satisfying
\begin{align*}
\| H^i(z^{i}+p^{i}) \|\leq (1-\beta(1-\eta^{i}))\| H^i(z^{i}) \|.
\end{align*}
\end{lemma}
\begin{proof}
Constraint $\| H^i(z^{i})+(H^{i})'(z^{i})\tilde{p}^{i} \|< \| H^i(z^{i}) \|$ implies that $\|H^i(z^{i})\|\neq0$ and $\tilde{p}^{i}\neq0$. However, note that if $\|H^i(z^{i})\|=0$ we would have terminated the algorithm in the previous iteration. We start by showing the first inequality. The step direction calculations are terminated when the $i$th term of $\|\hat{r}\|$ (denoted $\|\hat{r}^i\|$) is less than $\hat{\eta}(s^i)^T\lambda^i/m$, see \eqref{residualhat}, \eqref{epsADMMInexact1} and \eqref{epsADMMInexact2}. Consequently, we get
\begin{align*}
\| H^i(z^{i})+(H^{i})'(z^{i})\tilde{p}^{i} \| &\leq \left\|\mu\begin{bmatrix}0 \\ 0\\ 0  \\\hat{e} \end{bmatrix}+\hat{r}^i\right\|\\
&\leq \sqrt{m_i}\mu+\left\|\hat{r}^i\right\|\\
&\leq\frac{\sqrt{m_i}\sigma}{m}\underset{i}{\textrm{min}}\left( (s^i)^T\lambda^i\right)+ \frac{\hat{\eta}}{m}\left( (s^i)^T\lambda^i\right)\\
&\leq \left(\sigma+ \hat{\eta}\right)\left( (s^i)^T\lambda^i\right)\\
&\leq\left(\sigma+ \hat{\eta}\right)\|H^i(z^i)\|\\
&=\bar{\eta}\|H^i(z^i)\|\\
&<\|H^i(z^i)\|,
\end{align*}
 since $\bar{\eta}<1$. The rest of the proof follows that of Lemma 3.1 in \cite{eis:94}. A slight modification has been made to incorporate $i=1,\dots,N$. Define
\begin{align*}
\tilde{\eta} &= \frac{\| H^i(z^{i})+(H^{i})'(z^{i})\tilde{p}^{i} \|}{\| H^i(z^{i})\|},\\
\epsilon ^i &=\frac{(1-\beta)(1-\tilde{\eta}^{i})\|H^i(z^{i})\|}{\|\tilde{p}^{i}\|},\\
\eta_{\textrm{min}}^i&=\textrm{max}\left\{ \tilde{\eta}, 1-\frac{(1-\tilde{\eta})\delta^i}{\|\tilde{p}^{i}\|}\right\},
\end{align*}
where $\delta ^i>0$ is chosen such that
\begin{align*}
\| H^i(z^{i}+p^{i})-H^i(z^{i})-(H^{i})'(z^{i})p^{i} \|\leq \epsilon^i \| p^{i} \|
\end{align*}
 whenever $\| p^{i} \|\leq \delta^i$. We set
\begin{align*}
 p^{i}= \frac{1-\eta^{i}}{1-\tilde{\eta}}  \tilde{p}^{i},
\end{align*}
for any $\eta^{i} \in [\eta_{\textrm{min}}^i,1)$.  Then
\begin{align*}
\| H^i(z^{i})+(H^{i})'(z^{i})p^{i} \|&\leq \frac{\eta^{i}-\tilde{\eta}}{1-\tilde{\eta}}  \| H^i(z^{i})\|+\frac{1-\eta^{i}}{1-\tilde{\eta}} \| H^i(z^{i})+(H^{i})'(z^{i})\tilde{p}^{i} \|\\
&=\frac{\eta^{i}-\tilde{\eta}}{1-\tilde{\eta}}  \| H^i(z^{i})\|+\frac{1-\eta^{i}}{1-\tilde{\eta}}\tilde{\eta} \| H^i(z^{i})\|\\
&=\eta^{i}\| H^i(z^{i})\|,
\end{align*}
and, due to
\begin{align*}
\| p^{i}\|= \frac{1-\eta^{i}}{1-\tilde{\eta}} \| \tilde{p}^{i}\|\leq \frac{1-\eta_{\textrm{min}}^i}{1-\tilde{\eta}} \| \tilde{p}^{i}\|\leq \delta^i,
\end{align*}
we get
\begin{align*}
\| H^i(z^{i}+p^{i}) \|&=\| H^i(z^{i}+p^{i})-H^i(z^{i})-(H^{i})'(z^{i})p^{i} +H^i(z^{i})+(H^{i})'(z^{i})p^{i} \|\\
& \leq\| H^i(z^{i}+p^{i})-H^i(z^{i})-(H^{i})'(z^{i})p^{i}\|+\|H^i(z^{i})+(H^{i})'(z^{i})p^{i} \|\\
&\leq \| H^i(z^{i}+p^{i})-H^i(z^{i})-(H^{i})'(z^{i})p^{i} \|\\
&\leq\epsilon^i\frac{1-\eta^{i}}{1-\tilde{\eta}}\| \tilde{p}^{i}\|+\eta^{i}\| H^i(z^{i})\|\\
&=(1-\beta )(1-\eta^{i})\| H^i(z^{i})\|+\eta^{i}\| H^i(z^{i})\|\\
&=(1-\beta(1-\eta^{i}))\| H^i(z^{i})\|.
\end{align*}
\end{proof}
Recall that the line search in Algorithm \ref{alg:algNPD} is initialized with $\tilde{p}^{i,(l)}=\alpha^{i,(l)}\Delta z^{i,(l)}$ and $\eta^{i,(l)}=1-\alpha^{i,(l)}(1-\bar{\eta}^{(l)})$, and evolves as $\tilde{p}^{i,(l)}=\theta\tilde{p}^{i,(l)}$ and $\eta^{i,(l)}=1-\theta(1-\eta^{i,(l)})$. Assume that $H^i(z^{i,(l)})\neq 0$ and $\|\Delta z^{i,(l)}\|\neq0$. Consequently, from the proof of Lemma~\ref{breakdown}, we conclude that the while loop terminates with
\begin{align}\label{terminate}1-\eta^{i,(l)}\geq\textrm{min}\left\{ \alpha^{i,(l)}(1-\bar{\eta}^{(l)}), \frac{\theta(1-\bar{\eta}^{(l)})\delta^{i}}{\|\Delta z^{i,(l)}\|}\right\},\end{align}
which is equivalent to the result on page 114 in \cite{bel:98} for $N=1$ and Lemma 5.1 in \cite{eis:94} for $N=1$ and an additional $\alpha$-update. To see this, we follow the reasoning made in the proof of Lemma 5.1 in \cite{eis:94}. If $\eta^{i,(l)} \in [\eta_{\textrm{min}}^{i,(l)},1)$ such that $1-\eta^{i,(l)}<\delta^i(1-\bar{\eta})/\|\Delta z^{i,(l)}\|$ we get, from the proof above, that
\begin{align*}
\| H^i(z^{i}+p^{i}) \|\leq (1-\beta(1-\eta^{i}))\| H^i(z^{i}) \|.
\end{align*}
Also, note that $1-\eta^{i,(l)}$ decreases with a factor $\theta\in(0,1)$ for each run of the while loop. If no iteration of the while loop is necessary we have $1-\eta^{i,(l)}=\alpha^{i,(l)}(1-\bar{\eta}^{(l)})$. Suppose instead that $1-\eta^{i,(l)}=\theta\delta^i(1-\bar{\eta})/\|\Delta z^{i,(l)}\|$ which is less than $\delta^i(1-\bar{\eta})/\|\Delta z^{i,(l)}\|$, then the loop terminates. Thus, \eqref{terminate} holds.

The value of $1-\eta^{i,(l)}$ is bounded away from zero (since $\alpha^{i,(l)}$ is bounded away from zero, $\bar{\eta}^{(l)}$ is bounded away from one, $\theta>0$ and independent of $l$, $\delta^i>0$ and independent of $l$, and $\|\Delta z^{i,(l)}\|$ is bounded). The fact that $1-\eta^{i,(l)}$ is bounded away from zero will be used in the proof of global convergence. So at this point we observe that the combination of Lemma \ref{breakdown} and Theorem \ref{notZero}, and the complementary discussion assure that the iterates are persistently updated  through the run of the algorithm. As a corollary, we can then state that the algorithm, provided assumptions B1-B4 are fulfilled, can only break down at some iteration point $z^{i,(l)}$ if and only if $\|H^i(z^{i,(l)})\|=0$. Notice that for convex problems the break down can only happen when we have arrived at an optimal solution, hence there will always exist a suitable search direction for updating the iterates. We will now focus on convergence properties of the algorithm.

\subsection{Convergence Properties}
We first discuss some results needed for the proof of global convergence. The following theorem is used in the proof of Theorem \ref{conLimitPointTheorem}, and is therefor included here. It is based on Theorem 3.5 in \cite{eis:94}.

\begin{theorem}\label{conLimitPointTheoremPre}
Assume that Algorithm \ref{alg:algNPD} does not break down. If $z_*$ is a limit point of $\{z^{(l)}\}$ such that there exists a $\Gamma$ independent of $l$ for which
\begin{align}\label{conLimitPoint}
\|p^{(l)} \|\leq\Gamma(1-\eta^{(l)}_*)\|H(z^{(l)})\|
\end{align}
when $z^{(l)}$ is sufficiently close to $z_*$ and $l$ is sufficiently large, then $z^{(l)}\rightarrow z_*$. Here, $(1-\eta^{(l)}_*)$ corresponds to the actual step size used by all subproblems $i=1,\dots, N$.
\end{theorem}
\begin{proof}
See the proof of Theorem 3.5 in \cite{eis:94}.
\end{proof}

We are now ready to state the following theorem, which is based on Theorem 3.1 in \cite{bel:98}.
\begin{theorem}\label{conLimitPointTheorem}
If $z_*$ is a limit point of $\{z^{(l)}\}$ such that $\|H'(z_*)\|$ is nonsingular, then the sequence $\{z^{(l)}\}$ generated by Algorithm \ref{alg:algNPD} converges to $z_*$ .
\end{theorem}
\begin{proof}
The proof follows closely that of Theorem 3.1 in \cite{bel:98}. We define $K=\|(H'(z_*))^{-1}\|$ and choose $\delta>0$ such that $(H'(z))^{-1}$ exists and $\|(H'(z))^{-1}\|\leq2K$ whenever $z\in N_{\delta}(z_*)$. The actual step used in subproblem $i$ is
\begin{align*}
 p^{i,(l)}= \underset{i}{min}\left\{1-\eta^{i,(l)}\right\}\frac{1}{1-\bar{\eta}^{(l)}}  \Delta z^{i}=\frac{1-\eta^{(l)}_*}{1-\bar{\eta}^{(l)}} \Delta z^{i},
\end{align*}
that is, we choose the minimum step size over all subproblems. Assume $z^{(l)}\in N_{\delta}(z_*)$, then
\begin{align*}
\| p^{(l)}\|&=\frac{1-\eta^{(l)}_*}{1-\bar{\eta}^{(l)}} \|\Delta z\|\\
&\leq \frac{1-\eta^{(l)}_*}{1-\bar{\eta}^{(l)}}\|(H'(z^{(l)}))^{-1}\|\|-H(z^{(l)})+\tilde{r}^{(l)}\|\\
&\leq \frac{1-\eta^{(l)}_*}{1-\bar{\eta}^{(l)}}\|(H'(z^{(l)}))^{-1}\|(\|H(z^{(l)})\|+\bar{\eta}^{(l)}\|H(z^{(l)})\|)\\
&\leq \frac{1-\eta^{(l)}_*}{1-\bar{\eta}^{(l)}}2K(1+\bar{\eta}^{(l)})\|H(z^{(l)})\|\\
&\leq \Gamma(1-\eta^{(l)}_*)\|H(z^{(l)})\|\\
\end{align*}
where $$\Gamma=2K\frac{1+\bar{\eta}_{\textrm{max}}}{1-\bar{\eta}_{\textrm{max}}}$$
and $\bar{\eta}_{\textrm{max}}$ is the maximum value that $\bar{\eta}^{(l)}$ can be set to. Thus, there exists a $\Gamma$ independent of $l$ for which
inequality \eqref{conLimitPoint} holds when $z^{(l)}$ is sufficiently close to $z_*$ and $l$ is sufficiently large. Consequently, by Theorem \ref{conLimitPointTheoremPre}, $z^{(l)}\rightarrow z_*$.
\end{proof}

The following theorem is needed in the proof of global convergence. It is based on Theorem 3.4 in \cite{eis:94}.

\begin{theorem}\label{divergent}
Assume that Algorithm \ref{alg:algNPD} does not break down. If $\sum_{l\geq 0}(1-\eta^{(l)}_*)$ is divergent then $H(z^{(l)}) \rightarrow 0$.
\end{theorem}
\begin{proof}
The proof follows closely that of Theorem 3.4 in \cite{eis:94}. From the line search in Algorithm \ref{alg:algNPD}, we get
\begin{align*}
\|H^i(z^{i,(l)})\|&\leq(1-\beta(1-\eta^{(l-1)}_*))\|H^i(z^{i,(l-1)})\|\\
&\leq \Pi_{0\leq j <l}(1-\beta(1-\eta^{(j)}_*))\|H^i(z^{i,(0)})\|\\
&\leq \textrm{exp}\left(-\beta\sum_{0\leq j <l}(1-\eta^{(j)}_*)\right)\|H^i(z^{i,(0)})\|.
\end{align*}
Thus
\begin{align*}
\|H(z^{(l)})\|\leq \textrm{exp}\left(-\beta\sum_{0\leq j <l}(1-\eta^{(j)}_*)\right)\|H(z^{(0)})\|.
\end{align*}
The divergence of $\sum_{l\geq 0}(1-\eta^{{(l)}_*})$ implies $H(z^{(l)}) \rightarrow 0$ since $\beta>0$ and $1-\eta^{(l)}_*\geq 0$.
\end{proof}

We are now ready to state the theorem of global convergence of the proposed method. It is based on Theorem 3.3 in \cite{bel:98}.

\begin{theorem}\label{globalCon}
Assume $\{z^{(l)}\}$ is generated by Algorithm \ref{alg:algNPD} and assumptions B1-B4 are fulfilled, then $\{\|H(z^{(l)})\|\}$ converges to zero.
\end{theorem}
\begin{proof}
The proof follows closely that of Theorem 3.3 in \cite{bel:98} and Theorem 5.2 in \cite{eis:94}. Algorithm \ref{alg:algNPD} does not break down (unless at the optimal solution) and has step sizes bounded away from zero, see Theorem \ref{notZero}, Lemma \ref{breakdown} and the complementary discussion. Furthermore, it follows from Theorem \ref{conLimitPointTheorem} that $z^{(l)}\rightarrow z_*$. 

The sequence $\{\|H(z^{(l)})\|\}$ is, by construction of the algorithm, decreasing and bounded, and consequently convergent. Assume that it converges to $\kappa >0$. So, for sufficiently large $l$ we have $z^{i,(l)}\in N_{\delta^i}(z_*)$ and the while loop of Algorithm \ref{alg:algNPD} terminates with inequality \eqref{terminate} fulfilled. Since $z^{(l)}\rightarrow z_*$, all but finitely many $l$ satisfy that $z^{i,(l)}\in N_{\delta^i}(z_*)$,  consequently $\sum_{l\geq 0}(1-\eta^{(l)}_*)$ is divergent. Hence, from Theorem \ref{divergent} we get that $H(z^{(l)}) \rightarrow 0$, which contradicts our assumption that $H(z^{(l)}) \rightarrow \kappa >0$. Thus, $\{\|H(z^{(l)})\|\}$ converges to zero.
\end{proof}

\section{ADMM, Fixed Point Iterations and Uzawa's Method}\label{sec:Appendix2}
As stated in Section \ref{sec:UMAFPI}, ADMM can be viewed as fixed point iterations and as a modified version of Uzawa's method. In this appendix, we explore these relations in detail for our specific problem formulation.
\subsection{ADMM and Fixed Point Iterations}
Consider functions $F_1(\Delta W)$ and $F_2(\Delta x)$ in \eqref{eq:ADMM}, and let us rewrite them in the form \[F_1(\Delta W)= \frac{1}{2}\Delta W^T\tilde{F}_1\Delta W+\tilde{f}_1^T\Delta W \text{ and }  F_2(\Delta x)=\tilde{f}_2^T\Delta x.\]
We can then rewrite the optimality conditions for \eqref{eq:ADMM}, in \eqref{sysOfEq}, as
\begin{equation}\label{kkt}
\underbrace{\begin{bmatrix} \tilde{F}_1 & 0 & \rho A^T  \\ 0& 0& \rho B^T\\\rho A & \rho B & 0 \end{bmatrix}}_{A_{\text{KKT}}}
\begin{bmatrix} \Delta W  \\ \Delta x \\ \Delta \bar{u}\end{bmatrix} =
\underbrace{\begin{bmatrix}-\tilde{f}_1  \\ -\tilde{f}_2\\\rho c\end{bmatrix}}_{b_{\text{KKT}}},
\end{equation}
where $\Delta \bar{u} = (\Delta \bar{v},\Delta \bar{v}_c)$. With the newly defined notation, the ADMM iterations for the problem in \eqref{eq:ADMM} can be written in closed form as
\begin{equation}\label{eq:ADMMIter}
\begin{split}
\Delta W^{(k+1)} &= M_1(-\tilde{f}_1+\rho A^Tc -\rho A^TB\Delta x^{(k)}-\rho A^T \Delta \bar{u}^{(k)}),\\
\Delta x^{(k+1)} &=M_2(-\tilde{f}_2+\!\rho B^Tc\! -\!\rho B^T\!A\Delta W^{(k+1)}\!-\!\rho B^T\! \Delta \bar{u}^{(k)}),\\
\Delta \bar{u}^{(k+1)} &=  \Delta \bar{u}^{(k)} + A \Delta W^{(k+1)} + B \Delta x^{(k+1)} - c,
\end{split}
\end{equation}
with
\begin{equation*}
M_1=(\tilde{F}_1+\rho A^TA)^{-1} \text{ and } M_2=(\rho B^TB)^{-1}.
\end{equation*}
We can rewrite the equations in \eqref{eq:ADMMIter} in a more compact manner as
\begin{equation}\label{eq:ADMMIterSS}
\begin{split}
 \begin{bmatrix} \Delta W^{(k+1)}   \\ \Delta x^{(k+1)}  \\  \Delta \bar{u}^{(k+1)} \end{bmatrix}=G \begin{bmatrix} \Delta W^{(k)}   \\ \Delta x^{(k)} \\  \Delta \bar{u}^{(k)}\end{bmatrix}+f,
\end{split}
\end{equation}
with
\begin{equation*}
G\!=\!
\begin{bmatrix}\begin{array}{c c c}
0 & -\rho M_1 A^TB &-\rho M_1 A^T\\
0 & \rho^2M_2 B^T\!AM_1A^TB& \rho M_2 B^T\!(\rho AM_1A^T\!- I) \\
0&\rho  (\rho B M_2 B^T\!\!-I)AM_1A^T\!B  & \rho (\rho B M_2 B^T\!\!-I)AM_1 A^T\!-\rho B M_2 B^T\!\!+\!I\!\end{array}\end{bmatrix}
\end{equation*}
and
\begin{equation*}
\begin{split}
f=
\begin{bmatrix}M_1m_1\\
M_2m_2-\rho M_2B^TAM_1m_1\\
-c+AM_1m_1+B(M_2m_2-\rho M_2 B^TAM_1m_1)\\ \end{bmatrix}\!,
\end{split}
\end{equation*}
where
\begin{equation*}
m_1=-\tilde{f}_1+\rho A^Tc \text{ and } m_2=-\tilde{f}_2+\rho B^Tc.
\end{equation*}
The iterations in \eqref{eq:ADMMIterSS} clearly show that $\Delta x$ and $\Delta \bar{u}$ make up the state of the algorithm, whereas $\Delta W$ only is an intermediate result, see \cite{boyd:11}. Notice that we can view \eqref{eq:ADMMIterSS} as an iterative solver for a pre-conditioned version of the system of equations in \eqref{kkt}. That is, the iteration matrix $G$ and the vector $f$ can be expressed as
\begin{equation*}
G=I-M_{\text{PRE,1}}^{-1}A_{\text{KKT}}\textrm{ and }f = M_{\text{PRE,1}}^{-1}b_{\text{KKT}},
\end{equation*}
respectively, where $M_{\text{PRE,1}}$ is a pre-conditioner defined as
\begin{equation*}
\begin{split}
M_{\text{PRE,1}}=
\begin{bmatrix} \tilde{F}_1&-\rho A^TB&\rho A^T \\ 0&0&\rho B^T \\\rho A&\rho B &-\rho I\\ \end{bmatrix}.
\end{split}
\end{equation*}
For details regarding iterative solvers and pre-conditioners, see \cite{saa:03}. When ADMM converges, we get the fixed point iterations
\begin{equation*}
\begin{split}
 \begin{bmatrix} \Delta W \\ \Delta x \\\Delta \bar{u} \end{bmatrix}=G \begin{bmatrix} \Delta W  \\\Delta  x \\ \Delta \bar{u}\end{bmatrix}+f \Leftrightarrow
A_{\text{KKT}}\begin{bmatrix} \Delta W  \\ \Delta x \\ \Delta \bar{u}\end{bmatrix}=b_{\text{KKT}},
\end{split}
\end{equation*}
which is equal to the system of equations that we would like to solve, namely \eqref{kkt}.

\subsection{ADMM and Uzawa's Method}\label{admom}

Notice that solving \eqref{kkt} is equivalent to finding a saddle point of the Lagrangian function
\begin{align*}
\begin{split}
\mathcal{L}(\Delta W,\Delta x,\Delta \bar{u}) =  F_1(\Delta W)+F_2(\Delta x)+\rho \Delta \bar u^T(AW+Bx-c),
\end{split}
\end{align*}
which also is a saddle point of the augmented Lagrangian
\begin{align*}
\begin{split}
\mathcal{L}_{\rho}(\Delta W,\Delta x,\Delta \bar{u}) =   F_1(\Delta W)+F_2(\Delta x)+\frac{\rho}{2}\|A\Delta W+B\Delta x-c+\Delta \bar{u}\|_2^2.
\end{split}
\end{align*}
This augmented Lagrangian function is in turn the Lagrangian function of the optimization problem
\begin{equation}\label{eq:ADMMmod}
\begin{split}
\minimize_{\Delta S, \Delta x}  & \quad  F_1 (\Delta W) + F_2(\Delta x)+\frac{\rho}{2}\|A\Delta W+B\Delta x-c\|_2^2, \\
\subject & \quad  A \Delta W + B\Delta x = c,
\end{split}
\end{equation}
which is equivalent to \eqref{eq:ADMM}. ADMM applied to \eqref{eq:ADMM} is equivalent to Uzawa's method applied to \eqref{eq:ADMMmod} with one Gauss-Seidel iteration \citep{saa:03} in the update of the primal variables \citep{mar:75,gab:76} and the relaxation parameter equal to the penalty parameter $\rho$. This can be seen by first noting that the optimality conditions of \eqref{eq:ADMMmod}, i.e.,
\begin{equation}\label{augkkt}
\underbrace{\begin{bmatrix} M_1^{-1} & \rho A^TB & \rho A^T  \\ \rho B^TA& M_2^{-1}& \rho B^T\\\rho A & \rho B & 0 \end{bmatrix}}_{A_{\text{KKT}}^{\text{PRE,2}}}
\begin{bmatrix} \Delta W  \\ \Delta x \\ \Delta \bar{u}\end{bmatrix} =
\underbrace{\begin{bmatrix}m_1  \\ m_2\\\rho c\end{bmatrix}}_{b_{\text{KKT}}^{\text{PRE,2}}},
\end{equation}
are equivalent to those of Problem \eqref{eq:ADMM}. In fact, the system of equations \eqref{augkkt} is a preconditioned version of \eqref{kkt}, that is, $A_{\text{KKT}}^{\text{PRE,2}}=M_{\text{PRE,2}}^{-1}A_{\text{KKT}}$ and $b_{\text{KKT}}^{\text{PRE,2}}=M_{\text{PRE,2}}^{-1}b_{\text{KKT}}$,
with
\begin{align*}
M_{\text{PRE,2}} = \begin{bmatrix} I & 0 & -A^T  \\ 0& I& -B^T\\0 & 0 & I \end{bmatrix}.
\end{align*}
Uzawa's method minimizes the Lagrangian function by iteratively first minimizing with respect to primal variables and then with respect to dual variables. In ADMM, however, the minimization with respect to primal variables are performed sequentially over $\Delta W$ and $\Delta x$, that is, by performing one Gauss-Seidel iteration with respect to $\Delta W$ and $\Delta x$  \citep{saa:03}. The first step of Uzawa's method applied to our problem then requires solving
\begin{equation}\label{1uzawa}
\begin{bmatrix} M_1^{-1} & \rho A^TB\\ \rho B^TA& M_2^{-1}\end{bmatrix}
\begin{bmatrix} \Delta W  \\ \Delta x\end{bmatrix} =
\begin{bmatrix}m_1-\rho A^T\Delta \bar{u}^k  \\ m_2-\rho B^T\Delta \bar{u}^k\end{bmatrix},
\end{equation}
which is the closed form solution of
\begin{align*}
(\Delta W^{(k+1)},\Delta x^{(k+1)}) =& \argmin_{\Delta W,\Delta x} \left\{L_{\rho}(\Delta W,\Delta x,\Delta \bar{u}^k)  \right\}\\
=&\argmin_{\Delta W,\Delta x} \left\{\!\! F_1(\Delta W)\!\!+\!\!F_2(\Delta x)\!\!+\!\!\frac{\rho}{2}\|A\Delta W\!\!+\!\!B\Delta x\!\!-c\!\!+\!\!\Delta \bar{u}^{k}\|_2^2\right\}\!\!.
\end{align*}
The system of equations in \eqref{1uzawa} can be solved approximately using Gauss-Seidel, that is,
\begin{align}\label{gaussseidel}
\begin{split}
\begin{bmatrix} \Delta W^{k+1}  \\ \Delta x^{k+1}\end{bmatrix} =&
\begin{bmatrix} M_1^{-1} & 0\\ \rho B^T\!\!A& M_2^{-1}\end{bmatrix}^{\!-1}
\begin{bmatrix} 0 & -\rho A^TB\\ 0& 0\end{bmatrix}\!\!
\begin{bmatrix} \Delta W^{k}  \\ \Delta x^{k}\end{bmatrix}\!\!
+\\&
\begin{bmatrix} M_1^{-1} & 0\\ \rho B^T\!\!A& M_2^{-1}\end{bmatrix}^{\!-1}
\begin{bmatrix}m_1-\rho A^T\Delta \bar{u}^k  \\ m_2-\rho B^T\Delta \bar{u}^k\end{bmatrix}\\=&
\begin{bmatrix}0 & -\rho M_1 A^TB  \\
0 & \rho^2M_2 B^T\!\!AM_1 A^TB  \end{bmatrix}\!\!
\begin{bmatrix} \Delta W^{k}  \\ \Delta x^{k}\end{bmatrix}
+\\&
\begin{bmatrix}M_1(m_1-\rho A^T\Delta \bar{u}^k)\\
M_2(m_2\!-\!\rho B^T\!\Delta \bar{u}^k)-\rho M_2B^T\!\!AM_1(m_1\!-\!\rho A^T\!\Delta \bar{u}^k)\end{bmatrix}.
\end{split}
\end{align}
Notice that the iteration in \eqref{gaussseidel} is equivalent to the primal updates \eqref{eq:ADMMIterSS} in ADMM. The fact that the dual update in Uzawa's method is equivalent to that of ADMM then shows the equivalence. \\

\end{document}